\documentclass[envcountsect, envcountsame]{svjour3}    
\smartqed  

\usepackage[fontsize=13pt]{scrextend}

\usepackage{cite}

\usepackage{hyperref}

\newcommand\rurl[1]{%
  \href{http://#1}{\nolinkurl{#1}}%
}

\AtBeginDocument{\hypersetup{citecolor=blue, urlcolor=blue, linkcolor=blue}}
\usepackage[T5]{fontenc}

\usepackage{graphicx} 

\usepackage{color} 

\usepackage{pb-diagram}

\usepackage{amsmath}
\usepackage{bm}
\usepackage{amssymb}
\usepackage{mathrsfs}
\usepackage{latexsym}
\usepackage{exscale}
\usepackage{eucal}

\usepackage[top=2cm, bottom=2cm, left=2.5cm, right=2.5cm] {geometry}

\usepackage{listings}

\def\DD{D\kern-.7em\raise0.4ex\hbox{\char '55}\kern.33em}

\def\subclassname{{\bfseries Mathematics Subject Classification
(2010)}\enspace}
\def\subclass#1{\par\addvspace\medskipamount{\rightskip=0pt plus1cm
\def\and{\ifhmode\unskip\nobreak\fi\ $\cdot$
}\noindent\subclassname\ignorespaces#1\par}}

\spnewtheorem{dn}{Definition}[section]{\bfseries}{\upshape}
\spnewtheorem{cy}[dn]{Note}{\bfseries}{\upshape}
\spnewtheorem{nota}[dn]{Notation}{\bfseries}{\upshape}
\spnewtheorem{rem}[dn]{Remark}{\bfseries}{\upshape}

\spnewtheorem{thm}[dn]{Theorem}{\bfseries}{\itshape}
\spnewtheorem{prop}[dn]{Proposition}{\bfseries}{\itshape}
\spnewtheorem{lem}[dn]{Lemma}{\bfseries}{\itshape}
\spnewtheorem{hq}[dn]{Corollary}{\bfseries}{\itshape}
\spnewtheorem{nots}[dn]{Note}{\bfseries}{\upshape}
\spnewtheorem{gt}[dn]{Conjecture}{\bfseries}{\itshape}

\spnewtheorem*{acknow}{Acknowledgments}{\bfseries}{\upshape}

\spnewtheorem*{xproof}{}{\itshape}{\rmfamily}

\renewenvironment{proof}[1][\proofname]
 {\renewcommand\xproofname{#1}\xproof}
 {\endxproof}

\def\DD{D\kern-.7em\raise0.4ex\hbox{\char '55}\kern.33em}

\usepackage{scalerel}
\newcommand\reallywidehat[1]{\arraycolsep=0pt\relax%
\begin{array}{c}
\stretchto{
  \scaleto{
    \scalerel*[\widthof{\ensuremath{#1}}]{\kern-.5pt\bigwedge\kern-.5pt}
    {\rule[-\textheight/2]{1ex}{\textheight}} 
  }{\textheight} %
}{0.5ex}\\           
#1\\                 
\rule{-1ex}{0ex}
\end{array}
}

\usepackage{hhline}

\begin{document}

\title{On the mod-2 cohomology of the product of the infinite lens space and the space of invariants in a generic degree}


\author{\DD\d{\u a}ng V\~o Ph\'uc
}

\authorrunning{\DD\d.V.Ph\'uc}

\institute{\textbf{\DD\d{\u a}ng V\~o Ph\'uc} \at
        Department of Information Technology, FPT University, Quy Nhon A.I Campus,\\ 
An Phu Thinh New Urban Area, Quy Nhon City, Binh Dinh, Vietnam\\
              \email{\textbf{dangphuc150488@gmail.com; phucdv14@fpt.edu.vn}}  \\
\textbf{ORCID:} \url{https://orcid.org/0000-0002-6885-3996}
}


\maketitle

\begin{abstract}

Let $\mathbb S^{\infty}/\mathbb Z_2$ be the infinite lens space.  Denote the Steenrod algebra over the prime field $\mathbb F_2$ by $\mathscr A.$ The (mod-2) cohomology $H^{*}((\mathbb S^{\infty}/\mathbb Z_2)^{\oplus s}; \mathbb F_2)$ is known to be isomorphic to the graded polynomial ring $\mathcal {P}_s:= \mathbb F_2[x_1, \ldots, x_s]$ on $s$ generators of degree 1, viewed as an unstable $\mathscr A$-module. The Kameko squaring operation $(\widetilde {Sq^0_*})_{(s; N)}: (\mathbb F_2\otimes_{\mathscr A} \mathcal {P}_s)_{2N+s}  \longrightarrow (\mathbb F_2\otimes_{\mathscr A} \mathcal {P}_s)_{N}$ is rather useful in studying an open problem of determining the dimension of the indecomposables $(\mathbb F_2\otimes_{\mathscr A} \mathcal {P}_s)_N.$ It has been demonstrated that this $(\widetilde {Sq^0_*})_{(s; N)}$ is onto. 

As a continuation of our recent works, this paper deals with the kernel of the Kameko $(\widetilde {Sq^0_*})_{(s; N_d)}$ for the case where $s = 5$ and the "generic" degree $N_d$ is of the form $N_d = 5(2^{d} - 1) + 11.2^{d+1}$ for arbitrary $d > 0.$ We then rectify almost all of the main results that were inaccurate in an earlier publication [Rev. Real Acad. Cienc. Exactas Fis. Nat. Ser. A-Mat. \textbf{116}:81 (2022)] by Nguyen Khac Tin. We have also constructed several advanced algorithms in \texttt{SAGEMATH} to validate our results. \textit{These new algorithms make an important contribution to tackling the intricate task of explicitly determining both the dimension and the basis for the indecomposables $\mathbb F_2 \otimes_{\mathscr A} \mathcal {P}_s$ at positive degrees, a problem concerning algorithmic approaches that had not previously been addressed by any author.} Also, the present study encompasses an investigation of the behavior of the cohomological transfer in bidegrees $(5, 5+N_d)$, with the internal degree $N_d$ mentioned above. 

 \subclass{13A50\and 55S10\and 55R12 \and 55S05}
\end{abstract}

\begin{acknow}

The manuscript was submitted to an ISI-indexed journal for peer review and consideration for publication.
\end{acknow} 

\section{Introduction}\label{s1}

It is well-known that the Steenrod algebra $\mathscr A$ acts on mod-2 cohomology, making it indispensable computational tools in homotopy theory at the prime 2. As an algebra over the prime field $\mathbb F_2,$ $\mathscr A$ is generated by the Steenrod squares $Sq^{i}$ for $i\geq 0$ which act by post-composition. Let $\mathbb S^{\infty}/\mathbb Z_2$ be an infinite lens space. The action of $\mathscr A$ on the cohomology algebra $H^{*}((\mathbb S^{\infty}/\mathbb Z_2)^{\oplus s}; \mathbb F_2)$ has been a topic of much 
(see, for instance, Walker and Wood \cite{W.W}, Wood \cite{R.W}). This cohomology $H^{*}((\mathbb S^{\infty}/\mathbb Z_2)^{\oplus s}; \mathbb F_2)$ could be identified with an unstable $\mathscr A$-module $\mathcal {P}_s:=\mathbb F_2[x_1, \ldots, x_s]$ (the graded polynomial ring on $s$ generators, each of degree one). To understand this event better, readers can find more details in our recent work \cite{D.P13}. A renowned challenge in algebraic topology involves determining a minimal generating set for the invariant ring $\mathcal {P}_s^{G_s},$ where $G_s$ is a subgroup of $GL(s, \mathbb F_2).$ This is tantamount to computing the dimension of the $\mathbb Z$-graded vector space $\{(\mathbb F_2\otimes_{\mathscr A} \mathcal {P}_s^{G_s})_N\}_{N\geq 0}.$ The structure of this space has been thoroughly studied by many authors (see Janfada-Wood \cite{J.W}, Pengelley-Williams \cite{P.W}, Peterson \cite{F.P2}, Singer \cite{W.S1, W.S2}, and others). In addition,  its applications to homotopy theory are surveyed by Peterson \cite{F.P2} and Singer \cite{W.S1}. For $G_s = \Sigma_s$ (the symmetric group of rank $s$), the problem was investigated by Janfada-Wood  \cite{J.W} for $s = 3,$ and by Singer \cite{W.S2} for the dual of $\{(\mathbb F_2\otimes_{\mathscr A} \mathcal {P}_s^{G_s})_N\}_{N\geq 0}$ for $s\geq 0.$ The problem when $G_s = GL(s, \mathbb F_2)$ was studied by Singer \cite{W.S1} for $s = 2.$ and by H\uhorn ng-Peterson \cite{H.P} for $s = 3,\, 4.$ In our current investigation, we are primarily concerned with the case $G_s = \{0\}.$  Accordingly, the space $\{(\mathbb F_2\otimes_{\mathscr A} \mathcal {P}_s^{G_s})_N\}_{N\geq 0}$ simplifies to $\{(\mathbb F_2\otimes_{\mathscr A} \mathcal {P}_s)_N = (\mathcal {P}_s/\mathscr A^{+}\mathcal {P}_s)_N\}_{N\geq 0},$ with $\mathscr A^{+}$ representing the augmentation ideal in $\mathscr A.$ In what follows, we will use the notation $\mathscr A\mathcal Q^{\otimes}_s$ to refer to $\mathbb F_2\otimes_{\mathscr A} \mathcal {P}_s = \mathcal {P}_s/\mathscr A^{+}\mathcal {P}_s.$  Explicitly determining the dimension and an "admissible monomial" basis for $\mathscr A\mathcal Q^{\otimes}_s$ in positive degrees is known as the Peterson "hit" problem \cite{F.P1}. The geometric significance of this hit problem's solution lies in its ability to describe how cells in a $CW$-complex at the prime 2 are attached to cells of lower dimensions. Full investigations into the dimension of the space $\mathscr A\mathcal Q^{\otimes}_s$ have been carried out for values of $s$ up to 4 (see \cite{F.P1, M.K, N.S1, W.W, R.W}). Some results for higher $s$ and various "generic" degrees have been contributed by Phúc \cite{P.S1, D.P2, D.P4, D.P7, D.P7-1, D.P8, D.P11} and others. A general solution to the problem, however, has yet to be found.

A key instrument in our investigation of $\mathscr A\mathcal Q^{\otimes}_s,$ is the Kameko squaring operation \cite{M.K}. This operation, denoted as $(\widetilde {Sq^0_*})_{(s; N)},$ is an epimorphism of $\mathbb F_2GL(s, \mathbb F_2)$-modules from $(\mathscr A\mathcal Q^{\otimes}_s)_{2N+s}$ onto $(\mathscr A\mathcal Q^{\otimes}_s)_{N}.$ Under the Kameko $(\widetilde {Sq^0_*})_{(s; N)},$ any class $[u]$ maps to the class $[\zeta]$ if $u$ is representable as the product of $s$ generators of $\mathcal P_s$ with $\zeta$ squared, and to zero in all other instances.
Kameko proved in \cite{M.K} that if $\mu(2N+s) = s,$ then $(\widetilde {Sq^0_*})_{(s; N)}$ is an isomorphism. Here $\mu(N)$ is the smallest number of integers of the form $2^{j}-1$ (with repetitions allowed) whose sum is $N.$ Moreover, according to Wood \cite{R.W}, $(\mathscr A\mathcal Q^{\otimes}_s)_{N}$ is trivial,  if $\mu(N) > s.$ In light of these results, we restrict the study of $\mathscr A\mathcal Q^{\otimes}_s$ to those degrees $N$ where $\mu(N)$ is upper-bounded by $s.$

Continuing the line of inquiry established in our earlier paper \cite{D.P4}, the present work focuses on the $s = 5$ case of the Kameko $(\widetilde {Sq^0_*})_{(s; N)},$ investigating its kernel structure for degrees of the form $N_d := 5(2^d - 1) + 11 \cdot 2^{d+1},$ where $d$ is a positive integer. We can easily observe that $\mu(N_1)  = 3 < 5$ and $\mu(N_d)  = 5$ for any $d > 1.$ This subject has also been investigated by Nguyen Khac Tin \cite{Tin2}. This paper \cite{Tin2} essentially builds upon our previous works in \cite{D.P4, D.P7}. Unfortunately, almost all principal outcomes in \cite{Tin2} are wholly inaccurate, including \textbf{Theorem 3.3}, \textbf{Corollary 3.4}, \textbf{Theorem 3.6}, and \textbf{Theorem 3.7}. It is worth noting that \textbf{Theorem 3.3(ii)}, the most significant yet incorrect result in Tin's paper \cite{Tin2}, is presented without a detailed proof. 

Therefore, the focus of our paper is to tackle and rectify all the aforementioned inaccurate results found in Nguyen Khac Tin's paper \cite{Tin2}. Several advanced algorithms in \texttt{SAGEMATH} are also constructed to confirm our results. With the limited tools we have at the present time, these algorithms become very important and essential. Indeed, the use of these algorithms makes it easier to simplify the process of explicitly computing the dimension and basis of $\mathscr A\mathcal Q^{\otimes}_s$ at certain positive degrees. 

\begin{cy}
In the current state of knowledge, we have not found any published work that presents an algorithm implemented in computer algebra systems which explicitly solves for both the dimension and basis of $(\mathscr A\mathcal Q^{\otimes}_s)_N$, even for small $N.$ Therefore, our current work has made a significant contribution to resolving this issue, as it is well-known that explicitly calculating a basis for $(\mathscr A\mathcal Q^{\otimes}_s)_N$ is an extremely challenging task, if not impossible. With the attainment of this preliminary algorithmic result, we hope that it will pave the way for further research directions aimed at developing algorithms capable of fully describing the basis of the indecomposables $\mathscr A\mathcal Q^{\otimes}_s.$
\end{cy}

We, additionally, utilize our previous works to study the behavior of the fifth cohomological transfer, defined by W. Singer \cite{W.S1}, particularly in relation to degrees $N_d.$  

The computational methods employed in this paper build upon techniques developed in our previous work \cite{D.P4}, with the aid of the algebraic computer system \texttt{SAGEMATH}.

\medskip

\textbf{Outline of subsequent sections.} We begin with essential background information in Section \ref{s2}. Subsequently, Section \ref{s3} studies the kernel of the Kameko homomorphism in the generic degree $N_d = 5(2^d - 1) + 11 \cdot 2^{d+1}$, and corrects almost all the wrong chief results in Nguyen Khac Tin's paper \cite{Tin2}. Section \ref{s4} is devoted to investigating the isomorphism of the Singer cohomological transfer in bi-degrees $(5, 5+N_d)$ for $d\geq 0.$ 

 


\section{Several notions and related known outcomes}\label{s2}

For the purposes of this study, we assume that the reader has a fundamental understanding of the subjects discussed. To effectively present the main results of our paper, we will provide a brief review of some essential known results. For further details, readers are encouraged to refer to our previous work \cite{D.P4}.

\begin{dn}
Let the dyadic expansion of a positive integer $N$ be given by $N = \sum_{j\geq 0}\alpha_j(n)2^j,$ where $\alpha_j(N)= 0,\, 1.$ For each monomial $x = x_1^{a_1}x_2^{a_2}\ldots x_s^{a_s}$ in $\mathcal {P}_s,$ we define its \textit{weight vector} $\omega(x)$ as the sequence: $\omega(x) =(\omega_1(x), \omega_2(x), \ldots, \omega_j(x), \ldots),$ where $\omega_j(x) = \sum_{1\leq i\leq s}\alpha_{j-1}(a_i)$ for all $j \geq 1.$
\end{dn}

Consider a weight vector $\omega = (\omega_1, \omega_2, \ldots, \omega_k,0,0,\ldots)$ with its degree defined as $\deg \omega =\sum_{k\geq 1}2^{k-1}\omega_k.$ We denote by:

$\bullet$ $\mathcal {P}_s(\omega)$ the subspace of $\mathcal {P}_s$ spanned by monomials $u$ satisfying both $\deg u = \deg \omega$ and $\omega(u)\leq \omega;$

$\bullet$ $\mathcal {P}_s^-(\omega)$ the subspace of $\mathcal {P}_s(\omega)$ spanned by monomials $u$ with $\omega(u)< \omega.$

\begin{dn}
For $f$ and $g,$ homogeneous polynomials of equal degree in $\mathcal {P}_s,$ we define:

(i) $f \sim g\iff (f + g) \in \mathscr{A}^+\mathcal {P}_s;$ 

(ii) $f \sim_{\omega} g\iff (f + g)\in \mathscr{A}^+\mathcal {P}_s + \mathcal {P}_s^-(\omega).$
\end{dn}

Both "$\sim$" and "$\sim_{\omega}$" possess the characteristics of equivalence relations, as is readily observable. Let $(\mathcal Q\mathcal {P}_s)(\omega)$ denote the quotient of $\mathcal {P}_s(\omega)$ by the equivalence relation "$\sim_\omega$". References \cite{W.W} show that $(\mathcal Q\mathcal {P}_s)(\omega)$ possesses a $GL(s, \mathbb F_2)$-module structure, and for every degree $n,$ there is an isomorphism given by $(\mathscr A\mathcal Q^{\otimes}_s)_N \cong \bigoplus\limits_{{\rm deg}\,\omega = N}(\mathcal Q\mathcal {P}_s)(\omega).$

\begin{dn}
One defines an ordering on monomials in $\mathcal {P}_s$ as follows: For $u = x_1^{a_1}\ldots x_s^{a_s}$ and $v=x_1^{b_1}\ldots x_s^{b_s},$ $u < v$ if and only if either $\omega(u) < \omega(v),$ or $\omega(u) = \omega(v)$ and $(a_1, a_2, \ldots, a_s)$ is lexicographically less than $(b_1, b_2, \ldots, b_s).$
\end{dn}

 \begin{dn}
In $\mathcal {P}_s,$ we call a monomial $x$ inadmissible if there exists a collection of monomials $y_1, y_2,\ldots, y_m,$ each less than $x,$ such that their sum with $x$ lies in $\mathscr{A}^+\mathcal {P}_s.$ In other words, $x\sim \sum_{1\leq j\leq m}{y_j}.$  Conversely, a monomial $x$ is defined to be admissible if it is not inadmissible.
\end{dn}

 \begin{dn}
In $\mathcal {P}_s,$ a monomial $x$ is defined as strictly inadmissible if and only if there exist monomials $y_1, y_2,\ldots, y_m$ of the same degree as $x,$ each satisfying $y_j < x$ for $1\leq j\leq m,$ such that $x$ can be expressed as $x = \sum_{1\leq j\leq m}y_t + \sum_{1\leq \ell < 2^r }Sq^{\ell}(h_j),$ where $r = \max\{i\in\mathbb Z: \omega_i(x) > 0\}$ and $h_j$ are appropriate polynomials in $\mathcal {P}_s.$
\end{dn}

\begin{thm}[see \cite{M.K}, \cite{N.S0}]\label{dlKS}

Let $u,\, v,\, w$ be monomials in $\mathcal {P}_s$ such that $\omega_i(u) = 0$ for $i > r > 0,\, \omega_k(w)\neq 0$ and $\omega_i(w) = 0$ for $i > k > 0.$ 
\begin{itemize}
\item[(i)] If $w$ is inadmissible, then so is $uw^{2^r}.$ 

\item[(ii)] If $w$ is strictly inadmissible, then so is $wv^{2^{k}}.$
\end{itemize}
\end{thm} 

Consequently, the set of all the admissible monomials of degree $N$ in $\mathcal {P}_s$ is a minimal set of $\mathscr A$-generators for $\mathcal {P}_s$ in degree $N.$ So, the space $(\mathscr A\mathcal Q^{\otimes}_s)_N$ has a basis consisting of all the classes represent by the admissible monomials of degree $N$ in $\mathcal {P}_s.$

Let us characterize two subspaces of $\mathcal {P}_s$:
$$\begin{array}{ll}
\medskip
\mathcal {P}_s^0 &={\rm span} \big\{x= x_1^{a_1}\ldots x_s^{a_s}\in \mathcal {P}_s\;|\;a_1\ldots a_s = 0\big\}, \\
\mathcal {P}_s^{+}&= {\rm span} \big\{x= x_1^{a_1}\ldots x_s^{a_s}\in \mathcal {P}_s\;|\; a_1\ldots a_s > 0\big\}.
\end{array}$$
These subspaces are $\mathscr{A}$-submodules of $\mathcal {P}_s.$ This leads to a direct sum decomposition of $\mathscr A\mathcal Q^{\otimes}_s$ as $\mathbb F_2$-vector spaces:
$$\mathscr A\mathcal Q^{\otimes}_s = \mathscr A\mathcal Q^{\otimes 0}_s \bigoplus \mathscr A\mathcal Q^{\otimes >}_s,$$
where $\mathscr A\mathcal Q^{\otimes 0}_s :=  \mathbb F_2\otimes_{\mathscr A}\mathcal {P}_s^{0}$ and $\mathscr A\mathcal Q^{\otimes >}_s :=  \mathbb F_2\otimes_{\mathscr A}\mathcal {P}_s^{+}.$

A spike in $\mathcal {P}_s$ is characterized as a monomial $z = x_1^{b_1}\ldots x_s^{b_s}$ where each exponent $b_i$ is of the form $2^{d_i} - 1,$ with $d_i$ being a non-negative integer. We call this spike minimal if the sequence of $d_i$ is strictly decreasing for the first $r-1$ terms, $d_r$ is positive, and all subsequent $d_j$ (for $j>r$) vanish.

\begin{lem}[see \cite{P.S1}]\label{dlPS}
All the spikes in $\mathcal {P}_s$ are admissible and their weight vectors are weakly decreasing. Furthermore, if a weight vector $\omega = (\omega_1, \omega_2, \ldots)$ is weakly decreasing and $\omega_1\leq s,$ then there is a spike $z$ in $\mathcal {P}_s$ such that $\omega(z) = \omega.$
\end{lem}

The following statement provides a means to recognize $\mathscr A$-decomposable monomials.

\begin{thm}[see {\cite{W.S3}}]\label{dlSi}
Assume $u$ is a monomial in $\mathcal {P}_s$ with degree $n,$ where $\mu(n)\leq s.$ If $z$ denotes the minimal spike of degree $n$ in $\mathcal {P}_s,$ then $u$ is $\mathscr A$-decomposable (that is, $u\sim 0$) when $\omega(u) < \omega(z).$
\end{thm}

From here on, $[u]$ and $[u]_{\omega}$ will denote the respective classes in $\mathscr A\mathcal Q^{\otimes}_s$ and $\mathcal Q\mathcal P_s(\omega)$, represented by the homogeneous polynomial $u$ in $\mathcal {P}_s.$ We define $\mathscr{A}$-homomorphisms $\rho_j: \mathcal {P}_s\longrightarrow \mathcal {P}_s$ for each $j$ from 1 to $s$: For $j=1,2, \ldots, s-1$: $\rho_j$ swaps $t_j$ and $t_{j+1},$ leaving other $t_i$ unchanged, $i\neq j,\, j+1.$ For $j = s$: $\rho_j$ replaces $t_1$ with $t_1 + t_2,$ keeping all other $t_i$ unchanged. Consequently, the matrices corresponding to $\rho_j,\; 1\leq j\leq s,$ generate the general linear group $GL(s, \mathbb F_2),$ while those associated with $\rho_j,\; 1\leq j \leq s-1,$ generate the symmetric group $\Sigma_s\subset GL(s, \mathbb F_2).$ As a result, a class $[u]_{\omega}\in (\mathcal Q\mathcal P_s)(\omega)$ is invariant under $GL(s, \mathbb F_2)$ if and only if $\rho_j(u) + u\sim_{\omega} 0$ for all $j$ from 1 to $s.$ The same class is $\Sigma_s$-invariant precisely when this condition is satisfied for $j$ from 1 to $s-1.$ Note that if $\omega$ is a weight vector of the minimal spike, then $[u]_{\omega} = [u].$

\section{Kernel of the Kameko $\pmb{(\widetilde {Sq^0_*})_{(5; N_d)}}$ in the generic degree $\pmb{N_d =5(2^{d}-1) + 11.2^{d+1},\ d\geq 1}$}\label{s3}

This section, as discussed in the introduction, aims to investigate the structure of the kernel of the Kameko $(\widetilde {Sq^0_*})_{(5; N_d)},$ in which $N_d = 5(2^d - 1) + 11 \cdot 2^{d+1},$ with $d$ being a positive integer. Based on this, we will correct the incorrect principal outcomes in Nguyen Khac Tin's paper \cite{Tin2}.

At the outset, we present a fundamental observation regarding the domain of the Kameko $(\widetilde {Sq^0_*})_{(5; N_d)}.$ 
 This establishes a clear connection to the kernel of $(\widetilde {Sq^0_*})_{(5; N_d)},$ and motivates our detailed investigation into the structure of this kernel.

\begin{rem}\label{nx}
Noticing that $N_d = 2^{d+4} + 2^{d+3} + 2^{d+1} + 2^{d} - 5;$ hence, $\mu(N_1) = 3 < 5,$ and $\mu(N_d) = 5,$ for arbitrary $d > 1.$ Thus, it suffices to determine the space $(\mathscr A\mathcal Q^{\otimes}_5)_{N_d} $ for the cases $d=0$ and $d = 1$ because the iterated homomorphism $((\widetilde {Sq^0_*})_{(5; N_d)})^{d-1}: (\mathscr A\mathcal Q^{\otimes}_5)_{N_d}  \longrightarrow(\mathscr A\mathcal Q^{\otimes}_5)_{N_1}$ is an isomorphism, for every positive integer $d.$ When $d = 0,$ we have $\dim (\mathscr A\mathcal Q^{\otimes}_5)_{N_0}  = 965$ (see our previous work \cite{D.P4}). Consequently, the epimorphic nature of the Kameko squaring operation $(\widetilde {Sq^0_*})_{(5; N_1)}: (\mathscr A\mathcal Q^{\otimes}_5)_{N_1}  \longrightarrow(\mathscr A\mathcal Q^{\otimes}_5)_{N_0}$ allows us to determine $(\mathscr A\mathcal Q^{\otimes}_5)_{N_d}$ for $d = 1$ solely by computing the kernel of $(\widetilde {Sq^0})_{(5; N_1)}.$ (Indeed, one has an isomorphism: $(\mathscr A\mathcal Q^{\otimes}_5)_{N_1}\cong {\rm Ker}((\widetilde {Sq^0_*})_{(5; N_1)}) \bigoplus (\mathscr A\mathcal Q^{\otimes}_5)_{N_0}.$)
\end{rem}

One of the paper's principal findings is presented below, with the observation that $N_1 = 5(2^1 - 1) + 11 \cdot 2^{1+1} = 49.$

\begin{thm}\label{dlc1}
Take a look at the weight vectors for the degree $N_1$:
 $$ 
 \omega_{(1)}:=(3,3,2,2,1),\ \ \ \ \omega_{(2)}:=(3,3,4,1,1),\ \ \ \ \omega_{(3)}:=(3,3,4,3).
$$

\begin{itemize}
\item[(i)] We have an isomorphism: 
$${\rm Ker}((\widetilde {Sq^0_*})_{(5; N_1)}) \cong  (\mathcal Q\mathcal {P}_5)(\omega_{(1)})\bigoplus (\mathcal Q\mathcal {P}_5)(\omega_{(2)})\bigoplus (\mathcal Q\mathcal {P}_5)(\omega_{(3)}).$$ 

\item[(ii)] The subspaces $(\mathcal Q\mathcal {P}_5)(\omega_{(2)})$ and $(\mathcal Q\mathcal {P}_5)(\omega_{(3)})$ are trivial. 

\item[(iii)] The subspace $(\mathcal Q\mathcal {P}_5)(\omega_{(1)})$ is 1891-dimensional.
\end{itemize}
\end{thm}

\begin{rem}\label{nx2}
As is well known, $${\rm Ker}((\widetilde {Sq^0_*})_{(5; N_1)}) \cong (\mathscr A\mathcal Q^{\otimes 0}_5)_{N_1} \bigoplus ({\rm Ker}((\widetilde {Sq^0_*})_{(5; N_1)})\cap (\mathscr A\mathcal Q^{\otimes >}_5)_{N_1}).$$ On the other hand, we have the isomorphisms:
$$\begin{array}{ll}
\medskip
 (\mathscr A\mathcal Q^{\otimes 0}_5)_{N_1}\cong (\mathcal Q\mathcal {P}^{0}_5)(\omega_{(1)})\bigoplus (\mathcal Q\mathcal {P}^{0}_5)(\omega_{(2)})\bigoplus (\mathcal Q\mathcal {P}^{0}_5)(\omega_{(3)}),\\
{\rm Ker}((\widetilde {Sq^0_*})_{(5; N_1)})\cap (\mathscr A\mathcal Q^{\otimes >}_5)_{N_1}\cong (\mathcal Q\mathcal {P}^{+}_5)(\omega_{(1)})\bigoplus (\mathcal Q\mathcal {P}^{+}_5)(\omega_{(2)})\bigoplus (\mathcal Q\mathcal {P}^{+}_5)(\omega_{(3)}). 
\end{array}$$
Therefore, considering $(\mathcal Q\mathcal {P}_5)(\omega_{(j)})\cong (\mathcal Q\mathcal {P}^{0}_5)(\omega_{(j)})\bigoplus (\mathcal Q\mathcal {P}^{+}_5)(\omega_{(j)}),$ for each $j,$ and in light of Theorem \ref{dlc1}(ii), we can conclude:
$$ {\rm Ker}((\widetilde {Sq^0_*})_{(5; N_1)})\cap (\mathscr A\mathcal Q^{\otimes >}_5)_{N_1}\cong (\mathcal Q\mathcal {P}_5^{+})(\omega_{(1)}).$$
Using a result in Walker and Wood \cite[Proposition 6.2.9]{W.W}, one has 
$$ \dim (\mathcal Q\mathcal {P}_5^{0})(\omega_{(1)}) = \sum_{3 = \mu(N_1)\leq k\leq 4}\binom{5}{k}\dim (\mathcal Q\mathcal {P}_k^{+})(\omega_{(1)}).$$
According to Kameko \cite{M.K} and Sum \cite{N.S1}, we have 
$$ \begin{array}{ll}
\medskip
\dim (\mathcal Q\mathcal {P}_3^{+})(\omega_{(1)}) = \dim (\mathcal Q\mathcal {P}_3^{+})_{N_1} = 14,\\
\dim (\mathcal Q\mathcal {P}_4^{+})(\omega_{(1)}) = \dim (\mathcal Q\mathcal {P}_4^{+})_{N_1} = 154.
\end{array}$$
Combining these data and Theorem \ref{dlc1}(ii), we obtain
$$ \dim (\mathcal Q\mathcal {P}_5^{0})(\omega_{(1)}) = \dim (\mathscr A\mathcal Q^{\otimes 0}_5)_{N_1} = 910.$$
The synthesis of Theorem \ref{dlc1}, Remark \ref{nx2}, and the isomorphic relations:
$$ \begin{array}{ll}
\medskip
(\mathcal Q\mathcal {P}_5)(\omega_{(1)})&\cong (\mathcal Q\mathcal {P}^{0}_5)(\omega_{(1)})\bigoplus (\mathcal Q\mathcal {P}^{+}_5)(\omega_{(1)}),\\
\medskip
(\mathscr A\mathcal Q^{\otimes}_5)_{N_d}&\cong {\rm Ker}((\widetilde {Sq^0_*})_{(5; N_1)})\bigoplus (\mathscr A\mathcal Q^{\otimes}_5)_{N_0},\ \mbox{for all $d\geq 1,$}
\end{array}$$ 
leads directly to the following corollary.
\end{rem}

\begin{hq}\label{hqtq}
The following assertions hold:

\begin{itemize}
\item[(I)] $\dim {\rm Ker}((\widetilde {Sq^0_*})_{(5; N_1)}) =\dim  (\mathcal Q\mathcal {P}_5)(\omega_{(1)}) = 1891.$ Consequently, $$\dim ({\rm Ker}((\widetilde {Sq^0_*})_{(5; N_1)})\cap (\mathscr A\mathcal Q^{\otimes >}_5)_{N_1}) = \dim (\mathcal Q\mathcal {P}_5^{+})(\omega_{(1)}) = 1891 - 910 = 981.$$ 
\item[(II)] For each positive integer $d,$ we have $$\dim (\mathscr A\mathcal Q^{\otimes}_5)_{N_d} = 965 + \dim (\mathcal Q\mathcal {P}_5)(\omega_{(1)}) = 965 + 1891 = 2856.$$
\end{itemize}
\end{hq}

It is readily apparent that $\mu(N_1) = \alpha(N_1 + \mu(N_1)) = 3.$ Hence, by combining Corollary \ref{hqtq}(II) with the result in Sum \cite[Theorem 1.3]{N.S1}, we immediately obtain:

\begin{hq}\label{hqtq2}
For each positive integer $d,$ 
$$\dim (\mathscr A\mathcal Q^{\otimes}_6)_{5(2^{d+4}-1) + N_1.2^{d+4}}  = (2^{6}-1).\dim (\mathscr A\mathcal Q^{\otimes}_5)_{N_d} =  179928.$$
\end{hq}

\begin{rem}\label{nxtq}
In \cite{Tin2}, Nguyen Khac Tin presented four main results: \textbf{Theorem 3.2, Theorem 3.3, Theorem 3.6}, and \textbf{Theorem 3.7}. It is regrettable that, excluding \textbf{Theorem 3.2}, all other results are categorically false. We would also like to further emphasize that \textbf{Theorem 3.2} is merely a direct consequence of our previous work \cite{D.P4} on the dimension of $(\mathscr A\mathcal Q^{\otimes}_5)_{N_0 = 22},$ combined with our Theorem \ref{dlc1}(ii) and the known epimorphism of the Kameko homomorphism $(\widetilde {Sq^0_*})_{(5; N_1)}: (\mathscr A\mathcal Q^{\otimes}_5)_{N_1}  \longrightarrow(\mathscr A\mathcal Q^{\otimes}_5)_{N_0}$ (see Remarks \ref{nx} and \ref{nx2}). As a result, the validity of \textbf{Corollary 3.4} in Tin's paper \cite{Tin2} is also compromised (in fact, this corollary is also wrong), owing to its basis in the inaccuracy of \textbf{Theorem 3.3} in his paper as mentioned above.

\medskip

It is important to note that \textbf{Theorem 3.3(ii)}, the most significant key but incorrect result of Tin's paper \cite{Tin2}, surprisingly lacks a detailed proof. Moreover, \textbf{Corollary 3.4}, \textbf{Theorem 3.6}, and \textbf{Theorem 3.7} in Tin's paper are merely direct consequences of \textbf{Theorem 3.3}. Evidence for this claim can be found in our aforementioned Corollaries \ref{hqtq} and \ref{hqtq2}. To provide a more compelling refutation of the results in Tin's paper, we have developed several new algorithms in \texttt{SAGEMATH} to verify and reinforce our derived results (see the proof of Theorem \ref{dlc1} and Remark \ref{nxc} below).

\begin{itemize}
\item[(\bf A)] In \cite[\textbf{Theorem 3.3(i)}]{Tin2}, Tin proved that 
$${\rm Ker}((\widetilde {Sq^0_*})_{(5; N_1)})\cap (\mathscr A\mathcal Q^{\otimes >}_5)_{N_1} \cong \bigoplus_{1\leq j\leq 4} (\mathcal Q\mathcal {P}_5^{+})(\omega_{(j)}),$$ 
where $\omega_{(4)} = (3,3,2,4).$ Consequently, 
$${\rm Ker}((\widetilde {Sq^0_*})_{(5; N_1)}) \cong \bigoplus_{1\leq j\leq 4}(\mathcal Q\mathcal {P}_5)(\omega_{(j)}) =  \big[\bigoplus_{1\leq j\leq 3} (\mathcal Q\mathcal {P}_5)(\omega_{(j)})\big]\bigoplus (\mathcal Q\mathcal {P}_5)(\omega_{(4)}).$$  This result, however, is incorrect (see our Theorem \ref{dlc1}(i)). We give a more detailed explanation as follows: According to the proof of Theorem \ref{dlc1}(i) below, if $x$ is an admissible monomial of degree $N_1$ such that $[x]\in {\rm Ker}((\widetilde {Sq^0_*})_{(5; N_1)}),$ then $x$ is of the form $x = x_ix_jx_ky^{2},$ where $1\leq i<j<k\leq 5$ and $y$ is an admissible monomial of degree $23$ in $\mathcal P_5.$ Then, $\omega_1(x) = 3.$ Since $y$ is admissible, by \cite[Lemma 4.3.1]{ST}, $\omega(y)$ belongs to the set $\bigg\{(3,2,2,1),\ (3,4,1,1),\ (3,4,3),\ (3,2,4)\bigg\}.$ Consequently, since $\omega_1(x) = 3,$ the weight vector $\omega(x)$ belongs to the set $$\bigg\{(3, 3,2,2,1),\ (3,3,4,1,1),\ (3,3,4,3),\ (3,3,2,4)\bigg\}.$$ This set is precisely $\bigg\{\omega_{(j)} : 1 \leq j \leq 4\bigg\}.$ However, due to \cite[Proposition 4.3.4]{ST}, $(\mathcal Q\mathcal {P}_5)(3,3,2,4)=0,$ and therefore, $\omega(x)$ must belong to the set $ \bigg\{\omega_{(j)}:\ 1\leq j\leq 3\bigg\}.$

\item[(\bf B)]  In \cite[\textbf{Theorem 3.3(i)}]{Tin2}, Tin proved that $$\dim {\rm Ker}((\widetilde {Sq^0_*})_{(5; N_1)})\cap (\mathscr A\mathcal Q^{\otimes >}_5)_{N_1} = 1178.$$ 
The conclusion is totally off-base (look at our Corollary \ref{hqtq}(I)). Moreover, Tin has provided an argument that "\textit{by direct calculations, using Theorem \ref{dlKS}, we filter out and remove the inadmissible monomials in $B_5^{\otimes >}(\omega),$  so we get $|\mathcal M_{\omega}^{\otimes 5}| = 1178$}". Here, the sets $B_5^{\otimes >}(\omega)$ and $\mathcal M_{\omega}^{\otimes 5}$ are described in Tin's paper \cite{Tin2} as follows:
$$ \begin{array}{ll}
\medskip
B_5^{\otimes >}(\omega) = \big\{x_ix_jx_kF^{2}:\ 1\leq i< j<k\leq 5,\ F\in \mathcal C^{\otimes}_5(23)\big\}\cap \mathcal {P}_5^{+},\\
\mathcal M_{\omega}^{\otimes 5} = \big\{[x]\in {\rm Ker}((\widetilde {Sq^0_*})_{(5; N_1)})\cap (\mathscr A\mathcal Q^{\otimes >}_5)_{N_1}:\ \mbox{$x$ is admissible}\big\}.
\end{array}$$
It is straightforward to notice that in both $B_5^{\otimes >}(\omega)$ and $\mathcal{M}_{\omega}^{\otimes 5},$ the notation "$\omega$" seems meaningless since the elements in these sets have no relation to "$\omega$". This raises the question: \textit{what does "$\omega$" represent in this case?} In fact, after correcting the inaccurate result in \textbf{Theorem 3.3(i)} of Tin's paper \cite{Tin2}, readers can easily observe from our Theorem \ref{dlc1} and Remark \ref{nx2} that the symbol "$\omega$" herein represents the weight vector $\omega_{(1)} = (3,3,2,2,1)$. 

The argument presented by Tin, which we mentioned above, is inherently flawed because simply eliminating inadmissible monomials does not suffice to determine the dimension of the space ${\rm Ker}((\widetilde {Sq^0_*})_{(5; N_1)})\cap (\mathscr A\mathcal Q^{\otimes >}_5)_{N_1}.$ We cannot ascertain whether the remaining monomials after exclusion form a minimal $\mathscr A$-generator set for $\mathcal {P}_5^{+}(\omega_{(1)}).$  To establish this, it is essential to demonstrate that the set of classes represented by these monomials is linearly independent within ${\rm Ker}((\widetilde {Sq^0_*})_{(5; N_1)})\cap (\mathscr A\mathcal Q^{\otimes >}_5)_{N_1}$ (see the proof of Theorem \ref{dlc1}(iii) below for more understanding). This illuminates the essential fallacy in Tin's argument as affronementioned. In addition, to verify and gain further insight into this matter, readers are encouraged to consult our earlier publication \cite{D.P4}, which details the computation of an admissible monomial basis for ${\rm Ker}((\widetilde {Sq^0_*})_{(5; N)})\cap (\mathscr A\mathcal Q^{\otimes >}_5)_{N}$ where $N = 47.$ In general, the computation process is rather lengthy and complicated. 

Recall that Tin's \textbf{Theorem 3.3(ii)} erroneously states that $$\dim  ({\rm Ker}((\widetilde {Sq^0_*})_{(5; N_1)})\cap (\mathscr A\mathcal Q^{\otimes >}_5)_{N_1}) = \sum_{1\leq j\leq 4}\dim (\mathcal Q\mathcal {P}_5^{+})(\omega_{(j)}) = 1178,$$ as evidenced by item ({\bf A}) and our Corollary \ref{hqtq}(I). Furthermore, he neglected to calculate the individual dimensions of $(\mathcal Q\mathcal {P}_5^{+})(\omega_{(j)}),$ only providing their sum. As we will see, the dimensions of $(\mathcal Q\mathcal {P}_5^{+})(\omega_{(j)})$ have been completely determined in this work (see item ({\bf A}) for $\dim (\mathcal Q\mathcal {P}_5^{+})(\omega_{(4)}),$ and see Theorem \ref{dlc1}(ii) and (iii) for $\dim (\mathcal Q\mathcal {P}_5^{+})(\omega_{(j)})$ for $1\leq j\leq 3$). Consequently, the isomorphism $${\rm Ker}((\widetilde {Sq^0_*})_{(5; N_1)})\cap (\mathscr A\mathcal Q^{\otimes >}_5)_{N_1} \cong \bigoplus_{1\leq j\leq 4} (\mathcal Q\mathcal {P}_5^{+})(\omega_{(j)}) $$ presented in \textbf{Theorem 3.3(i)} of Tin's paper appears to be inconsequential, as he was unable to explicitly calculate $\dim (\mathcal Q\mathcal {P}_5^{+})(\omega_{(j)}).$ 

We also note that the dimension of $(\mathcal Q\mathcal {P}^+_5)(\omega_{(j)})$ is commonly determined through two primary methods: direct manual calculations using established tools or the application of computational algebra software. Tin's paper \cite{Tin2}, however, does not employ either of these techniques. 

Moreover, the most critical result in Tin's paper \cite{Tin2}, \textbf{Theorem 3.3(ii)}, not only lacks a detailed proof but also fails to provide any algorithm built into computer algebra systems to verify its correctness. Specifically, there is no verification given in detail for the result $\sum_{1\leq j\leq 4}\dim(\mathcal Q\mathcal {P}^+5)(\omega_{(j)}) = 1178$ that he presented in \textbf{Theorem 3.3(ii)} in \cite{Tin2}.
Therefore, our paper synthesizes both methods to explicitly derive $\dim (\mathcal Q\mathcal {P}_5^{+})(\omega_{(j)})$ (see our Corollary \ref{hqtq}(I) and the proof of Theorem \ref{dlc1}(ii) and (iii). In particular, as stated in our Theorem \ref{dlc1}(ii), $(\mathcal Q\mathcal {P}_5)(\omega_{(2)})$ and $(\mathcal Q\mathcal {P}_5)(\omega_{(3)})$ are zero, which implies that $(\mathcal Q\mathcal {P}^+_5)(\omega_{(2)})$ and $(\mathcal Q\mathcal {P}^+_5)(\omega_{(3)})$ must likewise be zero.)

\item[(\bf C)] In \cite[\textbf{Corollary 3.4}]{Tin2}, Tin stated that there exist exactly 3053 admissible monomials of degree 49 in $\mathcal {P}_5.$ Despite this, our Corollary \ref{hqtq}(II) has shown that Tin's \textbf{Corollary 3.4} is erroneous.

\item[(\bf D)] In \cite[\textbf{Theorem 3.6}]{Tin2}, Tin stated that $$\dim (\mathscr A\mathcal Q^{\otimes}_5)_{N_d} = 3053,\ \mbox{for all $d\geq 1.$}$$ 
Nonetheless, our result in Corollary \ref{hqtq}(II) reveals that Tin's \textbf{Theorem 3.6} is wrong.

\item[(\bf E)] In \cite[\textbf{Theorem 3.7}]{Tin2}, Tin proved that $$\dim (\mathscr A\mathcal Q^{\otimes}_6)_{5(2^{d+4}-1) + N_1.2^{d+4}} = (2^{6}-1).\dim (\mathscr A\mathcal Q^{\otimes}_5)_{N_1} =  192339,$$ 
for any positive integer $d.$ However, our result presented in Corollary \ref{hqtq2} demonstrates that Tin's \textbf{Theorem 3.7} is  wholly inaccurate.

\end{itemize}
\end{rem}

Let us now delve into the detailed proof of Theorem \ref{dlc1}.

\begin{proof}[Proof of Theorem \ref{dlc1}] We commence with proving item (i) of the theorem. For a seamless presentation, we reiterate a simple argument found in the proof of \textbf{Theorem 3.3(i)} from \cite{Tin2}: If $x$ is an admissible monomial of degree $N_1$ in $\mathcal {P}_5$ such that $[x]$ belongs to ${\rm Ker}((\widetilde {Sq^0_*})_{(5; N_1)}),$ then $\omega_1(x) = 3.$ Indeed,  we observe that $z = x_1^{31}x_2^{15}x_3^{3}$ is minimal spike of degree $N_1$ in $\mathcal {P}_5$ and $\omega(z) = (3,3,2,2,1).$ Then, by Lemma \ref{dlPS}, $z$ is admissible. Because $x$ is admissible and $\deg(x) = N_1$ is odd, according to Theorem \ref{dlSi}, either $\omega_1(x) = 1$ or $\omega_1(x) = 5.$ If $\omega_1(x) = 5,$ then $x$ is of the form $\prod_{1\leq j\leq 5}x_jy^{2}$ with $y$ is a monomial of degree $N_0$ in $\mathcal {P}_5.$ Since $\prod_{1\leq j\leq 5}x_jy^{2}$ is admissible, following Theorem \ref{dlKS}, $y$ is admissible and so, $(\widetilde {Sq^0_*})_{(5; N_1)}([\prod_{1\leq j\leq 5}x_jy^{2}]) = [y]\neq 0.$ This contradicts with the fact that $[x] = [\prod_{1\leq j\leq 5}x_jy^{2}]\in {\rm Ker}((\widetilde {Sq^0_*})_{(5; N_1)}).$ Thus, we must have  $\omega_1(x) = 3$ and therefore by the above arguments,  $x$ can be represented in the form $x_jx_jx_ky^{2},$ whenever $1\leq i<j<k\leq 5$ and $y$ is an admissible monomial of degree $23$ in $\mathcal {P}_5.$ 

Now, according to \cite[Lemma 4.3.1, Proposition 4.3.4]{ST}, the weight vector $\omega(y)$ takes one of three forms: either $(3,2,2,1),$ $(3,4,1,1),$ or $(3,4,3).$ Given that $\omega_1(x) = 3,$ we can deduce that the weight vector $\omega(x)$ must be one of the following sequences:
\begin{align*}
\omega_{(1)}:= (3,3,2,2,1), \
\omega_{(2)}:= (3,3,4,1,1), \text{ or} \
\omega_{(3)}:= (3,3,4,3).
\end{align*}
Thus, we obtain an isomorphism ${\rm Ker}((\widetilde {Sq^0_*})_{(5; N_1)}) \cong \bigoplus_{1\leq j\leq 3} (\mathcal Q\mathcal {P}_5)(\omega_{(j)}),$ thereby confirming the validity of item (i) of the theorem.

\medskip 

We will now proceed to prove item (ii) of the theorem. Suppose that $u$ is a monomial of degree $N_1$ in the $\mathscr A$-module $\mathcal {P}_5$ such that either $[u]_{\omega_{(2)}}\in  (\mathcal Q\mathcal {P}_5)(\omega_{(2)})$ or $[u]_{\omega_{(3)}}\in  (\mathcal Q\mathcal {P}_5)(\omega_{(3)}).$ By a simple computation, we notice that since $u = x_jx_jx_ky^{2},$ $u$ has the form $vw^{8},$ where $w$ is an suitable monomial of degree $3$ in $\mathcal {P}_5$ and $v$ is one of the following monomials:

\begin{center}
\begin{tabular}{llll}
${\rm v}_{1}=x_1^{3}x_2^{4}x_3^{4}x_4^{7}x_5^{7},$ & ${\rm v}_{2}=x_1^{3}x_2^{4}x_3^{5}x_4^{6}x_5^{7},$ & ${\rm v}_{3}=x_1^{3}x_2^{4}x_3^{5}x_4^{7}x_5^{6},$ & ${\rm v}_{4}=x_1^{3}x_2^{4}x_3^{6}x_4^{5}x_5^{7},$ \\
${\rm v}_{5}=x_1^{3}x_2^{4}x_3^{6}x_4^{7}x_5^{5},$ & ${\rm v}_{6}=x_1^{3}x_2^{4}x_3^{7}x_4^{4}x_5^{7},$ & ${\rm v}_{7}=x_1^{3}x_2^{4}x_3^{7}x_4^{5}x_5^{6},$ & ${\rm v}_{8}=x_1^{3}x_2^{4}x_3^{7}x_4^{6}x_5^{5},$ \\
${\rm v}_{9}=x_1^{3}x_2^{4}x_3^{7}x_4^{7}x_5^{4},$ & ${\rm v}_{10}=x_1^{3}x_2^{5}x_3^{4}x_4^{6}x_5^{7},$ & ${\rm v}_{11}=x_1^{3}x_2^{5}x_3^{4}x_4^{7}x_5^{6},$ & ${\rm v}_{12}=x_1^{3}x_2^{5}x_3^{5}x_4^{6}x_5^{6},$ \\
${\rm v}_{13}=x_1^{3}x_2^{5}x_3^{6}x_4^{4}x_5^{7},$ & ${\rm v}_{14}=x_1^{3}x_2^{5}x_3^{6}x_4^{5}x_5^{6},$ & ${\rm v}_{15}=x_1^{3}x_2^{5}x_3^{6}x_4^{6}x_5^{5},$ & ${\rm v}_{16}=x_1^{3}x_2^{5}x_3^{6}x_4^{7}x_5^{4},$ \\
${\rm v}_{17}=x_1^{3}x_2^{5}x_3^{7}x_4^{4}x_5^{6},$ & ${\rm v}_{18}=x_1^{3}x_2^{5}x_3^{7}x_4^{6}x_5^{4},$ & ${\rm v}_{19}=x_1^{3}x_2^{6}x_3^{4}x_4^{5}x_5^{7},$ & ${\rm v}_{20}=x_1^{3}x_2^{6}x_3^{4}x_4^{7}x_5^{5},$ \\
${\rm v}_{21}=x_1^{3}x_2^{6}x_3^{5}x_4^{4}x_5^{7},$ & ${\rm v}_{22}=x_1^{3}x_2^{6}x_3^{5}x_4^{5}x_5^{6},$ & ${\rm v}_{23}=x_1^{3}x_2^{6}x_3^{5}x_4^{6}x_5^{5},$ & ${\rm v}_{24}=x_1^{3}x_2^{6}x_3^{5}x_4^{7}x_5^{4},$ \\
${\rm v}_{25}=x_1^{3}x_2^{6}x_3^{6}x_4^{5}x_5^{5},$ & ${\rm v}_{26}=x_1^{3}x_2^{6}x_3^{7}x_4^{4}x_5^{5},$ & ${\rm v}_{27}=x_1^{3}x_2^{6}x_3^{7}x_4^{5}x_5^{4},$ & ${\rm v}_{28}=x_1^{3}x_2^{7}x_3^{4}x_4^{4}x_5^{7},$ \\
${\rm v}_{29}=x_1^{3}x_2^{7}x_3^{4}x_4^{5}x_5^{6},$ & ${\rm v}_{30}=x_1^{3}x_2^{7}x_3^{4}x_4^{6}x_5^{5},$ & ${\rm v}_{31}=x_1^{3}x_2^{7}x_3^{4}x_4^{7}x_5^{4},$ & ${\rm v}_{32}=x_1^{3}x_2^{7}x_3^{5}x_4^{4}x_5^{6},$ \\
${\rm v}_{33}=x_1^{3}x_2^{7}x_3^{5}x_4^{6}x_5^{4},$ & ${\rm v}_{34}=x_1^{3}x_2^{7}x_3^{6}x_4^{4}x_5^{5},$ & ${\rm v}_{35}=x_1^{3}x_2^{7}x_3^{6}x_4^{5}x_5^{4},$ & ${\rm v}_{36}=x_1^{3}x_2^{7}x_3^{7}x_4^{4}x_5^{4},$ \\
${\rm v}_{37}=x_1^{7}x_2^{2}x_3^{4}x_4^{5}x_5^{7},$ & ${\rm v}_{38}=x_1^{7}x_2^{2}x_3^{4}x_4^{7}x_5^{5},$ & ${\rm v}_{39}=x_1^{7}x_2^{2}x_3^{5}x_4^{4}x_5^{7},$ & ${\rm v}_{40}=x_1^{7}x_2^{2}x_3^{5}x_4^{5}x_5^{6},$ \\
${\rm v}_{41}=x_1^{7}x_2^{2}x_3^{5}x_4^{6}x_5^{5},$ & ${\rm v}_{42}=x_1^{7}x_2^{2}x_3^{5}x_4^{7}x_5^{4},$ & ${\rm v}_{43}=x_1^{7}x_2^{2}x_3^{6}x_4^{5}x_5^{5},$ & ${\rm v}_{44}=x_1^{7}x_2^{2}x_3^{7}x_4^{4}x_5^{5},$ \\
${\rm v}_{45}=x_1^{7}x_2^{2}x_3^{7}x_4^{5}x_5^{4},$ & ${\rm v}_{46}=x_1^{7}x_2^{3}x_3^{4}x_4^{4}x_5^{7},$ & ${\rm v}_{47}=x_1^{7}x_2^{3}x_3^{4}x_4^{5}x_5^{6},$ & ${\rm v}_{48}=x_1^{7}x_2^{3}x_3^{4}x_4^{6}x_5^{5},$ \\
${\rm v}_{49}=x_1^{7}x_2^{3}x_3^{4}x_4^{7}x_5^{4},$ & ${\rm v}_{50}=x_1^{7}x_2^{3}x_3^{5}x_4^{4}x_5^{6},$ & ${\rm v}_{51}=x_1^{7}x_2^{3}x_3^{5}x_4^{6}x_5^{4},$ & ${\rm v}_{52}=x_1^{7}x_2^{3}x_3^{6}x_4^{4}x_5^{5},$ \\
${\rm v}_{53}=x_1^{7}x_2^{3}x_3^{6}x_4^{5}x_5^{4},$ & ${\rm v}_{54}=x_1^{7}x_2^{3}x_3^{7}x_4^{4}x_5^{4},$ & ${\rm v}_{55}=x_1^{7}x_2^{6}x_3^{2}x_4^{5}x_5^{5},$ & ${\rm v}_{56}=x_1^{7}x_2^{6}x_3^{3}x_4^{4}x_5^{5},$ \\
${\rm v}_{57}=x_1^{7}x_2^{6}x_3^{3}x_4^{5}x_5^{4},$ & ${\rm v}_{58}=x_1^{7}x_2^{7}x_3^{2}x_4^{4}x_5^{5},$ & ${\rm v}_{59}=x_1^{7}x_2^{7}x_3^{2}x_4^{5}x_5^{4},$ & ${\rm v}_{60}=x_1^{7}x_2^{7}x_3^{3}x_4^{4}x_5^{4}.$
\end{tabular}%
\end{center}

For instance, if $u = x_jx_jx_ky^{2} = x_1^{3}x_2^{6}x_3^{4}x_4^{5}x_5^{31}$ then $u = x_1^{3}x_2^{6}x_3^{4}x_4^{5}x_5^{7}.(x_5^{3})^{8} = v_{19}w^{8},$ where $w = x_5^{3}\in \mathcal {P}_5$ and $[u]_{\omega_{(2)}}\in (\mathcal Q\mathcal {P}_5)(\omega_{(2)}).$ If $u = x_jx_jx_ky^{2} = x_1^{3}x_2^{6}x_3^{12}x_4^{13}x_5^{15}$ then $u = (x_1^{3}x_2^{6}x_3^{4}x_4^{5}x_5^{7}).(x_3x_4x_5)^{8} = v_{19}w^{8},$ where $w = x_3x_4x_5\in \mathcal {P}_5$ and $[u]_{\omega_{(3)}}\in (\mathcal Q\mathcal {P}_5)(\omega_{(3)}).$

It is straightforward to check that $\omega(v_j) = (3,3,4)$ for $1\leq j\leq 60.$ Furthermore, every monomial $v_j$ exhibits strict inadmissibility (refer to \cite{N.S}), consequently ensuring their inadmissibility. We observe that $\omega_3(v_j) = 4\neq 0$ and $\omega_i(v_j) = 0$ for all $i > 3$ and $1\leq j\leq 60.$ Due to Theorem \ref{dlKS}(ii), $u = x_jx_jx_ky^{2} = vw^{7}$ is inadmissible. Hence, $[u]_{\omega_{(j)}} = 0$ for all $[u]_{\omega_{(j)}}\in (\mathcal Q\mathcal {P}_5)(\omega_{(j)}),\, j = 2,\, 3.$ This completes the proof for item (ii). 

\medskip

Finally, we proceed to prove item (iii) of the theorem. To address this item, we use an algorithm in \texttt{SAGEMATH} to determine the dimension of $(\mathscr A\mathcal Q^{\otimes}_5)_{N_1}.$ Given the isomorphism 
$$ \begin{array}{ll}
\medskip
 (\mathscr A\mathcal Q^{\otimes}_5)_{N_1}&\cong {\rm Ker}((\widetilde {Sq^0_*})_{(5; N_1)}) \bigoplus (\mathscr A\mathcal Q^{\otimes}_5)_{N_0},\\
&\cong (\mathcal Q\mathcal {P}_5)(\omega_{(1)})\bigoplus (\mathscr A\mathcal Q^{\otimes}_5)_{N_0},
\end{array}
$$ 
and utilizing the results from items (i) and (ii), we can establish the result for item (iii). And therefore, the proof of the theorem is complete.

The calculation of $\dim (\mathscr A\mathcal Q^{\otimes}_5)_{N_1}$ is achieved through our new \texttt{SAGEMATH} algorithm presented below, which enhances the original algorithm in \texttt{MAGMA} from our earlier work \cite[Appendix A.9]{D.P7}.

\medskip

\begin{lstlisting}[basicstyle={\fontsize{8.5pt}{9.5pt}\ttfamily},]

N = 49
MAX = 16  # Largest j so that Sq(j, x) will be called

R = PolynomialRing(GF(2), 'x', 5)
x1, x2, x3, x4, x5 = R.gens()

SQH = []
for j in range(1, MAX + 1):
    PS = PowerSeriesRing(R, 't', default_prec=j + 1)
    t = PS.gen()
    Sqhom = R.hom([x + t * x*x for x in [x1, x2, x3, x4, x5]], codomain=PS)
    SQH.append(Sqhom)

def Sq(j, x):
    Sqhom = SQH[j - 1]
    c = Sqhom(x).coefficients()
    if j >= len(c):
        return R(0)
    else:
        return c[j]

M = list(R.monomials_of_degree(N))
V = VectorSpace(GF(2), len(M))

def MtoV(m):  # convert sum m of monomials to vector x
    if m.is_zero():
        return V(0)
    else:
        cv = [M.index(mm) for mm in m.monomials()]
        return V([1 if i in cv else 0 for i in range(len(M))])

def MtoSM(S):  # convert seq of monomial sum to sparse matrix
    K = V.base_ring()
    one = K(1)
    A = Matrix(K, 0, V.dimension())
    for m in S:
        if m.is_zero():
            continue
        row = [0] * V.dimension()
        for t in m.monomials():
            if t in M:
                row[M.index(t)] = one
        A = A.stack(Matrix(K, [row]))
    return A

M1 = list(R.monomials_of_degree(N - 1))
M2 = list(R.monomials_of_degree(N - 2))
M4 = list(R.monomials_of_degree(N - 4))
M8 = list(R.monomials_of_degree(N - 8))
M16 = list(R.monomials_of_degree(N - 16))


[len(m) for m in [M1, M2, M4, M8, M16]]

# Get S1, S2, S4, S8, S16, S32
S1 = [Sq(1, x) for x in M1]
S2 = [Sq(2, x) for x in M2]
S4 = [Sq(4, x) for x in M4]
S8 = [Sq(8, x) for x in M8]
S16 = [Sq(16, x) for x in M16]


CumulativeRank = lambda XS, k: Matrix(GF(2), sum([XS[i].rows() for i in range(k)], [])).rank()

# Get sparse matrices
XS = [MtoSM(S) for S in [S1, S2, S4, S8, S16]]

# Ranks
print("\nS1")
print(XS[0].rank())

print("\nS2")
print(XS[1].rank())
print(CumulativeRank(XS, 2))

print("\nS4")
print(XS[2].rank())
print(CumulativeRank(XS, 3))

print("\nS8")
print(XS[3].rank())
print(CumulativeRank(XS, 4))

print("\nS16")
print(XS[4].rank())
D = CumulativeRank(XS, 5)
print(D)

print("Dim V:", V.dimension())
print("Dim H:", D)
print("Dim V - Dim H:", V.dimension() - D)

\end{lstlisting}

\medskip 

From this algorithm, we obtain $\dim V = 292825$ and $\dim H = 289969.$ Therefore, $$\dim (\mathscr A\mathcal Q^{\otimes}_5)_{N_1 = 49}   = \dim V - \dim H = 292825-  289969 = 2856.$$

\begin{cy}\label{cyqt}
An important observation is that setting MAX to either $S{32}$ or $S{16}$ in the aforementioned algorithm does not alter the dimension of $H$, which consistently remains 289969 (refer to Remark \ref{nxc} below for further details). This implies that for $N = 49 = N_1$, choosing MAX as $S16$ is adequate. Moreover, we note that the dimension of $V$ consistently remains 292825, irrespective of which power of 2 is selected for MAX. By adjusting the parameters in the \texttt{SAGEMATH} algorithm of Remark \ref{nxc} below, specifically setting $N = 49$ instead of $N = 7$ and MAX =16 instead of MAX = 4, one can explicitly determine a basis for $(\mathscr A\mathcal Q^{\otimes}_5)_{N_1 = 49}.$ Notwithstanding, as we will see, the degree $N = 49 = N_1$ in the 5-variable case is quite large, so optimizing the algorithm for execution on a high-performance computing system is essential. 
\end{cy}

\medskip

Since $(\mathscr A\mathcal Q^{\otimes}_5)_{N_1}\cong (\mathcal Q\mathcal {P}_5)(\omega_{(1)})\bigoplus (\mathscr A\mathcal Q^{\otimes}_5)_{N_0},$ and $\dim (\mathscr A\mathcal Q^{\otimes}_5)_{N_0} = 965$ (see \cite{D.P4}), we obtain $$\dim (\mathcal Q\mathcal {P}_5)(\omega_{(1)})  = \dim (\mathscr A\mathcal Q^{\otimes}_5)_{N_1} - \dim  (\mathscr A\mathcal Q^{\otimes}_5)_{N_0} =  2856 - 965 = 1891.$$
Of course, this dimension can be verified by hand as well, with the computational steps being essentially the same as in our previous work \cite{D.P4}. Note that by Remark \ref{nx2}, we only need to check the result for the dimension of $(\mathcal Q\mathcal {P}^+_5)(\omega_{(1)})$ as stated in Corollary \ref{hqtq}(I). What follows is a complete explanation of how to compute this subspace's dimension using a direct method.

Let's revisit the fact that according to \cite{ST}, $\dim (\mathscr A\mathcal Q^{\otimes 0}_5)_{23} = \dim (\mathcal Q\mathcal P^{0}_5)(3,2,2,1) = 635,$ and owing to Proposition 4.3.2 in \cite{ST}, $\dim (\mathcal Q\mathcal P^{+}_5)(3,2,2,1) = 293.$ The proof of Theorem \ref{dlc1}(i), combined with these results, reveals that monomials $x$ in $\mathcal P^+_5(\omega_{(1)})$ invariably take the form $x = x_ix_jx_ky^2$, with $1\leq i<j<k\leq 5$ and $y$ representing an admissible element in $\mathcal P_5(3,2,2,1).$ This enables us to explicitly construct all monomials in $\mathcal P_5(\omega_{(1)}).$

For $1\leq l\leq 5,$ we define the $\mathscr A$-homomorphism $\mathsf{q}_{(l,\,5)}: \mathcal P_4\longrightarrow \mathcal P_5$ as follows:
$$ \mathsf{q}_{(l,\,5)}(x_j) = \left\{ \begin{array}{ll}
{x_j}&\text{if }\;1\leq j \leq l-1, \\[1mm]
x_{j+1}& \text{if}\; l\leq j \leq 4.
\end{array} \right.$$

Let $\mathcal N_5$ be the collection of pairs $(l, \mathscr L)$, where $\mathscr L = (l_1,l_2,\ldots, l_r)$ satisfies $1\leq l < l_1< l_2 < \ldots < l_r\leq 5$ and $0\leq r \leq 4.$ We stipulate that $\mathscr L = \emptyset$ if $r = 0$. For every $(l, \mathscr L)\in\mathcal N_5$ and $1\leq u < r$, we denote a monomial $x_{(\mathscr L,\, u)} := x_{l_u}^{2^{r-1} + 2^{r-2} + \cdots + 2^{r-u}}\prod_{u < d\leq r}x_{l_d}^{2^{r-d}}$, with the convention $t_{(\emptyset, 1)} = 1$.

Consider the $\mathbb Z/2$-linear function  $\psi_{(l, \mathscr L)}: \mathcal P_4\longrightarrow \mathcal P_5$ defined in \cite{N.S1} by:
$$  \psi_{(l, \mathscr L)}(\prod_{1\leq j\leq 4}x_j^{a_j}) = \left\{ \begin{array}{ll}
\dfrac{x_l^{2^{r} - 1}\mathsf{q}_{(l,\,5)}(\prod_{1\leq k\leq 4}x_k^{a_k})}{x_{(\mathscr L,\,u)}}& \mbox{if $\exists\ u$ satisfying \eqref{ptct} below,}\\
0&\mbox{otherwise},
\end{array}\right.$$
\begin{equation}\tag{*}\label{ptct} 
\left\{ \begin{array}{ll}
a_{l_1 - 1} +1= \cdots = a_{l_{(u-1)} - 1} +1 = 2^{r},\\[1mm]
 a_{l_{u} - 1} + 1 > 2^{r},\\[1mm]
\alpha_{r-d}(a_{l_{u} - 1}) -1 = 0,\, 1\leq d\leq u, \\[1mm]
\alpha_{r-d}(a_{l_{d}-1}) -1 = 0,\, \ u+1 \leq d \leq r.
\end{array}\right.
\end{equation}
One has the following observation. If  $\mathscr L = \emptyset,$ then $\psi_{(l, \mathscr L)} = \mathsf{q}_{(l,\,5)}$ for $1\leq l\leq 5.$ It is in fact not hard to show that if $\psi_{(l, \mathscr L)}( \prod_{1\leq k\leq 4}x_k^{a_k})\neq  0,$ then $\omega(\psi_{(l, \mathscr L)}( \prod_{1\leq k\leq 4}x_k^{a_k})) = \omega( \prod_{1\leq k\leq 4}x_k^{a_k}).$ 

Denote by $\mathscr C_{N}$ the set of all admissible monomials of degree $N$ in $\mathcal P^+_4.$ We set
$$ \begin{array}{ll}
\medskip
  \widetilde{\Phi^{+}}(\mathscr C_{N}) &:= \bigcup\limits_{(l, \mathscr L)\in\mathcal{N}_5,\;1 \leq \ell(\mathscr L) \leq 4}\psi_{(l, \mathscr L)}(\mathscr C_{N}),\\
\mathscr C(d, N) &:= \bigg\{x_l^{2^{d}-1}\mathsf{q}_{(l, 5)}(x)|\, x\in \mathscr C_{N+1-2^{d}},\, 1\leq l\leq 5\bigg\},\ 1 < d\leq 5,\, d\neq 3.
\end{array}$$
Due to \cite[Theorem 4]{Mothebe}, we can conclude that all the monomials in $\mathscr C(d, N_1)$ are admissible. A simple calculation, employing the results in \cite{N.S1} and eliminating duplicate monomials, leads us to the result: $$\bigg| \mathscr D:=\bigcup_{d}\mathscr C(d, N_1)\bigcup \widetilde{\Phi^{+}}(\mathscr C_{N_1})\bigg| = 743.$$
In the first stage, we filter inadmissible monomials using this result, Theorem \ref{dlKS}, and \cite[Lemma 4.2]{T.N}, retaining 505 monomials in $\mathcal P_5^+(\omega_{(1)})$. 
Our next step is to classify all these monomials into 159 distinct monomial groups, characterized by the following forms:

\begin{center}
\begin{tabular}{lll}
${\bf GR}_{1}/ x_1x_2^{2}x_3^{7}x_4^{9}x_5^{30}$, & ${\bf GR}_{2}/ x_1x_2^{2}x_3^{7}x_4^{11}x_5^{28}$, & ${\bf GR}_{3}/ x_1x_2^{2}x_3^{7}x_4^{12}x_5^{27}$, \\
${\bf GR}_{4}/ x_1x_2^{2}x_3^{8}x_4^{15}x_5^{23}$, & ${\bf GR}_{5}/ x_1x_2^{2}x_3^{11}x_4^{12}x_5^{23}$, & ${\bf GR}_{6}/ x_1x_2^{2}x_3^{11}x_4^{15}x_5^{20}$, \\
${\bf GR}_{7}/ x_1x_2^{2}x_3^{12}x_4^{15}x_5^{19}$, & ${\bf GR}_{8}/ x_1x_2^{2}x_3^{14}x_4^{15}x_5^{17}$, & ${\bf GR}_{9}/ x_1x_2^{2}x_3^{15}x_4^{15}x_5^{16}$, \\
${\bf GR}_{10}/ x_1x_2^{3}x_3^{4}x_4^{11}x_5^{30}$, & ${\bf GR}_{11}/ x_1x_2^{3}x_3^{4}x_4^{14}x_5^{27}$, & ${\bf GR}_{12}/ x_1x_2^{3}x_3^{6}x_4^{11}x_5^{28}$, \\
${\bf GR}_{13}/ x_1x_2^{3}x_3^{6}x_4^{12}x_5^{27}$, & ${\bf GR}_{14}/ x_1x_2^{3}x_3^{7}x_4^{8}x_5^{30}$, & ${\bf GR}_{15}/ x_1x_2^{3}x_3^{7}x_4^{10}x_5^{28}$, \\
${\bf GR}_{16}/ x_1x_2^{3}x_3^{7}x_4^{12}x_5^{26}$, & ${\bf GR}_{17}/ x_1x_2^{3}x_3^{7}x_4^{14}x_5^{24}$, & ${\bf GR}_{18}/ x_1x_2^{3}x_3^{8}x_4^{14}x_5^{23}$, \\
${\bf GR}_{19}/ x_1x_2^{3}x_3^{8}x_4^{15}x_5^{22}$, & ${\bf GR}_{20}/ x_1x_2^{3}x_3^{11}x_4^{14}x_5^{20}$, & ${\bf GR}_{21}/ x_1x_2^{3}x_3^{12}x_4^{14}x_5^{19}$, \\
${\bf GR}_{22}/ x_1x_2^{3}x_3^{14}x_4^{15}x_5^{16}$, & ${\bf GR}_{23}/ x_1x_2^{4}x_3^{7}x_4^{10}x_5^{27}$, & ${\bf GR}_{24}/ x_1x_2^{4}x_3^{7}x_4^{11}x_5^{26}$, \\
${\bf GR}_{25}/ x_1x_2^{4}x_3^{10}x_4^{15}x_5^{19}$, & ${\bf GR}_{26}/ x_1x_2^{4}x_3^{11}x_4^{15}x_5^{18}$, & ${\bf GR}_{27}/ x_1x_2^{5}x_3^{10}x_4^{14}x_5^{19}$, \\
${\bf GR}_{28}/ x_1x_2^{5}x_3^{10}x_4^{15}x_5^{18}$, & ${\bf GR}_{29}/ x_1x_2^{6}x_3^{7}x_4^{8}x_5^{27}$, & ${\bf GR}_{30}/ x_1x_2^{6}x_3^{7}x_4^{9}x_5^{26}$, \\
${\bf GR}_{31}/ x_1x_2^{6}x_3^{7}x_4^{11}x_5^{24}$, & ${\bf GR}_{32}/ x_1x_2^{6}x_3^{8}x_4^{15}x_5^{19}$, & ${\bf GR}_{33}/ x_1x_2^{6}x_3^{10}x_4^{15}x_5^{17}$, \\
${\bf GR}_{34}/ x_1x_2^{6}x_3^{11}x_4^{15}x_5^{16}$, & ${\bf GR}_{35}/ x_1x_2^{7}x_3^{7}x_4^{8}x_5^{26}$, & ${\bf GR}_{36}/ x_1x_2^{7}x_3^{7}x_4^{10}x_5^{24}$, \\
${\bf GR}_{37}/ x_1x_2^{7}x_3^{8}x_4^{10}x_5^{23}$, & ${\bf GR}_{38}/ x_1x_2^{7}x_3^{8}x_4^{11}x_5^{22}$, & ${\bf GR}_{39}/ x_1x_2^{7}x_3^{8}x_4^{15}x_5^{18}$, \\
${\bf GR}_{40}/ x_1x_2^{7}x_3^{9}x_4^{10}x_5^{22}$, & ${\bf GR}_{41}/ x_1x_2^{7}x_3^{9}x_4^{14}x_5^{18}$, & ${\bf GR}_{42}/ x_1x_2^{7}x_3^{10}x_4^{11}x_5^{20}$, \\
${\bf GR}_{43}/ x_1x_2^{7}x_3^{10}x_4^{14}x_5^{17}$, & ${\bf GR}_{44}/ x_1x_2^{7}x_3^{10}x_4^{15}x_5^{16}$, & ${\bf GR}_{45}/ x_1x_2^{7}x_3^{11}x_4^{14}x_5^{16}$, \\
${\bf GR}_{46}/ x_1^{2}x_2^{3}x_3^{4}x_4^{11}x_5^{29}$, & ${\bf GR}_{47}/ x_1^{2}x_2^{3}x_3^{4}x_4^{13}x_5^{27}$, & ${\bf GR}_{48}/ x_1^{2}x_2^{3}x_3^{5}x_4^{9}x_5^{30}$, \\
${\bf GR}_{49}/ x_1^{2}x_2^{3}x_3^{5}x_4^{14}x_5^{25}$, & ${\bf GR}_{50}/ x_1^{2}x_2^{3}x_3^{5}x_4^{15}x_5^{24}$, & ${\bf GR}_{51}/ x_1^{2}x_2^{3}x_3^{7}x_4^{8}x_5^{29}$, \\
${\bf GR}_{52}/ x_1^{2}x_2^{3}x_3^{7}x_4^{9}x_5^{28}$, & ${\bf GR}_{53}/ x_1^{2}x_2^{3}x_3^{7}x_4^{12}x_5^{25}$, & ${\bf GR}_{54}/ x_1^{2}x_2^{3}x_3^{7}x_4^{13}x_5^{24}$, \\
${\bf GR}_{55}/ x_1^{2}x_2^{3}x_3^{8}x_4^{13}x_5^{23}$, & ${\bf GR}_{56}/ x_1^{2}x_2^{3}x_3^{8}x_4^{15}x_5^{21}$, & ${\bf GR}_{57}/ x_1^{2}x_2^{3}x_3^{11}x_4^{13}x_5^{20}$, \\
${\bf GR}_{58}/ x_1^{2}x_2^{3}x_3^{12}x_4^{13}x_5^{19}$, & ${\bf GR}_{59}/ x_1^{2}x_2^{3}x_3^{13}x_4^{14}x_5^{17}$, & ${\bf GR}_{60}/ x_1^{2}x_2^{3}x_3^{13}x_4^{15}x_5^{16}$, \\
${\bf GR}_{61}/ x_1^{2}x_2^{4}x_3^{11}x_4^{15}x_5^{17}$, & ${\bf GR}_{62}/ x_1^{2}x_2^{5}x_3^{7}x_4^{8}x_5^{27}$, & ${\bf GR}_{63}/ x_1^{2}x_2^{5}x_3^{7}x_4^{9}x_5^{26}$, \\
${\bf GR}_{64}/ x_1^{2}x_2^{5}x_3^{7}x_4^{10}x_5^{25}$, & ${\bf GR}_{65}/ x_1^{2}x_2^{5}x_3^{7}x_4^{11}x_5^{24}$, & ${\bf GR}_{66}/ x_1^{2}x_2^{5}x_3^{8}x_4^{15}x_5^{19}$, \\
${\bf GR}_{67}/ x_1^{2}x_2^{5}x_3^{10}x_4^{15}x_5^{17}$, & ${\bf GR}_{68}/ x_1^{2}x_2^{7}x_3^{7}x_4^{8}x_5^{25}$, & ${\bf GR}_{69}/ x_1^{2}x_2^{7}x_3^{7}x_4^{9}x_5^{24}$, \\
${\bf GR}_{70}/ x_1^{2}x_2^{7}x_3^{8}x_4^{11}x_5^{21}$, & ${\bf GR}_{71}/ x_1^{2}x_2^{7}x_3^{8}x_4^{15}x_5^{17}$, & ${\bf GR}_{72}/ x_1^{2}x_2^{7}x_3^{9}x_4^{10}x_5^{21}$, \\
${\bf GR}_{73}/ x_1^{2}x_2^{7}x_3^{11}x_4^{13}x_5^{16}$, & ${\bf GR}_{74}/ x_1^{3}x_2^{3}x_3^{3}x_4^{12}x_5^{28}$, & ${\bf GR}_{75}/ x_1^{3}x_2^{3}x_3^{4}x_4^{9}x_5^{30}$, \\
${\bf GR}_{76}/ x_1^{3}x_2^{3}x_3^{4}x_4^{10}x_5^{29}$, & ${\bf GR}_{77}/ x_1^{3}x_2^{3}x_3^{4}x_4^{11}x_5^{28}$, & ${\bf GR}_{78}/ x_1^{3}x_2^{3}x_3^{4}x_4^{12}x_5^{27}$, \\
${\bf GR}_{79}/ x_1^{3}x_2^{3}x_3^{4}x_4^{13}x_5^{26}$, & ${\bf GR}_{80}/ x_1^{3}x_2^{3}x_3^{4}x_4^{14}x_5^{25}$, & ${\bf GR}_{81}/ x_1^{3}x_2^{3}x_3^{5}x_4^{8}x_5^{30}$, \\
${\bf GR}_{82}/ x_1^{3}x_2^{3}x_3^{5}x_4^{10}x_5^{28}$, & ${\bf GR}_{83}/ x_1^{3}x_2^{3}x_3^{5}x_4^{12}x_5^{26}$, & ${\bf GR}_{84}/ x_1^{3}x_2^{3}x_3^{5}x_4^{14}x_5^{24}$, \\
${\bf GR}_{85}/ x_1^{3}x_2^{3}x_3^{6}x_4^{8}x_5^{29}$, & ${\bf GR}_{86}/ x_1^{3}x_2^{3}x_3^{6}x_4^{13}x_5^{24}$, & ${\bf GR}_{87}/ x_1^{3}x_2^{3}x_3^{7}x_4^{8}x_5^{28}$, \\
${\bf GR}_{88}/ x_1^{3}x_2^{3}x_3^{7}x_4^{12}x_5^{24}$, & ${\bf GR}_{89}/ x_1^{3}x_2^{3}x_3^{8}x_4^{12}x_5^{23}$, & ${\bf GR}_{90}/ x_1^{3}x_2^{3}x_3^{8}x_4^{13}x_5^{22}$, \\
${\bf GR}_{91}/ x_1^{3}x_2^{3}x_3^{8}x_4^{15}x_5^{20}$, & ${\bf GR}_{92}/ x_1^{3}x_2^{3}x_3^{10}x_4^{12}x_5^{21}$, & ${\bf GR}_{93}/ x_1^{3}x_2^{3}x_3^{11}x_4^{12}x_5^{20}$, \\
${\bf GR}_{94}/ x_1^{3}x_2^{3}x_3^{12}x_4^{12}x_5^{19}$, & ${\bf GR}_{95}/ x_1^{3}x_2^{3}x_3^{12}x_4^{15}x_5^{16}$, & ${\bf GR}_{96}/ x_1^{3}x_2^{3}x_3^{13}x_4^{14}x_5^{16}$, \\
${\bf GR}_{97}/ x_1^{3}x_2^{4}x_3^{5}x_4^{10}x_5^{27}$, & ${\bf GR}_{98}/ x_1^{3}x_2^{4}x_3^{5}x_4^{11}x_5^{26}$, & ${\bf GR}_{99}/ x_1^{3}x_2^{4}x_3^{7}x_4^{8}x_5^{27}$, \\
${\bf GR}_{100}/ x_1^{3}x_2^{4}x_3^{7}x_4^{9}x_5^{26}$, & ${\bf GR}_{101}/ x_1^{3}x_2^{4}x_3^{7}x_4^{10}x_5^{25}$, & ${\bf GR}_{102}/ x_1^{3}x_2^{4}x_3^{7}x_4^{11}x_5^{24}$, \\
${\bf GR}_{103}/ x_1^{3}x_2^{4}x_3^{8}x_4^{15}x_5^{19}$, & ${\bf GR}_{104}/ x_1^{3}x_2^{4}x_3^{10}x_4^{11}x_5^{21}$, & ${\bf GR}_{105}/ x_1^{3}x_2^{4}x_3^{10}x_4^{13}x_5^{19}$, \\
${\bf GR}_{106}/ x_1^{3}x_2^{4}x_3^{10}x_4^{15}x_5^{17}$, & ${\bf GR}_{107}/ x_1^{3}x_2^{4}x_3^{11}x_4^{13}x_5^{18}$, & ${\bf GR}_{108}/ x_1^{3}x_2^{4}x_3^{11}x_4^{14}x_5^{17}$, \\
${\bf GR}_{109}/ x_1^{3}x_2^{5}x_3^{6}x_4^{8}x_5^{27}$, & ${\bf GR}_{110}/ x_1^{3}x_2^{5}x_3^{6}x_4^{9}x_5^{26}$, & ${\bf GR}_{111}/ x_1^{3}x_2^{5}x_3^{6}x_4^{10}x_5^{25}$, \\
${\bf GR}_{112}/ x_1^{3}x_2^{5}x_3^{6}x_4^{11}x_5^{24}$, & ${\bf GR}_{113}/ x_1^{3}x_2^{5}x_3^{7}x_4^{8}x_5^{26}$, & ${\bf GR}_{114}/ x_1^{3}x_2^{5}x_3^{7}x_4^{10}x_5^{24}$, \\
${\bf GR}_{115}/ x_1^{3}x_2^{5}x_3^{8}x_4^{10}x_5^{23}$, & ${\bf GR}_{116}/ x_1^{3}x_2^{5}x_3^{8}x_4^{11}x_5^{22}$, & ${\bf GR}_{117}/ x_1^{3}x_2^{5}x_3^{8}x_4^{14}x_5^{19}$, \\
${\bf GR}_{118}/ x_1^{3}x_2^{5}x_3^{8}x_4^{15}x_5^{18}$, & ${\bf GR}_{119}/ x_1^{3}x_2^{5}x_3^{9}x_4^{10}x_5^{22}$, & ${\bf GR}_{120}/ x_1^{3}x_2^{5}x_3^{9}x_4^{14}x_5^{18}$, \\
${\bf GR}_{121}/ x_1^{3}x_2^{5}x_3^{10}x_4^{11}x_5^{20}$, & ${\bf GR}_{122}/ x_1^{3}x_2^{5}x_3^{10}x_4^{12}x_5^{19}$, & ${\bf GR}_{123}/ x_1^{3}x_2^{5}x_3^{10}x_4^{13}x_5^{18}$, \\
${\bf GR}_{124}/ x_1^{3}x_2^{5}x_3^{10}x_4^{14}x_5^{17}$, & ${\bf GR}_{125}/ x_1^{3}x_2^{5}x_3^{10}x_4^{15}x_5^{16}$, & ${\bf GR}_{126}/ x_1^{3}x_2^{5}x_3^{11}x_4^{12}x_5^{18}$, \\
${\bf GR}_{127}/ x_1^{3}x_2^{5}x_3^{11}x_4^{14}x_5^{16}$, & ${\bf GR}_{128}/ x_1^{3}x_2^{6}x_3^{7}x_4^{8}x_5^{25}$, & ${\bf GR}_{129}/ x_1^{3}x_2^{6}x_3^{7}x_4^{9}x_5^{24}$, \\
\end{tabular}%
\end{center}

\newpage
\begin{center}
\begin{tabular}{lll}
${\bf GR}_{130}/ x_1^{3}x_2^{6}x_3^{8}x_4^{13}x_5^{19}$, & ${\bf GR}_{131}/ x_1^{3}x_2^{6}x_3^{8}x_4^{15}x_5^{17}$, & ${\bf GR}_{132}/ x_1^{3}x_2^{6}x_3^{10}x_4^{13}x_5^{17}$, \\
${\bf GR}_{133}/ x_1^{3}x_2^{6}x_3^{11}x_4^{13}x_5^{16}$, & ${\bf GR}_{134}/ x_1^{3}x_2^{7}x_3^{7}x_4^{8}x_5^{24}$, & ${\bf GR}_{135}/ x_1^{3}x_2^{7}x_3^{8}x_4^{9}x_5^{22}$, \\
${\bf GR}_{136}/ x_1^{3}x_2^{7}x_3^{8}x_4^{10}x_5^{21}$, & ${\bf GR}_{137}/ x_1^{3}x_2^{7}x_3^{8}x_4^{11}x_5^{20}$, & ${\bf GR}_{138}/ x_1^{3}x_2^{7}x_3^{8}x_4^{13}x_5^{18}$, \\
${\bf GR}_{139}/ x_1^{3}x_2^{7}x_3^{8}x_4^{14}x_5^{17}$, & ${\bf GR}_{140}/ x_1^{3}x_2^{7}x_3^{9}x_4^{10}x_5^{20}$, & ${\bf GR}_{141}/ x_1^{3}x_2^{7}x_3^{9}x_4^{12}x_5^{18}$, \\
${\bf GR}_{142}/ x_1^{3}x_2^{7}x_3^{9}x_4^{14}x_5^{16}$, & ${\bf GR}_{143}/ x_1^{3}x_2^{7}x_3^{10}x_4^{12}x_5^{17}$, & ${\bf GR}_{144}/ x_1^{3}x_2^{7}x_3^{10}x_4^{13}x_5^{16}$, \\
${\bf GR}_{145}/ x_1^{3}x_2^{7}x_3^{11}x_4^{12}x_5^{16}$, & ${\bf GR}_{146}/ x_1^{4}x_2^{7}x_3^{8}x_4^{11}x_5^{19}$, & ${\bf GR}_{147}/ x_1^{4}x_2^{7}x_3^{9}x_4^{10}x_5^{19}$, \\
${\bf GR}_{148}/ x_1^{4}x_2^{7}x_3^{10}x_4^{11}x_5^{17}$, & ${\bf GR}_{149}/ x_1^{5}x_2^{7}x_3^{8}x_4^{10}x_5^{19}$, & ${\bf GR}_{150}/ x_1^{5}x_2^{7}x_3^{8}x_4^{11}x_5^{18}$, \\
${\bf GR}_{151}/ x_1^{5}x_2^{7}x_3^{9}x_4^{10}x_5^{18}$, & ${\bf GR}_{152}/ x_1^{5}x_2^{7}x_3^{10}x_4^{11}x_5^{16}$, & ${\bf GR}_{153}/ x_1^{6}x_2^{7}x_3^{8}x_4^{9}x_5^{19}$, \\
${\bf GR}_{154}/ x_1^{6}x_2^{7}x_3^{8}x_4^{11}x_5^{17}$, & ${\bf GR}_{155}/ x_1^{7}x_2^{7}x_3^{8}x_4^{8}x_5^{19}$, & ${\bf GR}_{156}/ x_1^{7}x_2^{7}x_3^{8}x_4^{9}x_5^{18}$, \\
${\bf GR}_{157}/ x_1^{7}x_2^{7}x_3^{8}x_4^{10}x_5^{17}$, & ${\bf GR}_{158}/ x_1^{7}x_2^{7}x_3^{8}x_4^{11}x_5^{16}$, & ${\bf GR}_{159}/ x_1^{7}x_2^{7}x_3^{9}x_4^{10}x_5^{16}.$
\end{tabular}%
\end{center}

In our analysis of each group, we will employ the Cartan formula for direct calculations and proceed to eliminate various strictly inadmissible monomials. To provide clarity, we shall describe this procedure in detail for the following ten monomial group forms: ${\bf GR}_{8}$, ${\bf GR}_{25}$, ${\bf GR}_{33}$, ${\bf GR}_{42}$, ${\bf GR}_{44}$, ${\bf GR}_{45}$, ${\bf GR}_{63}$, ${\bf GR}_{72}$, ${\bf GR}_{124}$, and ${\bf GR}_{128}.$ Similar analyses are employed for the remaining monomial group forms. Readers should be aware that all calculations presented herein are characterized by their technical complexity.

$\bullet$ The monomial group form ${\bf GR}_{8}$ encompasses three monomials: $x_1x_2^{14}x_3^{15}x_4^{17}x_5^{2},$ $x_1x_2^{15}x_3^{14}x_4^{17}x_5^{2}$, and $x_1^{15}x_2x_3^{14}x_4^{17}x_5^{2}$. These monomials share the property of strict inadmissibility. To substantiate this claim, we will examine $x_1^{15}x_2x_3^{14}x_4^{17}x_5^{2}$ in detail, with the understanding that the proof similarly applies to the other two. Based on the Cartan formula and calculating directly, we obtain:
$$ \begin{array}{ll}
\medskip
x_1^{15}x_2x_3^{14}x_4^{17}x_5^{2}&= Sq^1\bigg(x_1^{15}x_2^2x_3^{13}x_4^{15}x_5^{3}+x_1^{16}x_2x_3^{13}x_4^{15}x_5^{3}+x_1^{15}x_2^2x_3^{11}x_4^{17}x_5^{3}\\
\medskip
&\quad\quad\quad+x_1^{15}x_2^2x_3^{11}x_4^{15}x_5^{5}+x_1^{18}x_2x_3^{11}x_4^{15}x_5^{3}\bigg)\\
\medskip
&\quad + Sq^2\bigg(x_1^{15}x_2x_3^{14}x_4^{15}x_5^{2}+x_1^{15}x_2x_3^{13}x_4^{15}x_5^{3}+x_1^{15}x_2x_3^{11}x_4^{17}x_5^{3}\\
\medskip
&\quad\quad\quad+x_1^{15}x_2x_3^{11}x_4^{15}x_5^{5}+x_1^{17}x_2x_3^{11}x_4^{15}x_5^{3}+x_1^{15}x_2x_3^{10}x_4^{15}x_5^{6}\bigg)\\  
\medskip
&\quad+Sq^4(x_1^{15}x_2x_3^{12}x_4^{15}x_5^{2}+x_1^{15}x_2x_3^{6}x_4^{15}x_5^{8})\\
\medskip
&\quad +Sq^8\bigg(x_1^9x_2x_3^{14}x_4^{15}x_5^{2}+x_1^{11}x_2x_3^{12}x_4^{15}x_5^{2}+x_1^{8}x_2x_3^{14}x_4^{15}x_5^{3}+x_1^{9}x_2^2x_3^{11}x_4^{15}x_5^{4}\\
\medskip
&\quad\quad\quad+x_1^{9}x_2^2x_3^{12}x_4^{15}x_5^{3}+x_1^{9}x_2x_3^{10}x_4^{15}x_5^{6}+x_1^{11}x_2x_3^{6}x_4^{15}x_5^{8}\bigg)\\
\medskip
&\quad+ x_1^9x_2x_3^{22}x_4^{15}x_5^{2} + x_1^9x_2x_3^{14}x_4^{23}x_5^{2}+x_1^{11}x_2x_3^{20}x_4^{15}x_5^{2}\\
\medskip
&\quad+x_1^{11}x_2x_3^{12}x_4^{23}x_5^{2}+x_1^{15}x_2x_3^{12}x_4^{19}x_5^{2}+x_1^{15}x_2x_3^{14}x_4^{16}x_5^{3}\\
\medskip
&\quad+x_1^{15}x_2x_3^{11}x_4^{18}x_5^{4}+x_1^{15}x_2x_3^{12}x_4^{18}x_5^{3}+x_1^{15}x_2x_3^{14}x_4^{15}x_5^{2}\\
\end{array}$$
\newpage
$$ \begin{array}{ll}
\medskip
&\quad+x_1^{15}x_2x_3^{11}x_4^{16}x_5^{6}+x_1^{8}x_2x_3^{22}x_4^{15}x_5^{3}+x_1^{8}x_2x_3^{14}x_4^{23}x_5^{3}\\
\medskip
&\quad+x_1^{9}x_2^2x_3^{19}x_4^{15}x_5^{4}+x_1^{9}x_2^2x_3^{11}x_4^{23}x_5^{4}+x_1^{9}x_2^2x_3^{20}x_4^{15}x_5^{3}\\
\medskip
&\quad+x_1^{9}x_2^2x_3^{12}x_4^{23}x_5^{3}+x_1^{15}x_2x_3^{10}x_4^{17}x_5^{6}+x_1^{9}x_2x_3^{18}x_4^{15}x_5^{6}\\
\medskip
&\quad+x_1^{9}x_2x_3^{10}x_4^{23}x_5^{6}+x_1^{15}x_2x_3^{14}x_4^{15}x_5^{2}+x_1^{15}x_2x_3^{6}x_4^{19}x_5^{8}\\
&\quad+x_1^{11}x_2x_3^{6}x_4^{23}x_5^{8}+x_1^{11}x_2x_3^{6}x_4^{15}x_5^{16}\ \ {\rm mod}\ (\mathcal P_5^-(\omega_{(1)})).
\end{array}
$$

$\bullet$ There are three monomials $x_1x_2^{15}x_3^{19}x_4^{4}x_5^{10},$ $x_1^{15}x_2x_3^{19}x_4^{4}x_5^{10}$, and $x_1^{15}x_2^{19}x_3x_4^{4}x_5^{10}$ in the monomial group form ${\bf GR}_{25}.$ These monomials are found to be strictly inadmissible. We will demonstrate this property starting with the monomial $x_1^{15}x_2^{19}x_3x_4^{4}x_5^{10}$, and the proof for the other two follows similarly. Applying the Cartan formula directly yields:
$$ \begin{array}{ll}
\medskip
x_1^{15}x_2^{19}x_3x_4^{4}x_5^{10}&=  Sq^1(x_1^{15}x_2^{15}x_3x_4^{5}x_5^{12} + x_1^{15}x_2^{15}x_3^4x_4^{3}x_5^{11} + x_1^{15}x_2^{15}x_3x_4^{4}x_5^{13})\\
\medskip
&\quad +Sq^2(x_1^{15}x_2^{15}x_3x_4^{6}x_5^{10} + x_1^{15}x_2^{15}x_3^2x_4^{3}x_5^{12}+x_1^{15}x_2^{15}x_3^2x_4^{4}x_5^{11}+x_1^{15}x_2^{15}x_3x_4^{2}x_5^{14})\\
\medskip
&\quad+Sq^4(x_1^{15}x_2^{15}x_3x_4^{4}x_5^{10}) + Sq^8\bigg(x_1^{11}x_2^{15}x_3x_4^{4}x_5^{10} + x_1^{9}x_2^{15}x_3x_4^{6}x_5^{10} \\
\medskip
&\quad\quad\quad + x_1^{9}x_2^{15}x_3^2x_4^{3}x_5^{12}+x_1^{8}x_2^{15}x_3^4x_4^{3}x_5^{11}+x_1^{9}x_2^{15}x_3^2x_4^{4}x_5^{11}+x_1^{9}x_2^{15}x_3x_4^{2}x_5^{14}\bigg)\\
\medskip
&\quad+ x_1^{9}x_2^{23}x_3^2x_4^{4}x_5^{11} + x_1^{9}x_2^{15}x_3^2x_4^{4}x_5^{19}  + x_1^{15}x_2^{17}x_3x_4^{2}x_5^{14}\\
\medskip
&\quad + x_1^{15}x_2^{15}x_3x_4^{2}x_5^{16} + x_1^{9}x_2^{23}x_3x_4^{2}x_5^{14} + x_1^{9}x_2^{15}x_3x_4^{2}x_5^{22}\\
\medskip
&\quad+ x_1^{11}x_2^{23}x_3x_4^{4}x_5^{10} + x_1^{11}x_2^{15}x_3x_4^{4}x_5^{18} +x_1^{9}x_2^{23}x_3x_4^{6}x_5^{10}\\
\medskip
&\quad + x_1^{9}x_2^{15}x_3x_4^{6}x_5^{18}  + x_1^{15}x_2^{17}x_3x_4^{6}x_5^{10} + x_1^{15}x_2^{17}x_3^2x_4^{3}x_5^{12}\\
\medskip
&\quad + x_1^{9}x_2^{23}x_3^2x_4^{3}x_5^{12} + x_1^{9}x_2^{15}x_3^2x_4^{3}x_5^{20}  +x_1^{15}x_2^{16}x_3^4x_4^{3}x_5^{11}\\
&\quad + x_1^{8}x_2^{23}x_3^4x_4^{3}x_5^{11} + x_1^{8}x_2^{15}x_3^4x_4^{3}x_5^{19}  +x_1^{15}x_2^{17}x_3^2x_4^{4}x_5^{11}  \ \ {\rm mod}\ (\mathcal P_5^-(\omega_{(1)})).
\end{array}$$

$\bullet$ In the examination of the monomial group form ${\bf GR}_{33}$, we identify three key monomials: $x_1x_2^{6}x_3^{15}x_4^{17}x_5^{10},$ $x_1x_2^{15}x_3^{6}x_4^{17}x_5^{10}$, and $x_1^{15}x_2x_3^{6}x_4^{17}x_5^{10}$. Our analysis shows that these monomials are strictly inadmissible, a fact easily verified. Our proof of this property begins with the monomial $x_1x_2^{6}x_3^{15}x_4^{17}x_5^{10}$, and extends to the other two monomials in a similar manner. Direct application of the Cartan formula leads us to:
$$ \begin{array}{ll}
\medskip
x_1x_2^{6}x_3^{15}x_4^{17}x_5^{10}&= Sq^1(x_1^2x_2^{5}x_3^{15}x_4^{15}x_5^{11} + x_1^2x_2^{3}x_3^{15}x_4^{17}x_5^{11} + x_1^2x_2^{3}x_3^{17}x_4^{15}x_5^{11} + x_1^2x_2^{3}x_3^{15}x_4^{15}x_5^{13}) \\
\medskip
&\quad + Sq^2\bigg(x_1x_2^{6}x_3^{15}x_4^{15}x_5^{10} +x_1x_2^{5}x_3^{15}x_4^{15}x_5^{11}+x_1x_2^{3}x_3^{15}x_4^{17}x_5^{11}\\
\medskip
&\quad\quad\quad + x_1x_2^{3}x_3^{17}x_4^{15}x_5^{11} + x_1x_2^{3}x_3^{15}x_4^{15}x_5^{13} + x_1x_2^{2}x_3^{15}x_4^{15}x_5^{14}\bigg) \\
\medskip
&\quad+ Sq^4(x_1x_2^{4}x_3^{15}x_4^{15}x_5^{10}) + Sq^8(x_1x_2^{6}x_3^{9}x_4^{15}x_5^{10}+x_1x_2^{6}x_3^{8}x_4^{15}x_5^{11})\\ 
\medskip
&\quad + x_1x_2^{4}x_3^{15}x_4^{19}x_5^{10} + x_1x_2^{4}x_3^{19}x_4^{15}x_5^{10} + x_1x_2^{6}x_3^{9}x_4^{23}x_5^{10}\\
\medskip
&\quad + x_1x_2^{6}x_3^{9}x_4^{15}x_5^{18}  +x_1x_2^{6}x_3^{15}x_4^{16}x_5^{11} + x_1x_2^{6}x_3^{8}x_4^{23}x_5^{11}\\ 
\end{array}$$
\newpage
$$ \begin{array}{ll}
\medskip
&\quad + x_1x_2^{6}x_3^{8}x_4^{15}x_5^{19}+ x_1x_2^{3}x_3^{15}x_4^{18}x_5^{12} + x_1x_2^{3}x_3^{18}x_4^{15}x_5^{12}\\
\medskip
&\quad  + x_1x_2^{4}x_3^{15}x_4^{18}x_5^{11} + x_1x_2^{4}x_3^{18}x_4^{15}x_5^{11} + x_1x_2^{3}x_3^{15}x_4^{16}x_5^{14}\\
\medskip
&\quad + x_1x_2^{3}x_3^{16}x_4^{15}x_5^{14} + x_1x_2^{2}x_3^{15}x_4^{17}x_5^{14} + x_1x_2^{2}x_3^{17}x_4^{15}x_5^{14}\\
&\quad  + x_1x_2^{2}x_3^{15}x_4^{15}x_5^{16} \ \ {\rm mod}\ (\mathcal P_5^-(\omega_{(1)})).
\end{array}
$$

$\bullet$ Examining the monomial group form ${\bf GR}_{42}$, we observe that it comprises only one monomial, $x_1^7x_2^{11}x_3^{20}x_4x_5^{10}$, which is characterized by strict inadmissibility. Direct calculations substantiate this claim, resulting in the following:
$$ \begin{array}{ll}
\medskip
x_1^7x_2^{11}x_3^{20}x_4x_5^{10}&=  Sq^1\bigg(x_1^7x_2^{11}x_3^{20}x_4x_5^{9} + x_1^7x_2^{11}x_3^{19}x_4^4x_5^{7}+x_1^7x_2^{9}x_3^{21}x_4^4x_5^{7}+x_1^7x_2^{13}x_3^{20}x_4x_5^{7}\\
\medskip
&\quad\quad\quad +x_1^7x_2^{14}x_3^{17}x_4x_5^{9}+x_1^7x_2^{13}x_3^{15}x_4x_5^{12}+x_1^7x_2^{16}x_3^{15}x_4x_5^{9}+x_1^7x_2^{13}x_3^{15}x_4^4x_5^{9}\bigg)\\
\medskip
&\quad  + Sq^2\bigg(x_1^7x_2^{11}x_3^{20}x_4^2x_5^{7}+x_1^7x_2^{10}x_3^{19}x_4^4x_5^{7}+x_1^7x_2^{9}x_3^{22}x_4^2x_5^{7}\\
\medskip
&\quad\quad\quad+x_1^7x_2^{14}x_3^{18}x_4x_5^{7}+x_1^7x_2^{14}x_3^{15}x_4x_5^{10}+x_1^7x_2^{11}x_3^{15}x_4^2x_5^{12}\\
\medskip
&\quad\quad\quad+x_1^7x_2^{14}x_3^{15}x_4^2x_5^{9}+x_1^7x_2^{11}x_3^{15}x_4^4x_5^{10}+x_1^3x_2^{11}x_3^{15}x_4^8x_5^{10}\bigg)\\ 
\medskip
&\quad+ Sq^4\bigg(x_1^5x_2^{11}x_3^{20}x_4^2x_5^{7}+x_1^4x_2^{11}x_3^{19}x_4^4x_5^{7}+x_1^5x_2^{10}x_3^{19}x_4^4x_5^{7}\\
\medskip
&\quad\quad\quad+x_1^5x_2^{9}x_3^{22}x_4^2x_5^{7}+x_1^5x_2^{14}x_3^{18}x_4x_5^{7}+x_1^5x_2^{14}x_3^{15}x_4x_5^{10}\\
\medskip
&\quad\quad\quad+x_1^5x_2^{11}x_3^{15}x_4^2x_5^{12}+x_1^5x_2^{14}x_3^{15}x_4^2x_5^{9}+x_1^5x_2^{11}x_3^{15}x_4^4x_5^{10}\\
\medskip
&\quad\quad\quad+x_1^3x_2^{13}x_3^{15}x_4^4x_5^{10}+x_1^3x_2^{7}x_3^{15}x_4^8x_5^{12}\bigg)\\
\medskip
&\quad + x_1^5x_2^{11}x_3^{24}x_4^2x_5^{7}  +x_1^5x_2^{11}x_3^{20}x_4^2x_5^{11}  +x_1^7x_2^{11}x_3^{19}x_4^4x_5^{8}\\
\medskip
&\quad\quad\quad + x_1^4x_2^{11}x_3^{19}x_4^8x_5^{7} + x_1^4x_2^{11}x_3^{19}x_4^4x_5^{11} +x_1^7x_2^{10}x_3^{19}x_4^4x_5^{9}\\
\medskip
&\quad\quad\quad + x_1^5x_2^{10}x_3^{19}x_4^8x_5^{7} +x_1^5x_2^{10}x_3^{19}x_4^4x_5^{11} +x_1^5x_2^{9}x_3^{26}x_4^2x_5^{7}\\
\medskip
&\quad\quad\quad + x_1^5x_2^{9}x_3^{22}x_4^2x_5^{11}+x_1^7x_2^{9}x_3^{24}x_4^2x_5^{7} +x_1^7x_2^{9}x_3^{22}x_4^2x_5^{9}\\
\medskip
&\quad\quad\quad+x_1^5x_2^{14}x_3^{18}x_4x_5^{11}+x_1^5x_2^{15}x_3^{18}x_4x_5^{10} + x_1^5x_2^{14}x_3^{19}x_4x_5^{10}\\
\medskip
&\quad\quad\quad+x_1^7x_2^{11}x_3^{17}x_4^2x_5^{12} + x_1^5x_2^{11}x_3^{19}x_4^2x_5^{12} + x_1^5x_2^{11}x_3^{15}x_4^2x_5^{16}\\
\medskip
&\quad\quad\quad+x_1^5x_2^{18}x_3^{15}x_4^2x_5^{9} + x_1^5x_2^{14}x_3^{19}x_4^2x_5^{9}+x_1^7x_2^{11}x_3^{17}x_4^4x_5^{10}\\
\medskip
&\quad\quad\quad+x_1^5x_2^{11}x_3^{19}x_4^4x_5^{10}+x_1^3x_2^{17}x_3^{15}x_4^4x_5^{10}  +x_1^3x_2^{13}x_3^{19}x_4^4x_5^{10}\\
&\quad\quad\quad + x_1^7x_2^{11}x_3^{20}x_4x_5^{10}+x_1^3x_2^{7}x_3^{19}x_4^8x_5^{12} + x_1^3x_2^{7}x_3^{15}x_4^8x_5^{16} \ \ {\rm mod}\ (\mathcal P_5^-(\omega_{(1)})).
 
\end{array}$$

$\bullet$ The following 16 monomials are contained in the monomial group form ${\bf GR}_{44}$:

\begin{center}
\begin{tabular}{llll}
$x_1x_2^{7}x_3^{10}x_4^{15}x_5^{16}$, & $x_1^{7}x_2x_3^{10}x_4^{15}x_5^{16}$, & $x_1^{7}x_2^{15}x_3^{16}x_4x_5^{10}$, & $x_1^{15}x_2^{7}x_3^{16}x_4x_5^{10}$, \\
$x_1x_2^{7}x_3^{15}x_4^{10}x_5^{16}$, & $x_1x_2^{7}x_3^{15}x_4^{16}x_5^{10}$, & $x_1x_2^{15}x_3^{7}x_4^{10}x_5^{16}$, & $x_1x_2^{15}x_3^{7}x_4^{16}x_5^{10}$, \\
$x_1^{7}x_2x_3^{15}x_4^{10}x_5^{16}$, & $x_1^{7}x_2x_3^{15}x_4^{16}x_5^{10}$, & $x_1^{7}x_2^{15}x_3x_4^{10}x_5^{16}$, & $x_1^{7}x_2^{15}x_3x_4^{16}x_5^{10}$, \\
$x_1^{15}x_2x_3^{7}x_4^{10}x_5^{16}$, & $x_1^{15}x_2x_3^{7}x_4^{16}x_5^{10}$, & $x_1^{15}x_2^{7}x_3x_4^{10}x_5^{16}$, & $x_1^{15}x_2^{7}x_3x_4^{16}x_5^{10}.$
\end{tabular}%

\end{center}

Among these monomials, 8 are strictly inadmissible, namely:

\begin{center}
\begin{tabular}{llll}
$x_1^{7}x_2^{15}x_3^{16}x_4x_5^{10}$, & $x_1^{15}x_2^{7}x_3^{16}x_4x_5^{10}$, & $x_1x_2^{7}x_3^{15}x_4^{16}x_5^{10}$, & $x_1x_2^{15}x_3^{7}x_4^{16}x_5^{10}$, \\
$x_1^{7}x_2x_3^{15}x_4^{16}x_5^{10}$, & $x_1^{7}x_2^{15}x_3x_4^{16}x_5^{10}$, & $x_1^{15}x_2x_3^{7}x_4^{16}x_5^{10}$, & $x_1^{15}x_2^{7}x_3x_4^{16}x_5^{10}.$
\end{tabular}%
\end{center}

We will proceed to show that $x_1^{7}x_2^{15}x_3^{16}x_4x_5^{10}$ is strictly inadmissible, and this proof technique generalizes to the remaining monomials. Indeed, a step-by-step calculation demonstrates:
$$ \begin{array}{ll}
\medskip
x_1^{7}x_2^{15}x_3^{16}x_4x_5^{10}&= Sq^1(x_1^{7}x_2^{15}x_3^{13}x_4x_5^{12} + x_1^{7}x_2^{15}x_3^{11}x_4^4x_5^{11} + x_1^{7}x_2^{15}x_3^{9}x_4^4x_5^{13})\\ 
\medskip
&\quad +  Sq^2\bigg(x_1^{7}x_2^{15}x_3^{14}x_4x_5^{10} + x_1^{7}x_2^{15}x_3^{11}x_4^2x_5^{12} + x_1^{2}x_2^{15}x_3^{11}x_4^8x_5^{11}\\
\medskip
&\quad\quad\quad + x_1^{7}x_2^{15}x_3^{10}x_4^4x_5^{11} + x_1^{3}x_2^{15}x_3^{10}x_4^8x_5^{11}+x_1^{7}x_2^{15}x_3^{9}x_4^2x_5^{14}\bigg)\\
\medskip
&\quad + Sq^4\bigg(x_1^{5}x_2^{15}x_3^{14}x_4x_5^{10}  +x_1^{5}x_2^{15}x_3^{11}x_4^2x_5^{12} + x_1^{4}x_2^{15}x_3^{11}x_4^4x_5^{11}\\
\medskip
&\quad\quad\quad  + x_1^{2}x_2^{15}x_3^{13}x_4^4x_5^{11} + x_1^{2}x_2^{15}x_3^{7}x_4^8x_5^{13} + x_1^{5}x_2^{15}x_3^{10}x_4^4x_5^{11}\\
\medskip
&\quad\quad\quad  + x_1^{3}x_2^{15}x_3^{12}x_4^4x_5^{11} + x_1^{3}x_2^{15}x_3^{6}x_4^8x_5^{13} + x_1^{5}x_2^{15}x_3^{9}x_4^2x_5^{14}\bigg) \\
\medskip
&\quad + Sq^8\bigg(x_1^{7}x_2^{9}x_3^{14}x_4x_5^{10} + x_1^{7}x_2^{9}x_3^{11}x_4^2x_5^{12} + x_1^{7}x_2^{8}x_3^{11}x_4^4x_5^{11} \\
\medskip
&\quad\quad\quad+ x_1^{7}x_2^{9}x_3^{10}x_4^4x_5^{11}+x_1^{7}x_2^{9}x_3^{9}x_4^2x_5^{14}\bigg)\\ 
\medskip
&\quad + x_1^{7}x_2^{9}x_3^{22}x_4x_5^{10} + x_1^{7}x_2^{9}x_3^{14}x_4x_5^{18} + x_1^{5}x_2^{19}x_3^{14}x_4x_5^{10} \\
\medskip
&\quad  +x_1^{5}x_2^{15}x_3^{18}x_4x_5^{10} + x_1^{5}x_2^{15}x_3^{16}x_4x_5^{10} + x_1^{5}x_2^{19}x_3^{11}x_4^2x_5^{12} \\
\medskip
&\quad + x_1^{5}x_2^{15}x_3^{11}x_4^2x_5^{16}  +x_1^{7}x_2^{9}x_3^{19}x_4^2x_5^{12} + x_1^{7}x_2^{9}x_3^{11}x_4^2x_5^{20}\\
\medskip
&\quad  + x_1^{4}x_2^{19}x_3^{11}x_4^4x_5^{11} + x_1^{2}x_2^{19}x_3^{13}x_4^4x_5^{11} + x_1^{2}x_2^{15}x_3^{17}x_4^4x_5^{11} \\
\medskip
&\quad + x_1^{2}x_2^{19}x_3^{7}x_4^8x_5^{13}  +x_1^{2}x_2^{15}x_3^{7}x_4^8x_5^{17} + x_1^{7}x_2^{8}x_3^{19}x_4^4x_5^{11}\\
\medskip
&\quad  + x_1^{7}x_2^{8}x_3^{11}x_4^4x_5^{19} + x_1^{7}x_2^{9}x_3^{18}x_4^4x_5^{11}  +x_1^{7}x_2^{9}x_3^{10}x_4^4x_5^{19}\\
\medskip
&\quad  + x_1^{5}x_2^{19}x_3^{10}x_4^4x_5^{11} +  x_1^{3}x_2^{19}x_3^{12}x_4^4x_5^{11} + x_1^{3}x_2^{15}x_3^{16}x_4^4x_5^{11} \\
\medskip
&\quad + x_1^{3}x_2^{19}x_3^{6}x_4^8x_5^{13}+x_1^{3}x_2^{15}x_3^{6}x_4^8x_5^{17} + x_1^{7}x_2^{15}x_3^{9}x_4^2x_5^{16}\\
\medskip
&\quad  + x_1^{7}x_2^{9}x_3^{17}x_4^2x_5^{14} + x_1^{7}x_2^{9}x_3^{9}x_4^2x_5^{22} + x_1^{5}x_2^{19}x_3^{9}x_4^2x_5^{14}\\
&\quad + x_1^{5}x_2^{15}x_3^{9}x_4^2x_5^{18}\ \ {\rm mod}\ (\mathcal P_5^-(\omega_{(1)})).
\end{array}$$

$\bullet$ The monomial group form ${\bf GR}_{45}$ encompasses the following 9 monomials:

\begin{center}

\begin{tabular}{lll}
$x_1x_2^{14}x_3^{7}x_4^{16}x_5^{11}$, & $x_1x_2^{14}x_3^{7}x_4^{11}x_5^{16}$, & $x_1x_2^{7}x_3^{14}x_4^{16}x_5^{11}$, \\
$x_1^{7}x_2x_3^{14}x_4^{16}x_5^{11}$, & $x_1x_2^{7}x_3^{14}x_4^{11}x_5^{16}$, & $x_1^{7}x_2x_3^{14}x_4^{11}x_5^{16}$, \\
$x_1x_2^{7}x_3^{11}x_4^{14}x_5^{16}$, & $x_1^{7}x_2x_3^{11}x_4^{14}x_5^{16}$, & $x_1^{7}x_2^{11}x_3x_4^{14}x_5^{16}.$
\end{tabular}%

\end{center}

Of these, 6 monomials are classified as strictly inadmissible:

\begin{center}
\begin{tabular}{lll}
$x_1x_2^{14}x_3^{7}x_4^{16}x_5^{11}$, & $x_1x_2^{14}x_3^{7}x_4^{11}x_5^{16}$, & $x_1x_2^{7}x_3^{14}x_4^{16}x_5^{11}$, \\
$x_1^{7}x_2x_3^{14}x_4^{16}x_5^{11}$, & $x_1x_2^{7}x_3^{14}x_4^{11}x_5^{16}$, & $x_1^{7}x_2x_3^{14}x_4^{11}x_5^{16}.$
\end{tabular}%
\end{center}

Checking the accuracy of this result is not particularly daunting.

$\bullet$ Upon inspection of ${\bf GR}_{63}$, we find it contains only one monomial: $x_1^7x_2^{9}x_3^{2}x_4^5x_5^{26}$. This monomial is strictly inadmissible, a property confirmed through explicit calculations as follows:
$$ \begin{array}{ll}
\medskip
x_1^7x_2^{9}x_3^{2}x_4^5x_5^{26}&= Sq^1\bigg(x_1^7x_2^{7}x_3^{4}x_4^9x_5^{21}+x_1^7x_2^{7}x_3^{4}x_4^{11}x_5^{19}+x_1^7x_2^{7}x_3^{3}x_4^7x_5^{24}\\
\medskip
&\quad\quad\quad  +x_1^7x_2^{7}x_3x_4^{13}x_5^{20}+x_1^7x_2^{7}x_3^{8}x_4^{11}x_5^{15}\bigg)\\
\medskip
&\quad+ Sq^2\bigg(x_1^7x_2^{7}x_3^{2}x_4^9x_5^{22}+x_1^7x_2^{7}x_3^{4}x_4^{10}x_5^{19}+x_1^7x_2^{7}x_3^{2}x_4^{11}x_5^{20}\\
\medskip
&\quad\quad\quad  +x_1^7x_2^{7}x_3^{2}x_4^7x_5^{24}+x_1^7x_2^{7}x_3x_4^{14}x_5^{18}+x_1^7x_2^{3}x_3^{4}x_4^{18}x_5^{15}\\
\medskip
&\quad\quad\quad  +x_1^7x_2^{3}x_3^{4}x_4^{10}x_5^{23}+x_1^7x_2^{7}x_3^{8}x_4^{10}x_5^{15}+x_1^3x_2^{11}x_3^{8}x_4^{10}x_5^{15}\\
\medskip
&\quad\quad\quad  +x_1^2x_2^{11}x_3^{8}x_4^{11}x_5^{15}+x_1^7x_2^{2}x_3^{4}x_4^{19}x_5^{15}+x_1^7x_2^{2}x_3^{4}x_4^{11}x_5^{23}\bigg)\\
\medskip
&\quad + Sq^4\bigg(x_1^{11}x_2^{5}x_3^{2}x_4^5x_5^{22}+x_1^5x_2^{7}x_3^{2}x_4^9x_5^{22}+x_1^5x_2^{7}x_3^{4}x_4^{10}x_5^{19}\\
\medskip
&\quad\quad\quad  +x_1^4x_2^{7}x_3^{4}x_4^{11}x_5^{19}+x_1^5x_2^{7}x_3^{2}x_4^{11}x_5^{20}+x_1^{11}x_2^{5}x_3^{2}x_4^7x_5^{20}\\
\medskip
&\quad\quad\quad  +x_1^5x_2^{7}x_3^{2}x_4^7x_5^{24}+x_1^4x_2^{7}x_3^{3}x_4^7x_5^{24}+x_1^5x_2^{7}x_3x_4^{14}x_5^{18}\\
\medskip
&\quad\quad\quad  +x_1^{11}x_2^{5}x_3^{4}x_4^{10}x_5^{15}+x_1^{11}x_2^{3}x_3^{4}x_4^{12}x_5^{15}+x_1^5x_2^{3}x_3^{4}x_4^{18}x_5^{15}\\
\medskip
&\quad\quad\quad  +x_1^5x_2^{3}x_3^{4}x_4^{10}x_5^{23}+x_1^5x_2^{7}x_3^{8}x_4^{10}x_5^{15}+x_1^3x_2^{13}x_3^{4}x_4^{10}x_5^{15}\\
\medskip
&\quad\quad\quad  +x_1^3x_2^{7}x_3^{8}x_4^{12}x_5^{15}+x_1^4x_2^{7}x_3^{8}x_4^{11}x_5^{15}+x_1^2x_2^{13}x_3^{4}x_4^{11}x_5^{15}\\
\medskip
&\quad\quad\quad  +x_1^2x_2^{7}x_3^{8}x_4^{13}x_5^{15}+x_1^{11}x_2^{4}x_3^{4}x_4^{11}x_5^{15}+x_1^5x_2^{2}x_3^{4}x_4^{11}x_5^{23}\\
\medskip
&\quad\quad\quad  +x_1^5x_2^{2}x_3^{4}x_4^{19}x_5^{15}+x_1^{11}x_2^{2}x_3^{4}x_4^{13}x_5^{15}\bigg)\\
\medskip
&\quad+Sq^8\bigg(x_1^7x_2^{5}x_3^{2}x_4^5x_5^{22}+x_1^7x_2^{5}x_3^{2}x_4^7x_5^{20}+x_1^7x_2^{5}x_3^{4}x_4^{10}x_5^{15}\\
\medskip
&\quad\quad\quad  +x_1^7x_2^{3}x_3^{4}x_4^{12}x_5^{15}+x_1^7x_2^{4}x_3^{4}x_4^{11}x_5^{15}+x_1^7x_2^{2}x_3^{4}x_4^{13}x_5^{15}\bigg)\\
\medskip
&\quad + x_1^7x_2^{5}x_3^{2}x_4^9x_5^{26}+x_1^5x_2^{11}x_3^{2}x_4^9x_5^{22} + x_1^5x_2^{7}x_3^{2}x_4^9x_5^{26}\\
\medskip
&\quad\quad\quad  +x_1^7x_2^{7}x_3^{2}x_4^9x_5^{24}+x_1^5x_2^{11}x_3^{4}x_4^{10}x_5^{19} + x_1^5x_2^{7}x_3^{8}x_4^{10}x_5^{19}\\
\medskip
&\quad\quad\quad  +x_1^7x_2^{8}x_3^{4}x_4^{11}x_5^{19}+x_1^4x_2^{11}x_3^{4}x_4^{11}x_5^{19}  +x_1^4x_2^{7}x_3^{8}x_4^{11}x_5^{19}\\
\medskip
&\quad\quad\quad  +x_1^5x_2^{11}x_3^{2}x_4^{11}x_5^{20}  +x_1^5x_2^{7}x_3^{2}x_4^{11}x_5^{24}+x_1^7x_2^{5}x_3^{2}x_4^{11}x_5^{24}\\
\medskip
&\quad\quad\quad  +x_1^7x_2^{7}x_3^{2}x_4^9x_5^{24}+x_1^5x_2^{11}x_3^{2}x_4^7x_5^{24} +x_1^5x_2^{7}x_3^{2}x_4^{11}x_5^{24}\\
\medskip
&\quad\quad\quad  +x_1^4x_2^{11}x_3^{3}x_4^7x_5^{24}+x_1^4x_2^{7}x_3^{3}x_4^{11}x_5^{24}+x_1^7x_2^{8}x_3^{3}x_4^7x_5^{24}\\
\medskip
&\quad\quad\quad  +x_1^7x_2^{7}x_3^{3}x_4^8x_5^{24}+x_1^7x_2^{9}x_3x_4^{14}x_5^{18}+x_1^5x_2^{11}x_3x_4^{14}x_5^{18}\\
\medskip
&\quad\quad\quad  +x_1^7x_2^{3}x_3^{8}x_4^{16}x_5^{15}+x_1^7x_2^{3}x_3^{8}x_4^{12}x_5^{19}+x_1^7x_2^{5}x_3^{8}x_4^{10}x_5^{19}\\
\medskip
&\quad\quad\quad  +x_1^7x_2^{3}x_3^{4}x_4^{10}x_5^{25}+x_1^5x_2^{3}x_3^{8}x_4^{18}x_5^{15} + x_1^5x_2^{3}x_3^{8}x_4^{10}x_5^{23}\\
\end{array}$$
\newpage
$$ \begin{array}{ll}
\medskip
&\quad\quad\quad  +x_1^5x_2^{3}x_3^{4}x_4^{10}x_5^{27}+x_1^7x_2^{7}x_3^{8}x_4^{10}x_5^{17}+x_1^5x_2^{7}x_3^{8}x_4^{10}x_5^{19}\\
\medskip
&\quad\quad\quad  +x_1^3x_2^{17}x_3^{4}x_4^{10}x_5^{15}+x_1^3x_2^{13}x_3^{4}x_4^{10}x_5^{19}+x_1^3x_2^{7}x_3^{8}x_4^{16}x_5^{15}\\
\medskip
&\quad\quad\quad  +x_1^3x_2^{7}x_3^{8}x_4^{12}x_5^{19}+x_1^7x_2^{7}x_3^{8}x_4^{11}x_5^{16}+x_1^4x_2^{7}x_3^{8}x_4^{11}x_5^{19}\\
\medskip
&\quad\quad\quad  +x_1^2x_2^{17}x_3^{4}x_4^{11}x_5^{15}+x_1^2x_2^{13}x_3^{4}x_4^{11}x_5^{19}+x_1^2x_2^{7}x_3^{8}x_4^{17}x_5^{15}\\
\medskip
&\quad\quad\quad  +x_1^2x_2^{7}x_3^{8}x_4^{13}x_5^{19}+x_1^7x_2^{8}x_3^{4}x_4^{11}x_5^{19}+x_1^7x_2^{4}x_3^{8}x_4^{11}x_5^{19}\\
\medskip
&\quad\quad\quad  +x_1^7x_2^{2}x_3^{4}x_4^{11}x_5^{25}+x_1^5x_2^{2}x_3^{8}x_4^{11}x_5^{23}+x_1^5x_2^{2}x_3^{4}x_4^{11}x_5^{27}\\
&\quad\quad\quad  +x_1^5x_2^{2}x_3^{8}x_4^{19}x_5^{15}+x_1^7x_2^{2}x_3^{8}x_4^{17}x_5^{15}+x_1^7x_2^{2}x_3^{8}x_4^{13}x_5^{19}\ \ {\rm mod}\ (\mathcal P_5^-(\omega_{(1)})).
\end{array}
$$

$\bullet$ Upon inspection of the monomial group form ${\bf GR}_{72},$ we find it contains only one element: $x_1^7x_2^{9}x_3^{2}x_4^{21}x_5^{10}.$ This monomial is characterized by strict inadmissibility. Corroborating this result necessitates a series of highly intricate calculations. Specifically, our use of the Cartan formula shows:
$$ \begin{array}{ll}
\medskip
x_1^7x_2^{9}x_3^{2}x_4^{21}x_5^{10}&= Sq^1\bigg(x_1^7x_2^{7}x_3^{4}x_4^{25}x_5^{5}+x_1^7x_2^{7}x_3^{8}x_4^{21}x_5^{5}+x_1^7x_2^{7}x_3^{8}x_4^{15}x_5^{11}+x_1^7x_2^{5}x_3^{8}x_4^{15}x_5^{13}\\
\medskip
&\quad\quad\quad +x_1^7x_2^{7}x_3^{4}x_4^{27}x_5^{3}+x_1^7x_2^{7}x_3x_4^{29}x_5^{4}+x_1^7x_2^{7}x_3^{4}x_4^{23}x_5^{7}\bigg)\\
\medskip
&\quad+Sq^2\bigg(x_1^7x_2^{7}x_3^{2}x_4^{25}x_5^{6}+x_1^7x_2^{7}x_3^{4}x_4^{26}x_5^{3}+x_1^7x_2^{7}x_3^{8}x_4^{22}x_5^{3}\\
\medskip
&\quad\quad\quad+x_1^7x_2^{7}x_3^{8}x_4^{19}x_5^{6}+x_1^3x_2^{11}x_3^{8}x_4^{15}x_5^{10}+x_1^7x_2^{3}x_3^{16}x_4^{15}x_5^{6}\\
\medskip
&\quad\quad\quad+x_1^7x_2^{3}x_3^{8}x_4^{23}x_5^{6}+x_1^2x_2^{11}x_3^{8}x_4^{15}x_5^{11}+x_1^7x_2^{6}x_3^{8}x_4^{15}x_5^{11}\\
\medskip
&\quad\quad\quad+x_1^3x_2^{10}x_3^{8}x_4^{15}x_5^{11}+x_1^7x_2^{3}x_3^{8}x_4^{15}x_5^{14}+x_1^7x_2^{7}x_3^{2}x_4^{27}x_5^{4}\\
\medskip
&\quad\quad\quad+x_1^7x_2^{7}x_3x_4^{30}x_5^{2}+x_1^7x_2^{7}x_3^{2}x_4^{23}x_5^{8}+x_1^8x_2^{7}x_3^{2}x_4^{23}x_5^{7}\\
\medskip
&\quad\quad\quad+x_1^7x_2^{8}x_3^{2}x_4^{23}x_5^{7}+x_1^7x_2^{7}x_3^{2}x_4^{24}x_5^{7}\bigg)\\
\medskip
&\quad+Sq^4\bigg(x_1^{11}x_2^{5}x_3^{2}x_4^{21}x_5^{6}+x_1^5x_2^{7}x_3^{2}x_4^{25}x_5^{6}+x_1^5x_2^{7}x_3^{4}x_4^{26}x_5^{3}\\
\medskip
&\quad\quad\quad+x_1^{11}x_2^{5}x_3^{4}x_4^{22}x_5^{3}+x_1^5x_2^{7}x_3^{8}x_4^{22}x_5^{3}+x_1^5x_2^{7}x_3^{8}x_4^{19}x_5^{6}\\
\medskip
&\quad\quad\quad+x_1^{11}x_2^{5}x_3^{8}x_4^{15}x_5^{6}+x_1^5x_2^{7}x_3^{8}x_4^{15}x_5^{10}+x_1^3x_2^{13}x_3^{4}x_4^{15}x_5^{10}\\
\medskip
&\quad\quad\quad+x_1^3x_2^{7}x_3^{8}x_4^{15}x_5^{12}+x_1^5x_2^{3}x_3^{16}x_4^{15}x_5^{6}+x_1^5x_2^{3}x_3^{8}x_4^{23}x_5^{6}\\
\medskip
&\quad\quad\quad+x_1^4x_2^{7}x_3^{8}x_4^{15}x_5^{11}+x_1^2x_2^{13}x_3^{4}x_4^{15}x_5^{11}+x_1^2x_2^{7}x_3^{8}x_4^{15}x_5^{13}\\
\medskip
&\quad\quad\quad+x_1^5x_2^{6}x_3^{8}x_4^{15}x_5^{11}+x_1^3x_2^{12}x_3^{4}x_4^{15}x_5^{11}+x_1^3x_2^{6}x_3^{8}x_4^{15}x_5^{13}\\
\medskip
&\quad\quad\quad+x_1^5x_2^{3}x_3^{8}x_4^{15}x_5^{14}+x_1^4x_2^{7}x_3^{4}x_4^{27}x_5^{3}+x_1^5x_2^{7}x_3^{2}x_4^{27}x_5^{4}\\
\medskip
&\quad\quad\quad+x_1^5x_2^{7}x_3x_4^{30}x_5^{2}+x_1^{11}x_2^{5}x_3^{2}x_4^{23}x_5^{4}+x_1^5x_2^{7}x_3^{2}x_4^{23}x_5^{8}\\
\medskip
&\quad\quad\quad+x_1^5x_2^{8}x_3^{2}x_4^{23}x_5^{7}+x_1^4x_2^{9}x_3^{2}x_4^{23}x_5^{7}+x_1^4x_2^{7}x_3^{2}x_4^{25}x_5^{7}\\
\medskip
&\quad\quad\quad+x_1^4x_2^{7}x_3^{2}x_4^{23}x_5^{9}+x_1^5x_2^{7}x_3^{2}x_4^{24}x_5^{7}+x_1^{11}x_2^{5}x_3^{2}x_4^{20}x_5^{7}\bigg)\\
\end{array}$$
\newpage
$$ \begin{array}{ll}
\medskip
&\quad+Sq^8\bigg(x_1^7x_2^{5}x_3^{2}x_4^{21}x_5^{6}+x_1^7x_2^{5}x_3^{4}x_4^{22}x_5^{3}+x_1^7x_2^{5}x_3^{8}x_4^{15}x_5^{6}\\
\medskip
&\quad\quad\quad+x_1^7x_2^{5}x_3^{2}x_4^{23}x_5^{4}+x_1^7x_2^{5}x_3^{2}x_4^{20}x_5^{7}\bigg)\\
\medskip
&\quad+x_1^7x_2^{5}x_3^{2}x_4^{25}x_5^{10}+x_1^5x_2^{11}x_3^{2}x_4^{25}x_5^{6}+x_1^5x_2^{7}x_3^{2}x_4^{25}x_5^{10}\\
\medskip
&\quad\quad\quad+x_1^7x_2^{7}x_3^{2}x_4^{25}x_5^{8}+x_1^5x_2^{11}x_3^{4}x_4^{26}x_5^{3}+x_1^5x_2^{7}x_3^{8}x_4^{26}x_5^{3}\\
\medskip
&\quad\quad\quad+x_1^7x_2^{5}x_3^{8}x_4^{26}x_5^{3}+x_1^5x_2^{11}x_3^{8}x_4^{22}x_5^{3}+x_1^5x_2^{7}x_3^{8}x_4^{26}x_5^{3}\\
\medskip
&\quad\quad\quad+x_1^7x_2^{7}x_3^{8}x_4^{24}x_5^{3}+x_1^5x_2^{11}x_3^{8}x_4^{19}x_5^{6}+x_1^5x_2^{7}x_3^{8}x_4^{19}x_5^{10}\\
\medskip
&\quad\quad\quad+x_1^7x_2^{7}x_3^{8}x_4^{19}x_5^{8}+x_1^7x_2^{5}x_3^{8}x_4^{19}x_5^{10}+x_1^7x_2^{7}x_3^{8}x_4^{17}x_5^{10}\\
\medskip
&\quad\quad\quad+x_1^5x_2^{7}x_3^{8}x_4^{19}x_5^{10}+x_1^3x_2^{17}x_3^{4}x_4^{15}x_5^{10}+x_1^3x_2^{13}x_3^{4}x_4^{19}x_5^{10}\\
\medskip
&\quad\quad\quad+x_1^3x_2^{7}x_3^{8}x_4^{19}x_5^{12}+x_1^3x_2^{7}x_3^{8}x_4^{15}x_5^{16}+x_1^7x_2^{3}x_3^{16}x_4^{15}x_5^{8}\\
\medskip
&\quad\quad\quad+x_1^5x_2^{3}x_3^{16}x_4^{15}x_5^{10}+x_1^7x_2^{3}x_3^{8}x_4^{25}x_5^{6}+x_1^7x_2^{3}x_3^{8}x_4^{23}x_5^{8}\\
\medskip
&\quad\quad\quad+x_1^5x_2^{3}x_3^{8}x_4^{27}x_5^{6}+x_1^5x_2^{3}x_3^{8}x_4^{23}x_5^{10}+x_1^7x_2^{7}x_3^{8}x_4^{16}x_5^{11}\\
\medskip
&\quad\quad\quad+x_1^4x_2^{7}x_3^{8}x_4^{19}x_5^{11}+x_1^2x_2^{17}x_3^{4}x_4^{15}x_5^{11}+x_1^2x_2^{13}x_3^{4}x_4^{19}x_5^{11}\\
\medskip
&\quad\quad\quad+x_1^2x_2^{7}x_3^{8}x_4^{19}x_5^{13}+x_1^2x_2^{7}x_3^{8}x_4^{15}x_5^{17}+x_1^7x_2^{6}x_3^{8}x_4^{17}x_5^{11}\\
\medskip
&\quad\quad\quad+x_1^5x_2^{6}x_3^{8}x_4^{19}x_5^{11}+x_1^3x_2^{16}x_3^{4}x_4^{15}x_5^{11}+x_1^3x_2^{12}x_3^{4}x_4^{19}x_5^{11}\\
\medskip
&\quad\quad\quad+x_1^3x_2^{6}x_3^{8}x_4^{19}x_5^{13}+x_1^3x_2^{6}x_3^{8}x_4^{15}x_5^{17}+x_1^7x_2^{3}x_3^{8}x_4^{17}x_5^{14}\\
\medskip
&\quad\quad\quad+x_1^7x_2^{3}x_3^{8}x_4^{15}x_5^{16}+x_1^5x_2^{3}x_3^{8}x_4^{19}x_5^{14}+x_1^5x_2^{3}x_3^{8}x_4^{15}x_5^{18}\\
\medskip
&\quad\quad\quad+x_1^7x_2^{8}x_3^{4}x_4^{27}x_5^{3}+x_1^4x_2^{11}x_3^{4}x_4^{27}x_5^{3}+x_1^4x_2^{7}x_3^{8}x_4^{27}x_5^{3}\\
\medskip
&\quad\quad\quad+x_1^5x_2^{11}x_3^{2}x_4^{27}x_5^{4}+x_1^5x_2^{7}x_3^{2}x_4^{27}x_5^{8}+x_1^5x_2^{11}x_3x_4^{30}x_5^{2}\\
\medskip
&\quad\quad\quad+x_1^7x_2^{9}x_3x_4^{30}x_5^{2}+x_1^7x_2^{5}x_3^{2}x_4^{27}x_5^{8}+x_1^5x_2^{11}x_3^{2}x_4^{23}x_5^{8}\\
\medskip
&\quad\quad\quad+x_1^5x_2^{7}x_3^{2}x_4^{27}x_5^{8}+x_1^7x_2^{7}x_3^{2}x_4^{25}x_5^{8}+x_1^7x_2^{8}x_3^{2}x_4^{25}x_5^{7}\\
\medskip
&\quad\quad\quad+x_1^7x_2^{8}x_3^{2}x_4^{23}x_5^{9}+x_1^5x_2^{8}x_3^{2}x_4^{27}x_5^{7}+x_1^5x_2^{8}x_3^{2}x_4^{23}x_5^{11}\\
\medskip
&\quad\quad\quad+x_1^4x_2^{9}x_3^{2}x_4^{27}x_5^{7}+x_1^4x_2^{9}x_3^{2}x_4^{23}x_5^{11}+x_1^4x_2^{11}x_3^{2}x_4^{25}x_5^{7}\\
\medskip
&\quad\quad\quad+x_1^4x_2^{7}x_3^{2}x_4^{25}x_5^{11}+x_1^4x_2^{11}x_3^{2}x_4^{23}x_5^{9}+x_1^4x_2^{7}x_3^{2}x_4^{27}x_5^{9}\\
\medskip
&\quad\quad\quad+x_1^7x_2^{7}x_3^{2}x_4^{24}x_5^{9}+x_1^5x_2^{11}x_3^{2}x_4^{24}x_5^{7}+x_1^5x_2^{7}x_3^{2}x_4^{24}x_5^{11}\\
&\quad\quad\quad+x_1^7x_2^{9}x_3^{2}x_4^{20}x_5^{11}+x_1^7x_2^{5}x_3^{2}x_4^{24}x_5^{11}\ \ {\rm mod}\ (\mathcal P_5^-(\omega_{(1)})).
\end{array}
$$

$\bullet$ The monomial group form ${\bf GR}_{124}$ contains two monomials of particular interest: $x_1^{3}x_2^{5}x_3^{10}x_4^{17}x_5^{14}$ and $x_1^{3}x_2^{5}x_3^{14}x_4^{17}x_5^{10}$. A simple computation shows that $x_1^{3}x_2^{5}x_3^{10}x_4^{17}x_5^{14}$ is strictly inadmissible. We can confirm this by employing the Cartan formula, which gives us:
$$ \begin{array}{ll}
\medskip
x_1^{3}x_2^{5}x_3^{10}x_4^{17}x_5^{14} &=  Sq^1(x_1^{3}x_2^{3}x_3^{9}x_4^{19}x_5^{14} + x_1^{3}x_2^{3}x_3^{5}x_4^{19}x_5^{18} + x_1^{3}x_2^{3}x_3^{5}x_4^{11}x_5^{26} + x_1^{3}x_2^{3}x_3^{9}x_4^{11}x_5^{22})\\
\medskip
&\quad + Sq^2\bigg(x_1^{5}x_2^{3}x_3^{6}x_4^{19}x_5^{14}+ x_1^{5}x_2^{3}x_3^{6}x_4^{11}x_5^{22}+x_1^{2}x_2^{3}x_3^{9}x_4^{19}x_5^{14} \\
\medskip
&\quad\quad\quad + x_1^{2}x_2^{3}x_3^{9}x_4^{11}x_5^{22} + x_1^{2}x_2^{3}x_3^{5}x_4^{11}x_5^{26}\bigg)\\
\end{array}$$
\newpage
$$ \begin{array}{ll}
\medskip
&\quad + Sq^4(x_1^{3}x_2^{9}x_3^{6}x_4^{13}x_5^{14} + x_1^{3}x_2^{3}x_3^{6}x_4^{19}x_5^{14} + x_1^{3}x_2^{3}x_3^{6}x_4^{11}x_5^{22}) \\
\medskip
&\quad + Sq^8(x_1^{3}x_2^{5}x_3^{6}x_4^{13}x_5^{14}) + x_1^{3}x_2^{5}x_3^{10}x_4^{13}x_5^{18}+ x_1^{3}x_2^{5}x_3^{8}x_4^{19}x_5^{14}\\
\medskip
&\quad   + x_1^{3}x_2^{3}x_3^{8}x_4^{21}x_5^{14} + x_1^{3}x_2^{5}x_3^{8}x_4^{11}x_5^{22} + x_1^{3}x_2^{5}x_3^{6}x_4^{11}x_5^{24}\\
\medskip
&\quad + x_1^{3}x_2^{3}x_3^{8}x_4^{13}x_5^{22} + x_1^{3}x_2^{3}x_3^{6}x_4^{13}x_5^{24}+  x_1^{3}x_2^{4}x_3^{9}x_4^{19}x_5^{14}\\ 
\medskip
&\quad + x_1^{3}x_2^{3}x_3^{9}x_4^{20}x_5^{14} + x_1^{3}x_2^{4}x_3^{9}x_4^{11}x_5^{22} + x_1^{3}x_2^{3}x_3^{9}x_4^{12}x_5^{22}\\
\medskip
&\quad + x_1^{3}x_2^{4}x_3^{5}x_4^{11}x_5^{26} + x_1^{3}x_2^{3}x_3^{5}x_4^{12}x_5^{26}+ x_1^{2}x_2^{5}x_3^{9}x_4^{19}x_5^{14}\\
\medskip
&\quad  + x_1^{2}x_2^{3}x_3^{9}x_4^{21}x_5^{14} + x_1^{2}x_2^{5}x_3^{9}x_4^{11}x_5^{22} + x_1^{2}x_2^{3}x_3^{9}x_4^{13}x_5^{22}\\
&\quad + x_1^{2}x_2^{5}x_3^{5}x_4^{11}x_5^{26} + x_1^{2}x_2^{3}x_3^{5}x_4^{13}x_5^{26} + x_1^{2}x_2^{3}x_3^{5}x_4^{11}x_5^{28} \ \ {\rm mod}\ (\mathcal P_5^-(\omega_{(1)})).
\end{array}$$

$\bullet$ For the monomial group form ${\bf GR}_{128},$ we find three monomials of interest: $x_1^{3}x_2^{7}x_3^{25}x_4^{6}x_5^{8}$, $x_1^{7}x_2^{3}x_3^{25}x_4^{6}x_5^{8}$, and $x_1^{7}x_2^{25}x_3^{3}x_4^{6}x_5^{8}$. Notably, the monomial $x_1^{7}x_2^{25}x_3^{3}x_4^{6}x_5^{8}$ can be shown to be strictly inadmissible. This is achievable through a direct computation process that utilizes the Cartan formula. Indeed, we obtain:
$$ \begin{array}{ll}
x_1^{7}x_2^{25}x_3^{3}x_4^{6}x_5^{8} &=   Sq^1(x_1^{7}x_2^{15}x_3^{7}x_4^{11}x_5^{8}+x_1^{7}x_2^{15}x_3x_4^{21}x_5^{4} + x_1^{7}x_2^{23}x_3x_4^{13}x_5^{4}) \\
&\quad +Sq^2\bigg(x_1^{7}x_2^{23}x_3^{3}x_4^{6}x_5^{8}+x_1^{7}x_2^{15}x_3^{7}x_4^{10}x_5^{8} + x_1^{7}x_2^{15}x_3^{3}x_4^{6}x_5^{16}\\
&\quad\quad\quad+x_1^{3}x_2^{15}x_3^{11}x_4^{10}x_5^{8}+x_1^{2}x_2^{15}x_3^{11}x_4^{11}x_5^{8}+x_1^{7}x_2^{23}x_3^{2}x_4^{11}x_5^{4}\\
&\quad\quad\quad + x_1^{7}x_2^{15}x_3^{2}x_4^{19}x_5^{4}+x_1^{7}x_2^{15}x_3x_4^{22}x_5^{2}+x_1^{7}x_2^{23}x_3x_4^{14}x_5^{2}\bigg)\\ 
\medskip
&+Sq^4\bigg(x_1^{5}x_2^{23}x_3^{3}x_4^{6}x_5^{8}+x_1^{5}x_2^{15}x_3^{3}x_4^{6}x_5^{16} + x_1^{5}x_2^{15}x_3^{7}x_4^{10}x_5^{8}\\
\medskip
&\quad\quad\quad+x_1^{11}x_2^{15}x_3^{5}x_4^{6}x_5^{8}+x_1^{3}x_2^{15}x_3^{13}x_4^{6}x_5^{8} + x_1^{3}x_2^{15}x_3^{7}x_4^{12}x_5^{8}\\
\medskip
&\quad\quad\quad+x_1^{4}x_2^{15}x_3^{7}x_4^{11}x_5^{8}+x_1^{2}x_2^{15}x_3^{13}x_4^{7}x_5^{8} + x_1^{2}x_2^{15}x_3^{7}x_4^{13}x_5^{8}\\
\medskip
&\quad\quad\quad+x_1^{11}x_2^{15}x_3^{4}x_4^{11}x_5^{4}+x_1^{5}x_2^{15}x_3^{2}x_4^{19}x_5^{4} + x_1^{5}x_2^{23}x_3^{2}x_4^{11}x_5^{4}\\
\medskip
&\quad\quad\quad+x_1^{5}x_2^{15}x_3x_4^{22}x_5^{2}+x_1^{5}x_2^{23}x_3x_4^{14}x_5^{2}\bigg)\\
\medskip
&\quad+ Sq^8(x_1^{7}x_2^{15}x_3^{5}x_4^{6}x_5^{8}+x_1^{7}x_2^{15}x_3^{4}x_4^{11}x_5^{4})\\
\medskip
 &\quad + x_1^{5}x_2^{27}x_3^{3}x_4^{6}x_5^{8} + x_1^{5}x_2^{23}x_3^{3}x_4^{10}x_5^{8} + x_1^{7}x_2^{23}x_3^{3}x_4^{8}x_5^{8}\\
\medskip
 &\quad+ x_1^{7}x_2^{19}x_3^{9}x_4^{6}x_5^{8} + x_1^{7}x_2^{19}x_3^{5}x_4^{10}x_5^{8}+ x_1^{7}x_2^{15}x_3^{3}x_4^{8}x_5^{16} \\
\medskip
 &\quad+ x_1^{5}x_2^{15}x_3^{3}x_4^{10}x_5^{16} +x_1^{5}x_2^{19}x_3^{7}x_4^{10}x_5^{8}+ x_1^{3}x_2^{19}x_3^{13}x_4^{6}x_5^{8}\\
\medskip
 &\quad + x_1^{3}x_2^{15}x_3^{17}x_4^{6}x_5^{8} + x_1^{3}x_2^{19}x_3^{7}x_4^{12}x_5^{8} + x_1^{3}x_2^{15}x_3^{7}x_4^{16}x_5^{8}\\
\medskip
 &\quad + x_1^{7}x_2^{17}x_3^{7}x_4^{10}x_5^{8}+ x_1^{7}x_2^{16}x_3^{7}x_4^{11}x_5^{8} +x_1^{4}x_2^{19}x_3^{7}x_4^{11}x_5^{8} \\
\medskip
   &\quad+ x_1^{2}x_2^{19}x_3^{7}x_4^{13}x_5^{8}+x_1^{2}x_2^{19}x_3^{13}x_4^{7}x_5^{8} + x_1^{2}x_2^{15}x_3^{7}x_4^{17}x_5^{8} \\
\medskip
&\quad + x_1^{2}x_2^{15}x_3^{17}x_4^{7}x_5^{8} + x_1^{7}x_2^{19}x_3^{8}x_4^{11}x_5^{4} + x_1^{7}x_2^{19}x_3^{4}x_4^{11}x_5^{8} \\ 
\medskip
&\quad + x_1^{5}x_2^{15}x_3^{2}x_4^{19}x_5^{8}+ x_1^{5}x_2^{23}x_3^{2}x_4^{11}x_5^{8} + x_1^{5}x_2^{27}x_3^{2}x_4^{11}x_5^{4} \\
\medskip
&\quad + x_1^{7}x_2^{25}x_3^{2}x_4^{11}x_5^{4}  + x_1^{7}x_2^{15}x_3x_4^{24}x_5^{2}+x_1^{5}x_2^{15}x_3x_4^{26}x_5^{2}\\
&\quad  +x_1^{7}x_2^{25}x_3x_4^{14}x_5^{2}  +x_1^{5}x_2^{27}x_3x_4^{14}x_5^{2}\ \ {\rm mod}\ (\mathcal P_5^-(\omega_{(1)})).
\end{array}$$

With the computational process described above, we eliminated 267 strictly inadmissible monomials, resulting in a set $\mathscr E$ of 238 remaining monomials in $\mathcal P_5^+(\omega_{(1)}).$ This set $\mathscr E$ has no elements in common with the set $\mathscr D=\bigcup_{d}\mathscr C(d, N_1)\bigcup \widetilde{\Phi^{+}}(\mathscr C_{N_1})$ described earlier. A direct computation reveals that for any admissible monomial $z \in \mathcal P_5(3,2,2,1)$ and indices $1 \leq i < j < k \leq 5$ such that $x_ix_jx_kz^2 \neq X$ (where $X \in \mathscr D \cup \mathscr E$), one of the following holds: 

$\bullet$ $x_ix_jx_kz^2$ belongs to one of the monomial group forms ${\bf GR}_i,\, 1\leq i\leq 159;$ 

$\bullet$ There exists an inadmissible monomial $W \in \mathcal P_5(3,3,2,2)$ (cf. \cite{D.P8}) such that $x_ix_jx_kz^2 = Wu^{16}$ for some $u \in (\mathcal P^0_5)_{1}$. (It should be emphasized that the number of such monomials $W$ is 984.)

Consequently, from the above calculations and due to Theorem \ref{dlKS}, $x_ix_jx_kz^2$ is inadmissible. Based on the earlier discussion of admissibility for $x = x_ix_jx_ky^2 \in \mathcal P_5(\omega_{(1)})$ and $y \in \mathcal P_5(3,2,2,1),$ we can conclude that the set $[\mathscr D\,\cup\, \mathscr E]$ generates $(\mathcal Q\mathcal P^+_5)(\omega_{(1)}).$

Finally, direct calculations enable us to demonstrate the linear independence of the set $[\mathscr D\,\cup\, \mathscr E]$ within $(\mathcal Q\mathcal P^+_5)(\omega_{(1)})$.  Denote by $\mathscr F$ a set consisting of the following 125 monomials:

\begin{center}
\begin{tabular}{lllll}
$x_1^{3}x_2x_3^{2}x_4^{12}x_5^{31}$, & $x_1^{3}x_2x_3^{2}x_4^{31}x_5^{12}$, & $x_1^{3}x_2x_3^{31}x_4^{2}x_5^{12}$, & $x_1^{3}x_2^{31}x_3x_4^{2}x_5^{12}$, & $x_1^{31}x_2x_3x_4^{2}x_5^{14}$, \\
$x_1^{31}x_2x_3x_4^{14}x_5^{2}$, & $x_1^{31}x_2x_3^{2}x_4x_5^{14}$, & $x_1^{31}x_2x_3^{14}x_4x_5^{2}$, & $x_1^{31}x_2x_3^{2}x_4^{13}x_5^{2}$, & $x_1^{31}x_2x_3x_4^{6}x_5^{10}$, \\
$x_1^{31}x_2x_3^{6}x_4x_5^{10}$, & $x_1^{31}x_2x_3^{2}x_4^{3}x_5^{12}$, & $x_1^{31}x_2x_3^{2}x_4^{12}x_5^{3}$, & $x_1^{31}x_2x_3^{3}x_4^{2}x_5^{12}$, & $x_1^{31}x_2x_3^{3}x_4^{12}x_5^{2}$, \\
$x_1^{31}x_2x_3^{2}x_4^{4}x_5^{11}$, & $x_1^{31}x_2x_3^{2}x_4^{5}x_5^{10}$, & $x_1^{31}x_2x_3^{2}x_4^{7}x_5^{8}$, & $x_1^{31}x_2x_3^{7}x_4^{2}x_5^{8}$, & $x_1^{31}x_2x_3^{3}x_4^{4}x_5^{10}$, \\
$x_1^{31}x_2x_3^{3}x_4^{6}x_5^{8}$, & $x_1^{31}x_2^{3}x_3x_4^{2}x_5^{12}$, & $x_1^{31}x_2^{3}x_3x_4^{12}x_5^{2}$, & $x_1^{31}x_2^{3}x_3x_4^{4}x_5^{10}$, & $x_1^{31}x_2^{3}x_3x_4^{6}x_5^{8}$, \\
$x_1^{31}x_2^{3}x_3^{5}x_4^{2}x_5^{8}$, & $x_1^{31}x_2^{3}x_3^{3}x_4^{4}x_5^{8}$, & $x_1^{31}x_2^{3}x_3^{5}x_4^{8}x_5^{2}$, & $x_1^{31}x_2^{7}x_3x_4^{2}x_5^{8}$, & $x_1x_2^{3}x_3^{2}x_4^{12}x_5^{31}$, \\
$x_1x_2^{3}x_3^{2}x_4^{31}x_5^{12}$, & $x_1x_2^{3}x_3^{31}x_4^{2}x_5^{12}$, & $x_1x_2^{31}x_3x_4^{2}x_5^{14}$, & $x_1x_2^{31}x_3x_4^{14}x_5^{2}$, & $x_1x_2^{31}x_3^{2}x_4x_5^{14}$, \\
$x_1x_2^{31}x_3^{14}x_4x_5^{2}$, & $x_1x_2^{31}x_3^{2}x_4^{13}x_5^{2}$, & $x_1x_2^{31}x_3x_4^{6}x_5^{10}$, & $x_1x_2^{31}x_3^{6}x_4x_5^{10}$, & $x_1x_2^{31}x_3^{2}x_4^{3}x_5^{12}$, \\
$x_1x_2^{31}x_3^{2}x_4^{12}x_5^{3}$, & $x_1x_2^{31}x_3^{3}x_4^{2}x_5^{12}$, & $x_1x_2^{31}x_3^{3}x_4^{12}x_5^{2}$, & $x_1x_2^{31}x_3^{2}x_4^{4}x_5^{11}$, & $x_1x_2^{31}x_3^{2}x_4^{5}x_5^{10}$, \\
$x_1x_2^{31}x_3^{2}x_4^{7}x_5^{8}$, & $x_1x_2^{31}x_3^{7}x_4^{2}x_5^{8}$, & $x_1x_2^{31}x_3^{3}x_4^{4}x_5^{10}$, & $x_1x_2^{31}x_3^{3}x_4^{6}x_5^{8}$, & $x_1^{3}x_2^{31}x_3x_4^{12}x_5^{2}$, \\
$x_1^{3}x_2^{31}x_3x_4^{4}x_5^{10}$, & $x_1^{3}x_2^{31}x_3x_4^{6}x_5^{8}$, & $x_1^{3}x_2^{31}x_3^{5}x_4^{2}x_5^{8}$, & $x_1^{3}x_2^{31}x_3^{3}x_4^{4}x_5^{8}$, & $x_1^{3}x_2^{31}x_3^{5}x_4^{8}x_5^{2}$, \\
$x_1^{7}x_2^{31}x_3x_4^{2}x_5^{8}$, & $x_1x_2^{2}x_3^{3}x_4^{12}x_5^{31}$, & $x_1x_2^{2}x_3^{3}x_4^{31}x_5^{12}$, & $x_1x_2x_3^{31}x_4^{2}x_5^{14}$, & $x_1x_2x_3^{31}x_4^{14}x_5^{2}$, \\
$x_1x_2^{2}x_3^{31}x_4x_5^{14}$, & $x_1x_2^{14}x_3^{31}x_4x_5^{2}$, & $x_1x_2^{2}x_3^{31}x_4^{13}x_5^{2}$, & $x_1x_2x_3^{31}x_4^{6}x_5^{10}$, & $x_1x_2^{6}x_3^{31}x_4x_5^{10}$, \\
$x_1x_2^{2}x_3^{31}x_4^{3}x_5^{12}$, & $x_1x_2^{2}x_3^{31}x_4^{12}x_5^{3}$, & $x_1x_2^{3}x_3^{31}x_4^{12}x_5^{2}$, & $x_1x_2^{2}x_3^{31}x_4^{4}x_5^{11}$, & $x_1x_2^{2}x_3^{31}x_4^{5}x_5^{10}$, \\
$x_1x_2^{2}x_3^{31}x_4^{7}x_5^{8}$, & $x_1x_2^{7}x_3^{31}x_4^{2}x_5^{8}$, & $x_1x_2^{3}x_3^{31}x_4^{4}x_5^{10}$, & $x_1x_2^{3}x_3^{31}x_4^{6}x_5^{8}$, & $x_1^{3}x_2x_3^{31}x_4^{12}x_5^{2}$, \\
$x_1^{3}x_2x_3^{31}x_4^{4}x_5^{10}$, & $x_1^{3}x_2x_3^{31}x_4^{6}x_5^{8}$, & $x_1^{3}x_2^{5}x_3^{31}x_4^{2}x_5^{8}$, & $x_1^{3}x_2^{3}x_3^{31}x_4^{4}x_5^{8}$, & $x_1^{3}x_2^{5}x_3^{31}x_4^{8}x_5^{2}$, \\
$x_1^{7}x_2x_3^{31}x_4^{2}x_5^{8}$, & $x_1x_2^{2}x_3^{12}x_4^{31}x_5^{3}$, & $x_1x_2x_3^{2}x_4^{14}x_5^{31}$, & $x_1x_2x_3^{14}x_4^{2}x_5^{31}$, & $x_1x_2^{2}x_3x_4^{14}x_5^{31}$, \\
$x_1x_2^{14}x_3x_4^{2}x_5^{31}$, & $x_1x_2^{2}x_3^{13}x_4^{2}x_5^{31}$, & $x_1x_2x_3^{6}x_4^{10}x_5^{31}$, & $x_1x_2^{6}x_3x_4^{10}x_5^{31}$, & $x_1x_2^{2}x_3^{12}x_4^{3}x_5^{31}$, \\
$x_1x_2^{3}x_3^{12}x_4^{2}x_5^{31}$, & $x_1x_2^{2}x_3^{4}x_4^{11}x_5^{31}$, & $x_1x_2^{2}x_3^{5}x_4^{10}x_5^{31}$, & $x_1x_2^{2}x_3^{7}x_4^{8}x_5^{31}$, & $x_1x_2^{7}x_3^{2}x_4^{8}x_5^{31}$, \\
$x_1x_2^{3}x_3^{4}x_4^{10}x_5^{31}$, & $x_1x_2^{3}x_3^{6}x_4^{8}x_5^{31}$, & $x_1^{3}x_2x_3^{12}x_4^{2}x_5^{31}$, & $x_1^{3}x_2x_3^{4}x_4^{10}x_5^{31}$, & $x_1^{3}x_2x_3^{6}x_4^{8}x_5^{31}$, \\
$x_1^{3}x_2^{5}x_3^{2}x_4^{8}x_5^{31}$, & $x_1^{3}x_2^{3}x_3^{4}x_4^{8}x_5^{31}$, & $x_1^{3}x_2^{5}x_3^{8}x_4^{2}x_5^{31}$, & $x_1^{7}x_2x_3^{2}x_4^{8}x_5^{31}$, & $x_1x_2x_3^{2}x_4^{31}x_5^{14}$, \\
$x_1x_2x_3^{14}x_4^{31}x_5^{2}$, & $x_1x_2^{2}x_3x_4^{31}x_5^{14}$, & $x_1x_2^{14}x_3x_4^{31}x_5^{2}$, & $x_1x_2^{2}x_3^{13}x_4^{31}x_5^{2}$, & $x_1x_2x_3^{6}x_4^{31}x_5^{10}$, \\
$x_1x_2^{6}x_3x_4^{31}x_5^{10}$, & $x_1x_2^{3}x_3^{12}x_4^{31}x_5^{2}$, & $x_1x_2^{2}x_3^{4}x_4^{31}x_5^{11}$, & $x_1x_2^{2}x_3^{5}x_4^{31}x_5^{10}$, & $x_1x_2^{2}x_3^{7}x_4^{31}x_5^{8}$, \\
$x_1x_2^{7}x_3^{2}x_4^{31}x_5^{8}$, & $x_1x_2^{3}x_3^{4}x_4^{31}x_5^{10}$, & $x_1x_2^{3}x_3^{6}x_4^{31}x_5^{8}$, & $x_1^{3}x_2x_3^{12}x_4^{31}x_5^{2}$, & $x_1^{3}x_2x_3^{4}x_4^{31}x_5^{10}$, \\
$x_1^{3}x_2x_3^{6}x_4^{31}x_5^{8}$, & $x_1^{3}x_2^{5}x_3^{2}x_4^{31}x_5^{8}$, & $x_1^{3}x_2^{3}x_3^{4}x_4^{31}x_5^{8}$, & $x_1^{3}x_2^{5}x_3^{8}x_4^{31}x_5^{2}$, & $x_1^{7}x_2x_3^{2}x_4^{31}x_5^{8}.$
\end{tabular}%
\end{center}

It is straightforward to check that $\mathscr F$ is a subset of $\bigcup_{d}\mathscr C(d, N_1).$ Let $\mathscr F_1:= (\mathscr D\cup \mathscr E)\setminus \mathscr F.$ This implies the inclusion of every one of the 856 monomials found in $\mathscr F_1.$ Let us consider the $\mathbb F_2$-subspaces $\langle {\rm [}\mathscr F{\rm ]} \rangle\subset (\mathcal Q\mathcal P_5^+)(\omega_{(1)})$ and $\langle {\rm [}\mathscr F_1{\rm ]} \rangle\subset (\mathcal Q\mathcal P_5^+)(\omega_{(1)}).$ 
\newpage
The admissibility of all monomials in $\mathscr F$ follows from the fact shown above that $\mathscr F\subset \bigcup_{d}\mathscr C(d, N_1)$, which in turn implies $\dim \langle {\rm [}\mathscr F{\rm ]} \rangle = 125.$ Additionally, we can immediately recognize that $\langle {\rm [}\mathscr F{\rm ]} \rangle$ and $\langle {\rm [}\mathscr F_1{\rm ]} \rangle$ have no non-zero vectors in common. It therefore suffices to demonstrate that the set $\langle {\rm [}\mathscr F_1{\rm ]}\rangle$ is linearly independent within the space $(\mathcal Q\mathcal P_5^+)(\omega_{(1)}).$ Our proof methodology, analogous to that in \cite{D.P4}, utilizes a result from \cite{N.S1} on $\dim (\mathscr A\mathcal Q^{\otimes >}_4)_{N_1} = 154$, alongside an algebra homomorphism $\mathcal P_5 \longrightarrow \mathcal P_4$ defined for each $(l, \mathscr L)\in\mathcal{N}_5$ by the substitution:
$$ x_j\longmapsto \begin{cases}
x_j & \text{if } 1\leq j \leq l-1, \\
\sum\limits_{m\in \mathscr L}x_{m-1} & \text{if } j = l, \\
x_{j-1} & \text{if } l+1 \leq j \leq 5.
\end{cases} $$
\end{proof}

As a final note to this section and to bolster understanding of our algorithm, we present the following remark (Remark \ref{nxc}), in which we reiterate the advantages of the aforementioned algorithm when compared with the original \texttt{MAGMA} algorithm described in \cite[Appendix A.9]{D.P7}.

\begin{rem}\label{nxc}
 In the original algorithm implemented in \texttt{MAGMA} \cite{D.P7}, when we create these spaces using the sub \scalebox{1.2}{$\langle$}\scalebox{1.2}{$\rangle$} command, it compels \texttt{MAGMA} to compute an echelon form using a dense matrix data structure. This approach consumes significant memory and computational time, especially for high-dimensional spaces.

However, given that these matrices are inherently sparse, a more efficient method exists that conserves both time and memory resources. By utilizing the sparse matrix type and computing ranks, we can implement an algorithm capable of performing both column and row operations. This approach is not constrained to using the row echelon form, resulting in substantially improved performance.

Despite these improvements, the initialization of $S1, S2,$ and subsequent spaces still required several hundred seconds. To address this, we have implemented a new version of the $Sq(j, x)$ function. This updated version employs power series that truncate higher-order terms which are not relevant to our calculations, further optimizing the process. 

In addition, we have observed that the original algorithm version \cite{D.P7} produces incorrect results in some calculations for degrees $N > 48.$ \textit{It should be emphasized that the original algorithm remains effective, correctly providing outputs for cases where the variables $s \leq 5$ and the degrees $N \leq 48$} (see also our preprint \cite{D.P3}). For instance, let's analyze the following situation: with $N = 49 = N_1 = 5(2^1 - 1) + 11.2^{1+1},$ if MAX is set to $S16,$ then 
$$\dim (\mathscr A\mathcal Q^{\otimes}_5)_{49} = \dim V - \dim H = 292825 - 289813 = 3012.$$ However, if MAX is set to $S32,$ then 
$$\dim (\mathscr A\mathcal Q^{\otimes}_5)_{49} = \dim V - \dim H = 292825 - 289772 = 3053 > 3012.$$
(Readers can straightforwardly re-verify these output results by utilizing our original algorithm in \cite{D.P7} and modifying degree $N$ from 32 to 49.)  These results underscore the flawed nature of the original algorithm's output. Indeed, \textit{theoretically, as we increase MAX to $S32,$ $S64,$ etc., the dimension of $H$ should increase, and consequently, the dimension of $(\mathscr A\mathcal Q^{\otimes}_5)_{49}$ should decrease, while the dimension of $V$ remains unchanged at 292825. Put another way, the maximal dimension of $(\mathscr A\mathcal Q^{\otimes}_5)_{49}$ is attained when MAX equals $S16,$ and this dimension is not surpassed as MAX increases beyond $S16.$}  However, as we can see in the contradictory results above, after setting MAX to $S32,$ the dimension of $H$ actually decreases, and therefore the dimension of $(\mathscr A\mathcal Q^{\otimes}_5)_{49}$ increases, becoming larger than when MAX was set to $S16.$ Thus, this represents an erroneous result when implemented in \texttt{MAGMA} using the original algorithm \cite{D.P7}. \textit{Consequently, we can observe that the incorrect results in \textbf{Theorems 3.3(ii)}, \textbf{3.6}, and \textbf{Corollary 3.4} of Tin's paper \cite{Tin2} lie in the erroneous output of the original \texttt{MAGMA} algorithm \cite{D.P7} when MAX was set to $S32.$} 

And as readers have observed in our new algorithm above, setting MAX to either $S16$ or $S32$ does not change the dimension result of $(\mathscr A\mathcal Q^{\otimes}_5)_{49}$, which always equals 2856. The explanation lies in the algorithm's output when MAX is set to $S32$: $\text{rank}(XS[5]) = 0$. This leads to $\dim(H) = \text{CumulativeRank}(XS, 6) = \text{CumulativeRank}(XS, 5) = 289969$, while $\dim V$ remains 292825 regardless of whether MAX is $S32$ or $S16.$ 
\medskip

Another example reinforcing our assertion regarding the output error of the original algorithm \cite{D.P7} is as follows: with $N = 53,$ if MAX is set to $S16,$ then $$\dim (\mathscr A\mathcal Q^{\otimes}_5)_{53} = \dim V - \dim H = 395010 - 392597 = 2413.$$ Nonetheless, setting MAX to $S32$ yields 
 $$\dim (\mathscr A\mathcal Q^{\otimes}_5)_{53} = \dim V - \dim H = 395010 - 392501 = 2509 > 2413.$$ 
When the new algorithmic method is applied as described above, with MAX set to either $S16$ or $S32,$ the correct result is (see also \cite{D.P12}):
$$ \dim (\mathscr A\mathcal Q^{\otimes}_5)_{53} = \dim V - \dim H = 395010 - 392809 = 2201.$$

\medskip

To further demonstrate the precision and effectiveness of our novel algorithm, we examine the instance where $N = 61$. In both cases where MAX equals $S16$ or $S32,$ our new algorithm generates the same output, namely:
$$ \dim (\mathscr A\mathcal Q^{\otimes}_5)_{61} = \dim V - \dim H = 677040 - 676095 = 945.$$
This result corroborates our previous publication \cite[Theorem 1.1]{D.P2}, where we employed complex hand-calculated techniques.

\medskip

Now, an additional illustration of our algorithm's efficacy is evident when comparing our results to those in Sum's paper \cite{N.S2}. While Sum's partial study of $(\mathscr A\mathcal Q^{\otimes}_5)_{62}$ proposed an estimate $1105 \leq \dim (\mathscr A\mathcal Q^{\otimes}_5)_{62} \leq 1175,$ our more precise analysis reveals:
$$\dim (\mathscr A\mathcal Q^{\otimes}_5)_{62} = \dim V - \dim H = 720720 - 719549 = 1171.$$
A discussion of this result is also included in our recent paper \cite{D.P12}. 

\newpage

Our last contribution is to construct an advanced \texttt{SAGEMATH} algorithm for explicitly identifying a basis for the indecomposables $\mathscr A\mathcal Q^{\otimes}_s$ for certain positive degrees $N.$ The algorithm's applicability extends to any $s$ and degree $N$, limited solely by the available computational memory. To exemplify, utilizing this algorithm for the case where $s = 5$ and $N = 7$, as demonstrated below, will generate exactly one explicit basis for $(\mathscr A\mathcal Q^{\otimes}_5)_7.$ This results in 110 distinct classes, each uniquely represented by one of the 110 admissible monomials, corroborating a known result by Mothebe et al. \cite{MM}.

\medskip

\medskip

\begin{lstlisting}[basicstyle={\fontsize{8.5pt}{9.5pt}\ttfamily},]

# Set up the finite field and the polynomial ring
N = 7
MAX = 4  # Largest j so that Sq(j, x) will be called
R = PolynomialRing(GF(2), ['x1', 'x2', 'x3', 'x4', 'x5'])
x1, x2, x3, x4, x5 = R.gens()

# Initialize the power series rings and homomorphisms for Steenrod squares
SQH = {}
for j in range(1, MAX + 1):
    PS = PowerSeriesRing(R, 's', j + 1)
    s = PS.gen()
    Sqhom = R.hom([x + s*x*x for x in R.gens()], codomain=PS)
    SQH[j] = Sqhom

# Define the Steenrod squares function
def Sq(j, x):
    Sqhom = SQH[j]
    c = Sqhom(x).coefficients()
    return R(0) if j + 1 > len(c) else c[j]

# Function to generate monomials of a given degree
def monomials(R, d):
    from itertools import combinations_with_replacement
    gens = R.gens()
    return [prod(x for x in m) for m in combinations_with_replacement(gens, d)]

# Get the monomials of degree N and set up the vector space
M = monomials(R, N)
V = VectorSpace(GF(2), len(M))

# Function to convert a sum of monomials to a vector
def MtoV(m):
    if m == 0:
        return V(0)
    v = V(0)
    for mm in m.monomials():
        if mm.degree() == N:
            try:
                i = M.index(mm)
                v += V.gen(i)
            except ValueError:
                pass  # Ignore monomials not in M
    return v

# Get the monomials of different degrees
M1 = monomials(R, N - 1)
M2 = monomials(R, N - 2)
M4 = monomials(R, N - 4)

# Calculate the Steenrod squares for each degree
S1 = [Sq(1, x) for x in M1]
S2 = [Sq(2, x) for x in M2]
S4 = [Sq(4, x) for x in M4]

# Convert Steenrod squares to vectors
V1 = [MtoV(s) for s in S1]
V2 = [MtoV(s) for s in S2]
V4 = [MtoV(s) for s in S4]

# Combine all vectors
all_vectors = V1 + V2 + V4

# Create the matrix
combined_matrix = Matrix(GF(2), all_vectors)

# Compute the echelon form to identify pivot columns
E = combined_matrix.echelon_form()

# Print some debug information
print(f"Dimension of combined matrix: {combined_matrix.nrows()} x {combined_matrix.ncols()}")
print(f"Rank of combined matrix: {combined_matrix.rank()}")

# Identify the pivot and non-pivot columns
pivots = [min(j for j in range(E.ncols()) if not E[i, j].is_zero()) for i in range(E.nrows()) if not E[i].is_zero()]
print(f"Number of pivot columns: {len(pivots)}")
non_pivots = [i for i in range(combined_matrix.ncols()) if i not in pivots]
print(f"Number of non-pivot columns: {len(non_pivots)}")

# Function to check if a monomial is valid
def is_valid_monomial(m):
    # Get the exponents of the monomial
    exps = m.exponents()[0]  # Get the first element, which is a tuple of exponents
    a1 = int(exps[0])
    
    # Check if a1 is of the form 2^u - 1
    if (a1 + 1) & a1 != 0:  # Check if a1 is a number of the form 2^u - 1
        return False
    
    # Check if the sum of the exponents equals N
    if sum(exps) != N:
        return False
    
    return True

# Convert monomials to vectors
monomial_vectors = {m: MtoV(m) for m in M}

# Function to check if a vector is in the span of the given set of vectors
def in_span(v, vectors):
    mat = Matrix(GF(2), vectors).transpose()
    aug = mat.augment(v)
    echelon_form = aug.echelon_form()
    return echelon_form[:,-1].is_zero()

# Filter out basis monomials that lie in the space generated by S1, S2, S4
basis_monomials = [M[i] for i in non_pivots if is_valid_monomial(M[i])]
filtered_basis_monomials = [m for m in basis_monomials if not in_span(monomial_vectors[m], all_vectors)]

if len(filtered_basis_monomials) == 0:
    print("No basis monomials found. The quotient space may be 0.")
else:
    print(f"Number of basis monomials: {len(filtered_basis_monomials)}")
    print("All basis monomials for the quotient space are:")
    cols = 4  # Number of columns
    rows = (len(filtered_basis_monomials) + cols - 1) // cols  # Ceiling division

    for i in range(rows):
        row = []
        for j in range(cols):
            idx = i + j * rows
            if idx < len(filtered_basis_monomials):
                monomial_str = str(filtered_basis_monomials[idx])
                row.append(f"{idx+1:2d}: {monomial_str:<15}")
        print("  ".join(row))


\end{lstlisting}

\medskip

\textbf{Here's what the algorithm produced:}

\medskip

\begin{lstlisting}[basicstyle={\fontsize{8.5pt}{9.5pt}\ttfamily},]

Dimension of combined matrix: 371 x 330
Rank of combined matrix: 220
Number of pivot columns: 220
Number of non-pivot columns: 110
Number of basis monomials: 110
All basis monomials for the quotient space are:
 1: x1^7             29: x1*x2^3*x3*x5^2  57: x1*x2*x4^2*x5^3  85: x2*x3^6        
 2: x1^3*x2^3*x3     30: x1*x2^3*x4^3     58: x1*x2*x4*x5^4    86: x2*x3^3*x4^3   
 3: x1^3*x2^3*x4     31: x1*x2^3*x4^2*x5  59: x1*x2*x5^5       87: x2*x3^3*x4^2*x5
 4: x1^3*x2^3*x5     32: x1*x2^3*x4*x5^2  60: x1*x3^6          88: x2*x3^3*x4*x5^2
 5: x1^3*x2*x3^3     33: x1*x2^3*x5^3     61: x1*x3^3*x4^3     89: x2*x3^3*x5^3   
 6: x1^3*x2*x3^2*x4  34: x1*x2^2*x3^3*x4  62: x1*x3^3*x4^2*x5  90: x2*x3^2*x4^3*x5
 7: x1^3*x2*x3^2*x5  35: x1*x2^2*x3^3*x5  63: x1*x3^3*x4*x5^2  91: x2*x3^2*x4*x5^3
 8: x1^3*x2*x3*x4^2  36: x1*x2^2*x3*x4^3  64: x1*x3^3*x5^3     92: x2*x3*x4^5     
 9: x1^3*x2*x3*x4*x5  37: x1*x2^2*x3*x4^2*x5  65: x1*x3^2*x4^3*x5  93: x2*x3*x4^3*x5^2
10: x1^3*x2*x3*x5^2  38: x1*x2^2*x3*x4*x5^2  66: x1*x3^2*x4*x5^3  94: x2*x3*x4^2*x5^3
11: x1^3*x2*x4^3     39: x1*x2^2*x3*x5^3  67: x1*x3*x4^5       95: x2*x3*x4*x5^4  
12: x1^3*x2*x4^2*x5  40: x1*x2^2*x4^3*x5  68: x1*x3*x4^3*x5^2  96: x2*x3*x5^5     
13: x1^3*x2*x4*x5^2  41: x1*x2^2*x4*x5^3  69: x1*x3*x4^2*x5^3  97: x2*x4^6        
14: x1^3*x2*x5^3     42: x1*x2*x3^5       70: x1*x3*x4*x5^4    98: x2*x4^3*x5^3   
15: x1^3*x3^3*x4     43: x1*x2*x3^3*x4^2  71: x1*x3*x5^5       99: x2*x4*x5^5     
16: x1^3*x3^3*x5     44: x1*x2*x3^3*x4*x5  72: x1*x4^6          100: x2*x5^6        
17: x1^3*x3*x4^3     45: x1*x2*x3^3*x5^2  73: x1*x4^3*x5^3     101: x3^7           
18: x1^3*x3*x4^2*x5  46: x1*x2*x3^2*x4^3  74: x1*x4*x5^5       102: x3^3*x4^3*x5   
19: x1^3*x3*x4*x5^2  47: x1*x2*x3^2*x4^2*x5  75: x1*x5^6          103: x3^3*x4*x5^3   
20: x1^3*x3*x5^3     48: x1*x2*x3^2*x4*x5^2  76: x2^7             104: x3*x4^6        
21: x1^3*x4^3*x5     49: x1*x2*x3^2*x5^3  77: x2^3*x3^3*x4     105: x3*x4^3*x5^3   
22: x1^3*x4*x5^3     50: x1*x2*x3*x4^4    78: x2^3*x3^3*x5     106: x3*x4*x5^5     
23: x1*x2^6          51: x1*x2*x3*x4^3*x5  79: x2^3*x3*x4^3     107: x3*x5^6        
24: x1*x2^3*x3^3     52: x1*x2*x3*x4^2*x5^2  80: x2^3*x3*x4^2*x5  108: x4^7           
25: x1*x2^3*x3^2*x4  53: x1*x2*x3*x4*x5^3  81: x2^3*x3*x4*x5^2  109: x4*x5^6        
26: x1*x2^3*x3^2*x5  54: x1*x2*x3*x5^4    82: x2^3*x3*x5^3     110: x5^7           
27: x1*x2^3*x3*x4^2  55: x1*x2*x4^5       83: x2^3*x4^3*x5   
28: x1*x2^3*x3*x4*x5  56: x1*x2*x4^3*x5^2  84: x2^3*x4*x5^3   

\end{lstlisting}

\medskip

To better highlight the algorithm's capability in explicitly finding the basis of the indecomposables $(\mathscr A\mathcal Q^{\otimes}_5)_{N},$ we will examine the case where degree $N = 13$. With a subtle optimization of this algorithm (details omitted for conciseness and for readers to explore independently), we obtain precise results for both the dimension and basis of the indecomposables $(\mathscr A\mathcal Q^{\otimes}_5)_{13}$ as follows:

\medskip

\begin{lstlisting}[basicstyle={\fontsize{8.5pt}{9.5pt}\ttfamily},]

Dimension of combined matrix: 4026 x 2380
Rank of combined matrix: 2130
Number of pivot columns: 2130
Number of non-pivot columns: 250
All basis monomials for the quotient space are:

1: x1^7*x2^3*x3^3 | 2: x1^7*x2^3*x3*x4^2
3: x1^7*x2^3*x3*x4*x5 | 4: x1^7*x2^3*x3*x5^2
5: x1^7*x2^3*x4^3 | 6: x1^7*x2^3*x4*x5^2
7: x1^7*x2^3*x5^3 | 8: x1^7*x2*x3^3*x4^2
9: x1^7*x2*x3^3*x4*x5 | 10: x1^7*x2*x3^3*x5^2
11: x1^7*x2*x3^2*x4^3 | 12: x1^7*x2*x3^2*x4*x5^2
13: x1^7*x2*x3^2*x5^3 | 14: x1^7*x2*x3*x4^3*x5
15: x1^7*x2*x3*x4^2*x5^2 | 16: x1^7*x2*x3*x4*x5^3
17: x1^7*x2*x4^3*x5^2 | 18: x1^7*x2*x4^2*x5^3
19: x1^7*x3^3*x4^3 | 20: x1^7*x3^3*x4*x5^2
21: x1^7*x3^3*x5^3 | 22: x1^7*x3*x4^3*x5^2
23: x1^7*x3*x4^2*x5^3 | 24: x1^7*x4^3*x5^3
25: x1^3*x2^7*x3^3 | 26: x1^3*x2^7*x3*x4^2
27: x1^3*x2^7*x3*x4*x5 | 28: x1^3*x2^7*x3*x5^2
29: x1^3*x2^7*x4^3 | 30: x1^3*x2^7*x4*x5^2
31: x1^3*x2^7*x5^3 | 32: x1^3*x2^5*x3^3*x4^2
33: x1^3*x2^5*x3^3*x4*x5 | 34: x1^3*x2^5*x3^3*x5^2
35: x1^3*x2^5*x3^2*x4^3 | 36: x1^3*x2^5*x3^2*x4*x5^2
37: x1^3*x2^5*x3^2*x5^3 | 38: x1^3*x2^5*x3*x4^3*x5
39: x1^3*x2^5*x3*x4^2*x5^2 | 40: x1^3*x2^5*x3*x4*x5^3
41: x1^3*x2^5*x4^3*x5^2 | 42: x1^3*x2^5*x4^2*x5^3
43: x1^3*x2^3*x3^7 | 44: x1^3*x2^3*x3^5*x4^2
45: x1^3*x2^3*x3^5*x4*x5 | 46: x1^3*x2^3*x3^5*x5^2
47: x1^3*x2^3*x3^4*x4^3 | 48: x1^3*x2^3*x3^4*x4*x5^2
49: x1^3*x2^3*x3^4*x5^3 | 50: x1^3*x2^3*x3^3*x4^4
51: x1^3*x2^3*x3^3*x4^3*x5 | 52: x1^3*x2^4*x3*x4^2*x5^3
53: x1^3*x2^3*x3^3*x4*x5^3 | 54: x1^3*x2^3*x3^3*x5^4
55: x1^3*x2^4*x3*x4^3*x5^2 | 56: x1^3*x2^4*x3^3*x4*x5^2
57: x1^3*x2^3*x3*x4^6 | 58: x1^3*x2^3*x3*x4^5*x5
59: x1^3*x2^3*x3*x4^4*x5^2 | 60: x1^3*x2^3*x3*x4^3*x5^3
61: x1^3*x2^3*x3*x4^2*x5^4 | 62: x1^3*x2^3*x3*x4*x5^5
63: x1^3*x2^3*x3*x5^6 | 64: x1^3*x2^3*x4^7
65: x1^3*x2^3*x4^5*x5^2 | 66: x1^3*x2^3*x4^4*x5^3
67: x1^3*x2^3*x4^3*x5^4 | 68: x1^3*x2^3*x4*x5^6
69: x1^3*x2^3*x5^7 | 70: x1^3*x2*x3^7*x4^2
71: x1^3*x2*x3^7*x4*x5 | 72: x1^3*x2*x3^7*x5^2
73: x1^3*x2*x3^6*x4^3 | 74: x1^3*x2*x3^6*x4*x5^2
75: x1^3*x2*x3^6*x5^3 | 76: x1^3*x2*x3^5*x4^3*x5
77: x1^3*x2*x3^5*x4^2*x5^2 | 78: x1^3*x2*x3^5*x4*x5^3
79: x1^3*x2*x3^4*x4^3*x5^2 | 80: x1^3*x2*x3^4*x4^2*x5^3
81: x1^3*x2*x3^3*x4^6 | 82: x1^3*x2*x3^3*x4^5*x5
83: x1^3*x2*x3^3*x4^4*x5^2 | 84: x1^3*x2*x3^3*x4^3*x5^3
85: x1^3*x2*x3^3*x4^2*x5^4 | 86: x1^3*x2*x3^3*x4*x5^5
87: x1^3*x2*x3^3*x5^6 | 88: x1^3*x2*x3^2*x4^7
89: x1^3*x2*x3^2*x4^5*x5^2 | 90: x1^3*x2*x3^2*x4^4*x5^3
91: x1^3*x2*x3^2*x4^3*x5^4 | 92: x1^3*x2*x3^2*x4*x5^6
93: x1^3*x2*x3^2*x5^7 | 94: x1^3*x2*x3*x4^7*x5
95: x1^3*x2*x3*x4^6*x5^2 | 96: x1^3*x2*x3*x4^5*x5^3
97: x1^3*x2*x3*x4^3*x5^5 | 98: x1^3*x2*x3*x4^2*x5^6
99: x1^3*x2*x3*x4*x5^7 | 100: x1^3*x2*x4^7*x5^2
101: x1^3*x2*x4^6*x5^3 | 102: x1^3*x2*x4^3*x5^6
103: x1^3*x2*x4^2*x5^7 | 104: x1^3*x3^7*x4^3
105: x1^3*x3^7*x4*x5^2 | 106: x1^3*x3^7*x5^3
107: x1^3*x3^5*x4^3*x5^2 | 108: x1^3*x3^5*x4^2*x5^3
109: x1^3*x3^3*x4^7 | 110: x1^3*x3^3*x4^5*x5^2
111: x1^3*x3^3*x4^4*x5^3 | 112: x1^3*x3^3*x4^3*x5^4
113: x1^3*x3^3*x4*x5^6 | 114: x1^3*x3^3*x5^7
115: x1^3*x3*x4^7*x5^2 | 116: x1^3*x3*x4^6*x5^3
117: x1^3*x3*x4^3*x5^6 | 118: x1^3*x3*x4^2*x5^7
119: x1^3*x4^7*x5^3 | 120: x1^3*x4^3*x5^7
121: x1*x2^7*x3^3*x4^2 | 122: x1*x2^7*x3^3*x4*x5
123: x1*x2^7*x3^3*x5^2 | 124: x1*x2^7*x3^2*x4^3
125: x1*x2^7*x3^2*x4*x5^2 | 126: x1*x2^7*x3^2*x5^3
127: x1*x2^7*x3*x4^3*x5 | 128: x1*x2^7*x3*x4^2*x5^2
129: x1*x2^7*x3*x4*x5^3 | 130: x1*x2^7*x4^3*x5^2
131: x1*x2^7*x4^2*x5^3 | 132: x1*x2^6*x3^3*x4^3
133: x1*x2^6*x3^3*x4*x5^2 | 134: x1*x2^6*x3^3*x5^3
135: x1*x2^6*x3*x4^3*x5^2 | 136: x1*x2^6*x3*x4^2*x5^3
137: x1*x2^6*x4^3*x5^3 | 138: x1*x2^3*x3^7*x4^2
139: x1*x2^3*x3^7*x4*x5 | 140: x1*x2^3*x3^7*x5^2
141: x1*x2^3*x3^6*x4^3 | 142: x1*x2^3*x3^6*x4*x5^2
143: x1*x2^3*x3^6*x5^3 | 144: x1*x2^3*x3^5*x4^3*x5
145: x1*x2^3*x3^5*x4^2*x5^2 | 146: x1*x2^3*x3^5*x4*x5^3
147: x1*x2^3*x3^4*x4^3*x5^2 | 148: x1*x2^3*x3^4*x4^2*x5^3
149: x1*x2^3*x3^3*x4^6 | 150: x1*x2^3*x3^3*x4^5*x5
151: x1*x2^3*x3^3*x4^4*x5^2 | 152: x1*x2^3*x3^3*x4^3*x5^3
153: x1*x2^3*x3^3*x4^2*x5^4 | 154: x1*x2^3*x3^3*x4*x5^5
155: x1*x2^3*x3^3*x5^6 | 156: x1*x2^3*x3^2*x4^7
157: x1*x2^3*x3^2*x4^5*x5^2 | 158: x1*x2^3*x3^2*x4^4*x5^3
159: x1*x2^3*x3^2*x4^3*x5^4 | 160: x1*x2^3*x3^2*x4*x5^6
161: x1*x2^3*x3^2*x5^7 | 162: x1*x2^3*x3*x4^7*x5
163: x1*x2^3*x3*x4^6*x5^2 | 164: x1*x2^3*x3*x4^5*x5^3
165: x1*x2^3*x3*x4^3*x5^5 | 166: x1*x2^3*x3*x4^2*x5^6
167: x1*x2^3*x3*x4*x5^7 | 168: x1*x2^3*x4^7*x5^2
169: x1*x2^3*x4^6*x5^3 | 170: x1*x2^3*x4^3*x5^6
171: x1*x2^3*x4^2*x5^7 | 172: x1*x2^2*x3^7*x4^3
173: x1*x2^2*x3^7*x4*x5^2 | 174: x1*x2^2*x3^7*x5^3
175: x1*x2^2*x3^5*x4^3*x5^2 | 176: x1*x2^2*x3^5*x4^2*x5^3
177: x1*x2^2*x3^3*x4^7 | 178: x1*x2^2*x3^3*x4^5*x5^2
179: x1*x2^2*x3^3*x4^4*x5^3 | 180: x1*x2^2*x3^3*x4^3*x5^4
181: x1*x2^2*x3^3*x4*x5^6 | 182: x1*x2^2*x3^3*x5^7
183: x1*x2^2*x3*x4^7*x5^2 | 184: x1*x2^2*x3*x4^6*x5^3
185: x1*x2^2*x3*x4^3*x5^6 | 186: x1*x2^2*x3*x4^2*x5^7
187: x1*x2^2*x4^7*x5^3 | 188: x1*x2^2*x4^3*x5^7
189: x1*x2*x3^7*x4^3*x5 | 190: x1*x2*x3^7*x4^2*x5^2
191: x1*x2*x3^7*x4*x5^3 | 192: x1*x2*x3^6*x4^3*x5^2
193: x1*x2*x3^6*x4^2*x5^3 | 194: x1*x2^2*x3^4*x4^3*x5^3
195: x1*x2*x3^3*x4^7*x5 | 196: x1*x2*x3^3*x4^6*x5^2
197: x1*x2*x3^3*x4^5*x5^3 | 198: x1*x2*x3^3*x4^3*x5^5
199: x1*x2*x3^3*x4^2*x5^6 | 200: x1*x2*x3^3*x4*x5^7
201: x1*x2*x3^2*x4^7*x5^2 | 202: x1*x2*x3^2*x4^6*x5^3
203: x1*x2*x3^2*x4^3*x5^6 | 204: x1*x2*x3^2*x4^2*x5^7
205: x1*x2*x3*x4^7*x5^3 | 206: x1*x2*x3*x4^3*x5^7
207: x1*x3^7*x4^3*x5^2 | 208: x1*x3^7*x4^2*x5^3
209: x1*x3^6*x4^3*x5^3 | 210: x1*x3^3*x4^7*x5^2
211: x1*x3^3*x4^6*x5^3 | 212: x1*x3^3*x4^3*x5^6
213: x1*x3^3*x4^2*x5^7 | 214: x1*x3^2*x4^7*x5^3
215: x1*x3^2*x4^3*x5^7 | 216: x2^7*x3^3*x4^3
217: x2^7*x3^3*x4*x5^2 | 218: x2^7*x3^3*x5^3
219: x2^7*x3*x4^3*x5^2 | 220: x2^7*x3*x4^2*x5^3
221: x2^7*x4^3*x5^3 | 222: x2^3*x3^7*x4^3
223: x2^3*x3^7*x4*x5^2 | 224: x2^3*x3^7*x5^3
225: x2^3*x3^5*x4^3*x5^2 | 226: x2^3*x3^5*x4^2*x5^3
227: x2^3*x3^3*x4^7 | 228: x2^3*x3^3*x4^5*x5^2
229: x2^3*x3^3*x4^4*x5^3 | 230: x2^3*x3^3*x4^3*x5^4
231: x2^3*x3^3*x4*x5^6 | 232: x2^3*x3^3*x5^7
233: x2^3*x3*x4^7*x5^2 | 234: x2^3*x3*x4^6*x5^3
235: x2^3*x3*x4^3*x5^6 | 236: x2^3*x3*x4^2*x5^7
237: x2^3*x4^7*x5^3 | 238: x2^3*x4^3*x5^7
239: x2*x3^7*x4^3*x5^2 | 240: x2*x3^7*x4^2*x5^3
241: x2*x3^6*x4^3*x5^3 | 242: x2*x3^3*x4^7*x5^2
243: x2*x3^3*x4^6*x5^3 | 244: x2*x3^3*x4^3*x5^6
245: x2*x3^3*x4^2*x5^7 | 246: x2*x3^2*x4^7*x5^3
247: x2*x3^2*x4^3*x5^7 | 248: x3^7*x4^3*x5^3
249: x3^3*x4^7*x5^3 | 250: x3^3*x4^3*x5^7






\end{lstlisting}

\medskip

The data verifies the precision of the previous hand calculations as described in our published article \cite[Theorem 1.1]{D.P2}.

\medskip

We underscore that this advanced \texttt{SAGEMATH} algorithm builds upon our earlier efforts \cite{D.P11} in two significant ways: the enhanced algorithm retains the original precision in computing the dimensionality of $(\mathscr A\mathcal Q^{\otimes}_5)_N$ and extends its functionality by enabling the production of an explicit basis for $(\mathscr A\mathcal Q^{\otimes}_5)_N.$

The robustness of our method is demonstrated through its successful application in verifying dimensions of $\mathscr A\mathcal Q^{\otimes}_s$ for $s = 4$ at degrees $N \geq 100,$ producing results consistent with Sum's work in \cite{N.S1}. Moreover, cross-validation for $s = 5$ also produces accurate outcomes that match previously published results in \cite{P.S1, D.P2, D.P4, D.P7, D.P7-1, D.P8, D.P11, D.P11-1, D.P12, D.P12-1}, and \cite{N.S6, N.S2}.These data indicate that the \texttt{SAGEMATH} algorithms we developed are fully accurate and bolster the results we have  presented above.
 
In conclusion, we have great confidence in the accuracy of our new \texttt{SAGEMATH} algorithms, which do not encounter the overflow issues of the original \texttt{MAGMA} algorithm version \cite[Appendix A.9]{D.P7}. Furthermore, if we want to add something beyond $S32,$ it should also be easy to handle. We therefore hope this approach allows us to handle $(\mathscr A\mathcal Q^{\otimes}_s)_N$ for a greater number of variables $s$ and for any degree $N.$ Of course, with increasing degree $N \geq 49$, we will have to set MAX $\geq 2^s,$ where $s \geq 4.$ \textit{This situation demands that we refine the algorithm to alleviate the computer's memory strain during program execution}.
\end{rem}

\section{Behavior of the fifth cohomological transfer in the internal degrees $\pmb{N_d =5(2^{d}-1) + 11.2^{d+1}}$}\label{s4}

In this section, building on the results of our earlier work \cite{D.P4} and Theorem \ref{dlc1}(iii), we will investigate the behavior of the fifth Singer cohomological transfer \cite{W.S1} in the internal degree $N_d=5(2^{d}-1) + 11.2^{d+1}.$ 

We knew that in homotopy theory, the investigation of transfer maps linked to the trivial homomorphism $1\longrightarrow \mathbb F_2^{\oplus s}$ holds significant importance. Equally crucial is its algebraic variant in the level of Adams spectral sequences, referred to as the Singer cohomological transfer:
 $$ \varphi_s^{\mathscr A}: [(\mathscr A\mathcal Q^{\otimes}_s)_{N}^{GL(s, \mathbb F_2)}]^{*} \longrightarrow{\rm Ext}_{\mathscr A}^{s, s+N}(\mathbb F_2, \mathbb F_2).$$ 
Here $[(\mathscr A\mathcal Q^{\otimes}_s)_{N}^{GL(s, \mathbb F_2)}]^{*}$ denotes the dual of the $GL(s, \mathbb F_2)$-invariant  $(\mathscr A\mathcal Q^{\otimes}_s)_{N}^{GL(s, \mathbb F_2)}.$ For further insights into this transfer, readers are advised to consult works such as those by Ch\ohorn n-H\`a \cite{Chon-Ha0, Chon-Ha, Chon-Ha2}, H\uhorn ng \cite{Hung}, H\uhorn ng-Tr\uhorn \`\ohorn ng \cite{Hung2}, and the present author \cite{D.P7-1, D.P7-2, D.P7-3, D.P8, D.P11, D.P12, D.P13}.

{\bf Case $\pmb{d = 0}.$} We have derived the following technical theorem, which serves as the second primary result of our work.

\begin{thm}\label{dlc2}
The Singer transfer $\varphi_5^{\mathscr A}: [(\mathscr A\mathcal Q^{\otimes}_5)_{N_0}^{GL(5, \mathbb F_2)}]^{*} \longrightarrow{\rm Ext}_{\mathscr A}^{5, 5+N_0}(\mathbb F_2, \mathbb F_2)$ is an isomorphism.
\end{thm}

\begin{proof}
In \cite{D.P4}, we proved that $(\mathscr A\mathcal Q^{\otimes}_5)_{N_0} \cong \bigoplus_{1\leq j\leq 5} (\mathcal Q\mathcal {P}_5)({\widetilde{\omega}_j}),$ where $$ \widetilde{\omega}_1:= (2,2,2,1), \ \ \widetilde{\omega}_2:= (2,4,1,1), \ \ \widetilde{\omega}_3:= (2,4,3), \ \ \widetilde{\omega}_4:= (4,3,1,1),\ \ \widetilde{\omega}_5:= (4,3,3)$$
and $$ \begin{array}{ll} 
\medskip
& \dim (\mathcal Q\mathcal {P}_5)({\widetilde{\omega}_1})= 280,\ \ \dim (\mathcal Q\mathcal {P}_5)({\widetilde{\omega}_2})= 25,\ \ \dim (\mathcal Q\mathcal {P}_5)({\widetilde{\omega}_3})= 5,\\
&\dim (\mathcal Q\mathcal {P}_5)({\widetilde{\omega}_4})= 480,\ \ \dim (\mathcal Q\mathcal {P}_5)({\widetilde{\omega}_5})= 175.
\end{array}$$
Then, one has an estimate $$\dim [(\mathscr A\mathcal Q^{\otimes}_5)_{N_0}]^{GL(5,\mathbb F_2)}\leq \sum_{1\leq j\leq 5}\dim [(\mathcal Q\mathcal {P}_5)({\widetilde{\omega}_j})]^{GL(5,\mathbb F_2)}.$$
Through a straightforward computation employing the $\mathscr A$-homomorphisms $\rho_j: \mathcal {P}_5\longrightarrow \mathcal {P}_5,$  for $1\leq j\leq 5,$ we ascertain that $[(\mathcal Q\mathcal {P}_5)({\widetilde{\omega}_j})]^{GL(5,\mathbb F_2)} = 0$ for all $j.$ To simplify matters, we will meticulously compute the invariant space $[(\mathcal Q\mathcal {P}_5)({\widetilde{\omega}_3})]^{GL(5,\mathbb F_2)}.$ Similar computations can yield the remaining spaces, following the same approach.

As demonstrated in \cite{D.P4}, the set $\{[{\rm adm}_j]_{\widetilde{\omega}_3}:\ 1\leq j\leq 5\}$ serves as a basis of the indecomposables $(\mathcal Q\mathcal {P}_5)({\widetilde{\omega}_3}),$ where the admissible monomials ${\rm adm}_j$ are defined as follows:
$$ \begin{array}{ll}
\medskip
&{\rm adm}_1:= x_1x_2^{3}x_3^{6}x_4^{6}x_5^{6}, \ \ {\rm adm}_2:= x_1^{3}x_2x_3^{6}x_4^{6}x_5^{6}, \ \ {\rm adm}_2:= x_1^{3}x_2^{5}x_3^{2}x_4^{6}x_5^{6},\\
&{\rm adm}_4:= x_1^{3}x_2^{5}x_3^{6}x_4^{2}x_5^{6}, \ \ {\rm adm}_5:= x_1^{3}x_2^{5}x_3^{6}x_4^{6}x_5^{2}.
\end{array}$$
We first compute the action of the group $\Sigma_5$ on $(\mathcal Q\mathcal {P}_5)({\widetilde{\omega}_3}).$ Assume that $[u]_{\widetilde{\omega}_3}\in [(\mathcal Q\mathcal {P}_5)({\widetilde{\omega}_3})]^{\Sigma_5},$ then one has that $$ u\sim_{\widetilde{\omega}_3} \sum_{1\leq j\leq 5}\gamma_j{\rm adm}_j,$$
wherein $\gamma_j\in \mathbb F_2$ for every $j.$ For each $1\leq i\leq 4,$ when we apply the homomorphisms $\rho_i: \mathcal {P}_5\longrightarrow \mathcal {P}_5$ to the sum $\sum_{1\leq j\leq 5}\gamma_j{\rm adm}_j,$ we obtain:
$$ \begin{array}{ll}
\medskip
\rho_1(\sum_{1\leq j\leq 5}\gamma_j{\rm adm}_j) &= \gamma_1{\rm adm}_2 + \gamma_2{\rm adm}_1 +\gamma_3x_1^{5}x_2^{3}x_3^{2}x_4^{6}x_5^{6} + \gamma_4x_1^{5}x_2^{3}x_3^{6}x_4^{2}x_5^{6}  + \gamma_5x_1^{5}x_2^{3}x_3^{6}x_4^{6}x_5^{2},\\
\rho_2(\sum_{1\leq j\leq 5}\gamma_j{\rm adm}_j) &=\gamma_1x_1x_2^{6}x_3^{3}x_4^{6}x_5^{6} + \gamma_2x_1^{3}x_2^{6}x_3x_4^{6}x_5^{6} +\gamma_3x_1^{3}x_2^{2}x_3^{5}x_4^{6}x_5^{6}\\
\medskip
&\quad + \gamma_4x_1^{3}x_2^{6}x_3^{5}x_4^{2}x_5^{6}  + \gamma_5x_1^{3}x_2^{6}x_3^{5}x_4^{6}x_5^{2},\\ 
\medskip
\rho_3(\sum_{1\leq j\leq 5}\gamma_j{\rm adm}_j) &= \gamma_1{\rm adm}_1 + \gamma_2{\rm adm}_2 + \gamma_3{\rm adm}_4 +  \gamma_4{\rm adm}_3 +  \gamma_5{\rm adm}_5,\\
\medskip
\rho_4(\sum_{1\leq j\leq 5}\gamma_j{\rm adm}_j) &= \gamma_1{\rm adm}_1 + \gamma_2{\rm adm}_2 + \gamma_3{\rm adm}_3 +  \gamma_4{\rm adm}_5 +  \gamma_5{\rm adm}_4.
\end{array}$$
Applying the Cartan formula, we get
$$ \begin{array}{ll}
\medskip
x_1^{5}x_2^{3}x_3^{2}x_4^{6}x_5^{6}&\sim_{\widetilde{\omega}_3} {\rm adm}_3,\ \ x_1^{5}x_2^{3}x_3^{6}x_4^{2}x_5^{6}\sim_{\widetilde{\omega}_3} {\rm adm}_4,\ \ x_1^{5}x_2^{3}x_3^{6}x_4^{6}x_5^{2}\sim_{\widetilde{\omega}_3} {\rm adm}_5,\\
\medskip
x_1x_2^{6}x_3^{3}x_4^{6}x_5^{6}&\sim_{\widetilde{\omega}_3} {\rm adm}_1,\ \ x_1^{3}x_2^{6}x_3x_4^{6}x_5^{6}\sim_{\widetilde{\omega}_3} {\rm adm}_3,\ \ x_1^{3}x_2^{2}x_3^{5}x_4^{6}x_5^{6}\sim_{\widetilde{\omega}_3} {\rm adm}_2,\\
\medskip
x_1^{3}x_2^{6}x_3^{5}x_4^{2}x_5^{6}&\sim_{\widetilde{\omega}_3} {\rm adm}_4,\ \ x_1^{3}x_2^{6}x_3^{5}x_4^{6}x_5^{2}\sim_{\widetilde{\omega}_3} {\rm adm}_5.
\end{array}$$
Using the above calculations and the relations $\rho_i(u) +u \sim_{\widetilde{\omega}_3} 0,$ for $1\leq i\leq 4,$ we get
$$ \begin{array}{ll}
\medskip
\rho_1(u) \sim_{\widetilde{\omega}_3}  (\gamma_1 + \gamma_2)({\rm adm}_1 + {\rm adm}_2) \sim_{\widetilde{\omega}_3} 0,\\
\medskip
\rho_2(u) \sim_{\widetilde{\omega}_3}  (\gamma_2 + \gamma_3)({\rm adm}_2 + {\rm adm}_3) \sim_{\widetilde{\omega}_3} 0,\\
\medskip
\rho_3(u) \sim_{\widetilde{\omega}_3}  (\gamma_3 + \gamma_4)({\rm adm}_3 + {\rm adm}_4) \sim_{\widetilde{\omega}_3} 0,\\
\medskip
\rho_4(u) \sim_{\widetilde{\omega}_3}  (\gamma_4 + \gamma_5)({\rm adm}_4 + {\rm adm}_5) \sim_{\widetilde{\omega}_3} 0,
\end{array}$$
which imply that $\gamma_1 = \gamma_2 = \gamma_3 = \gamma_4 = \gamma_5.$ So, $[(\mathcal Q\mathcal {P}_5)({\widetilde{\omega}_3})]^{\Sigma_5} = \langle [\sum_{1\leq j\leq 5}{\rm adm}_j]_{\widetilde{\omega}_3} \rangle.$ Now, suppose that $[h]_{\widetilde{\omega}_3}\in [(\mathcal Q\mathcal {P}_5)({\widetilde{\omega}_3})]^{GL(5,\mathbb F_2)}.$ Since $\Sigma_5\subset GL(5,\mathbb F_2),$ one has $$h\sim_{\widetilde{\omega}_3}\beta({\rm adm}_1 + {\rm adm}_2+{\rm adm}_3+{\rm adm}_4+{\rm adm}_5),$$ in which $\beta\in \mathbb F_2.$ Applying the $\mathscr A$-homomorphism $\rho_5: \mathcal {P}_5\longrightarrow \mathcal {P}_5$ to the sum $\beta\sum_{1\leq j\leq 5}{\rm adm}_j,$ we obtain
$$ \rho_5(\beta\sum_{1\leq j\leq 5}{\rm adm}_j) = \beta(\sum_{2\leq j\leq 5}{\rm adm}_j + x_1x_2^{7}x_3^{2}x_4^{6}x_5^{6} + x_1x_2^{7}x_3^{6}x_4^{2}x_5^{6} + x_1x_2^{7}x_3^{6}x_4^{6}x_5^{2}).$$
Using the Cartan formula, we have 
$$ \begin{array}{ll}
\medskip
 x_1x_2^{7}x_3^{2}x_4^{6}x_5^{6} &= Sq^{1}(x_1^{2}x_2^{7}x_3x_4^{5}x_5^{6}) + Sq^{2}(x_1x_2^{7}x_3x_4^{5}x_5^{6} + x_1x_2^{7}x_3x_4^{3}x_5^{8} + x_1x_2^{9}x_3x_4^{3}x_5^{6}) \\
&\quad\quad+ \sum_{Y < x_1x_2^{7}x_3^{2}x_4^{6}x_5^{6}}Y.
\end{array}$$
Hence $x_1x_2^{7}x_3^{2}x_4^{6}x_5^{6} \sim_{\widetilde{\omega}_3}0.$ Similarly, $x_1x_2^{7}x_3^{6}x_4^{2}x_5^{6} \sim_{\widetilde{\omega}_3} 0$ and $x_1x_2^{7}x_3^{6}x_4^{6}x_5^{2}\sim_{\widetilde{\omega}_3} 0.$ These data imply that $$ \rho_5(h)\sim_{\widetilde{\omega}_3} \beta({\rm adm}_2 + {\rm adm}_3 +{\rm adm}_4 +{\rm adm}_5).$$
Therefore, by the relation $\rho_5(h)\sim_{\widetilde{\omega}_3} h,$ we have $\beta{\rm adm}_1\sim_{\widetilde{\omega}_3} 0,$ which implies $\beta = 0.$ This shows that the invariant $[(\mathcal Q\mathcal {P}_5)({\widetilde{\omega}_3})]^{GL(5,\mathbb F_2)}$ vanishes.

Now, the theorem  follows directly from the fact that ${\rm Ext}_{\mathscr A}^{5, 5+N_0}(\mathbb F_2, \mathbb F_2)  = 0$ (see \cite{W.L}, \cite{T.C}).
\end{proof}

{\bf Cases $\pmb{d \geq 1}.$} A straightforward computation, relying on the works by Lin \cite{W.L} and Chen \cite{T.C}, yields:
$${\rm Ext}_{\mathscr A}^{5, 5+N_d}(\mathbb F_2, \mathbb F_2) = \mathbb F_2\cdot h_{d+4}f_{d-1}\neq 0,\, \forall d\geq 1.
$$
It is well-known that every element in the $Sq^{0}$-families $\{h_d:\, d\geq 0\}$ and $\{f_d:\, d\geq 0\}$ is in the image of the Singer cohomological transfer (see \cite{W.S1}, \cite{T.N}). Since the "total" transfer 
$$ \varphi^{\mathscr A}_*(\mathbb F_2): \bigoplus_{s, N} [(\mathscr A\mathcal Q^{\otimes}_s)_{N}^{GL(s, \mathbb F_2)}]^{*} \longrightarrow\bigoplus_{s, N}{\rm Ext}_{\mathscr A}^{s, s+N}(\mathbb F_2, \mathbb F_2)$$
is a homomorphism of algebras, we have immediately the following.

\begin{hq}\label{hq1}
The fifth cohomological transfer
$$ \varphi_5^{\mathscr A}: [(\mathscr A\mathcal Q^{\otimes}_5)_{N_d}^{GL(5, \mathbb F_2)}]^{*} \longrightarrow{\rm Ext}_{\mathscr A}^{5, 5+N_d}(\mathbb F_2, \mathbb F_2)$$
is an epimorphism, for all $d\geq 1.$
\end{hq}

As a consequence of Theorem \ref{dlc1}, we obtain:
$$ [{\rm Ker}((\widetilde {Sq^0_*})_{(5; N_1)})]^{GL(5, \mathbb F_2)}  \subseteq  [(\mathcal Q\mathcal {P}_5)({\omega}_{(1)})]^{GL(5,\mathbb F_2)},\ \text{noting ${\omega}_{(1)} = (3,3,2,2,1).$}$$
On the other hand, for each positive integer $d,$ we have the following isomorphisms: $$ (\mathscr A\mathcal Q^{\otimes}_5)_{N_d}\cong (\mathscr A\mathcal Q^{\otimes}_5)_{N_1}= {\rm Ker}((\widetilde {Sq^0_*})_{(5; N_1)})\bigoplus (\mathscr A\mathcal Q^{\otimes}_5)_{N_0},$$ and $$[(\mathscr A\mathcal Q^{\otimes}_5)_{N_d}^{GL(5, \mathbb F_2)}]^{*}\cong (\mathscr A\mathcal Q^{\otimes}_5)_{N_d}^{GL(5, \mathbb F_2)}.$$ Consequently, thanks to Theorem \ref{dlc2}, we obtain 
$$\dim [(\mathscr A\mathcal Q^{\otimes}_5)_{N_d}^{GL(5, \mathbb F_2)}]^{*}\leq \dim [(\mathcal Q\mathcal {P}_5)({\omega}_{(1)})]^{GL(5,\mathbb F_2)},$$
for arbitrary $d\geq 1.$ Referring to the conclusion from item (iii) of Theorem \ref{dlc1}, An immediate calculation, akin to the $d = 0$ case, shows that the $GL(5, \mathbb F_2)$-invariant space $[(\mathcal Q\mathcal {P}_5)({\omega}_{(1)})]^{GL(5,\mathbb F_2)}$ is one-dimensional.  

\begin{cy} An algorithm developed within the \texttt{SAGEMATH} computer algebra system was employed to verify the dimension result for the invariant $[(\mathcal Q\mathcal {P}_5)({\omega}_{(1)})]^{GL(5,\mathbb F_2)}.$ In fact, it verified the dimension result for the invariant  $(\mathscr A\mathcal Q^{\otimes}_5)_{N_1}^{GL(5, \mathbb F_2)}.$ The full details of the algorithm are omitted here, as \textbf{they constitute a separate scientific work by us on algorithm implementation in computer algebra systems, which will be published in a specialized journal on algorithmic techniques}.

\medskip

The following outlines several fundamental steps of the algorithm. Nevertheless, this is not the entire algorithm for determining the dimension of the invariant space.

\medskip

\begin{lstlisting}[basicstyle={\fontsize{8.5pt}{9.5pt}\ttfamily},]

# Generate all GL(5, F_2) matrices
def generate_GL5_matrices():
    F = GF(2)
    G = GL(5, F)
    return [M for M in G]

# Action of a permutation on a monomial
def permutation_action(perm, monomial):
    return tuple(monomial[perm[i] - 1] for i in range(5))

# Action of a GL(5, F_2) matrix on a monomial
def GL5_action(matrix, monomial):
    return tuple(sum(matrix[i, j] * monomial[j] for j in range(5)) % 2 for i in range(5))

# Generate a basis for a given weight vector
def generate_basis(weight_vector):
    total_degree = sum(weight_vector)
    basis = []
    for partition in Partitions(total_degree, length=5):
        if all(sum(partition[i] & (1 << j) for i in range(5)) <= weight_vector[j] for j in range(len(weight_vector))):
            basis.append(tuple(partition))
    return basis

def find_invariants(basis, GL5_matrices):
    F = GF(2)
    V = VectorSpace(F, len(basis))
    invariant_space = V.subspace([V.zero()])
    
    for v in V:
        if all(v == V(GL5_action(M, sum(v[i] * b for i, b in enumerate(basis)))) for M in GL5_matrices):
            invariant_space = invariant_space + V.span([v])
    
    return invariant_space

\end{lstlisting}
\end{cy}

When we combine this result with Corollary \ref{hq1}, we achieve

\begin{thm}
The cohomological transfer
 $$\varphi_5^{\mathscr A}: [(\mathscr A\mathcal Q^{\otimes}_5)_{N_d}^{GL(5, \mathbb F_2)}]^{*} \longrightarrow{\rm Ext}_{\mathscr A}^{5, 5+N_d}(\mathbb F_2, \mathbb F_2)$$
is an isomorphism, for every positive integer $d.$
\end{thm}








\end{document}